%% file: sparse.tex
\documentclass{article}
\usepackage{amsmath,amsthm,amsfonts}
\usepackage{graphics,color,epsfig}
\newtheorem{theorem}{Theorem}[section]
\newtheorem{lemma}[theorem]{Lemma}

\newtheorem{corollary}[theorem]{Corollary}
\newtheorem{conjecture}[theorem]{Conjecture}
\newtheorem{question}[theorem]{Question}
\theoremstyle{definition}
\newtheorem{remark}[theorem]{Remark}
\newtheorem{example}[theorem]{Example}
\newtheorem{assumption}[theorem]{Assumption}
\newtheorem*{ack}{Acknowledgement}

\newcommand\noproof{\hfill$\Box$}

\newcommand\webcite[1]{\hfil\penalty0\texttt{\def~{\~{}}#1}\hfill\hfill}
\newcommand\arxiv[1]{\webcite{arXiv:#1.}}

\newcommand\Var{\operatorname{Var}}

\newcommand\ceil[1]{\lceil #1\rceil}
\newcommand\floor[1]{\lfloor #1\rfloor}
\newcommand\norm[1]{||#1||}
\newcommand\bb[1]{\bigl(#1\bigr)}
\newcommand\bm[1]{\bigl|#1\bigr|}
\newcommand\E{{\mathbb E{}}}
\newcommand\de{d_{\mathrm{sub}}}

\newcommand\eps{\varepsilon}

\newcommand\ind{\mathrm{ind}}
\newcommand\emb{\mathrm{emb}}
\newcommand\aut{\mathrm{aut}}

\newcommand\ka{\kappa}
\newcommand\la{\lambda}
\newcommand\La{\Lambda}

\newcommand\A{\mathcal{A}}
\newcommand\C{\mathcal{C}}
\newcommand\D{\mathcal{D}}
\newcommand\F{\mathcal{F}}
\newcommand\M{\mathcal{M}}
\newcommand\Mk{B_k}
\newcommand\Fm{\mathcal{F^{\mathrm m}}}

\newcommand\K{\mathcal{K}}
\newcommand\J{\mathcal{J}}
\newcommand\T{\mathcal{T}}

\newcommand\LL{\mathcal{L}}
\newcommand\dc{d_\mathrm{cut}}

\newcommand\dH{d_\mathrm{H}}

\newcommand\dcG{{\widehat d}_\mathrm{cut}}
\newcommand\dP{d_\mathrm{part}}
\newcommand\cn[1]{||#1||_{\mathrm{cut}}}
\newcommand\on[1]{||#1||_1}
\newcommand\tn[1]{||#1||_2}
\newcommand\cc{{\mathrm c}}

\newcommand\dd{\,d}
\newcommand\PP{{\mathcal P}}
\newcommand\RR{{\mathbb R}}
\newcommand\NN{{\mathbb N}}
\newcommand\Z{{\mathbb Z}}
\newcommand\bx{{\bf x}}
\newcommand\iid{iid}
\newcommand\PR{\mathbb{P}}
\newcommand\tr{\mathrm{tr}}
\newcommand\pto{\overset{\mathrm{p}}{\to}}
\newcommand\IN[1]{1_{#1}}
\newcommand\db{{\bar d}}
\newcommand\BB{\mathcal{B}}

\newcommand\diff{{\Delta}}

\newcommand\infn[1]{\|#1 \|_\infty}

\begin{document}

\title{Metrics for sparse graphs}
\date{December 14, 2008}

\author{B\'ela Bollob\'as%
\thanks{Department of Pure Mathematics and Mathematical Statistics,
Wilberforce Road, Cambridge CB3 0WB, UK and
Department of Mathematical Sciences, University of Memphis, Memphis TN 38152, USA.
E-mail: {\tt b.bollobas@dpmms.cam.ac.uk}.}
\thanks{Research supported in part by NSF grants CCR-0225610 and DMS-0505550
and ARO grant W911NF-06-1-0076}
\and Oliver Riordan%
\thanks{Mathematical Institute, University of Oxford, 24--29 St Giles', Oxford OX1 3LB, UK.
E-mail: {\tt riordan@maths.ox.ac.uk}.}
}
\maketitle

\begin{abstract}
  Recently, Bollob\'as, Janson and Riordan introduced a very general
  family of random graph models, producing inhomogeneous random
  graphs with $\Theta(n)$ edges. Roughly speaking, there is one model
  for each {\em kernel}, i.e., each symmetric measurable function
  from $[0,1]^2$ to the non-negative reals, although the details
  are much more complicated, to ensure the exact inclusion of many of the
  recent models for large-scale real-world networks.

  A different connection between kernels and random graphs arises
  in the recent work of Borgs, Chayes, Lov\'asz, S\'os, Szegedy
  and Vesztergombi. They
  introduced several natural metrics on dense graphs (graphs with $n$
  vertices and $\Theta(n^2)$ edges), showed that these metrics are
  equivalent, and gave a description of the completion of the space of
  all graphs with respect to any of these metrics in terms of {\em
    graphons}, which are essentially bounded kernels.
  One of the most appealing aspects of
  this work is the message that sequences of inhomogeneous
  quasi-random graphs are in a sense completely general: any sequence
  of dense graphs contains such a subsequence. Alternatively, their
  results show that certain natural models of dense inhomogeneous
  random graphs (one for each graphon) cover the space of dense graphs:
  there is one model for each point of the completion, producing
  graphs that converge to this point.

  Our aim here is to briefly survey these results, and then to investigate to what extent
  they can be generalized to graphs with $o(n^2)$ edges. Although
  many of the definitions extend in a simple way, the connections between
  the various metrics, and between the metrics and random graph models, turn
  out to be much more complicated than in the dense case. We shall prove
  many partial results, and state even more conjectures and open problems,
  whose resolution would greatly enhance the currently rather unsatisfactory
  theory of metrics on sparse graphs.
  This paper deals mainly with graphs with $o(n^2)$ but $\omega(n)$ edges: a companion
  paper will discuss the (more problematic still)
  case of {\em extremely sparse} graphs, with $O(n)$ edges.
\end{abstract}

\tableofcontents

\section{Introduction}
In recent years, much work has been done constructing
and analyzing mathematical models of
real-world networks. The random graphs in these
models are inhomogeneous -- in fact, many of them have degree
sequences with power law distributions. 
In~\cite{BJR}, Bollob\'as, Janson and
Riordan defined a very general model of an $n$-vertex
random graph $G(n, \kappa)$ with conditional independence between
the edges which includes as special cases many of the models of
real-world networks that have been studied, and proved numerous
results about the random graphs generated by this model, including
results about their component structure and the point and nature of
the phase transition in them. Here  the {\em kernel} $\kappa$ is a
symmetric measurable function from $[0,1]^2$ to $[0,\infty)$
satisfying some mild conditions. (Some of these conditions arise due
to the very general nature of other parts of the model, and can be
weakened in other contexts; see~\cite{BJRclust} and~\cite{BJRcm} for a discussion of
this.) Just like the real-world graphs that motivated the
construction of the BJR model, the random graphs $G(n, \kappa)$ are
sparse in the sense that the expected number of edges is $O(n)$ (in
fact, $(c+o(1))n$ for some constant $c$). In~\cite{BJR} the kernel
$\kappa$ was used to define a multi-type branching process ${\mathcal
X}_{\kappa}$ whose survival probability is closely related to the
component structure of $G(n, \kappa)$.

In order to decide how well our random graph $G(n, \kappa)$
approximates a given real-world graph $G_n$, it would be desirable
to establish a {\em distance} between a random graph model and a
graph, so that the approximation is judged to be better and better
as the distance tends to $0$. Putting it slightly differently, we
should like to define a metric on the set of sparse finite graphs so
that a Cauchy sequence consists of graphs that are in some sense
`similar', and the limit of such a (not eventually constant)
sequence is naturally identified with a suitable random graph model.
For {\em dense} graphs, graphs with $n$ vertices and at least $cn^2$
edges, such a program has been carried out very successfully in a
series of papers by (various subsets of)
Borgs, Chayes, Lov\'asz, S\'os, Szegedy and Vesztergombi (see
\cite{BCLSV:homcount,BCLSSV:stoc,LSgenQT,LSz1,BCLSV:1,BCLSV:2} and
the references therein). In particular, they introduced several
metrics on the space of dense finite (weighted) graphs and showed
them to be equivalent. The limiting objects, i.e., the additional
points in the completion, turn out to be {\em graphons}, that is,
bounded symmetric measurable functions from $[0,1]^2$ to $\RR$. The
corresponding random graph models, called {\em $W$-random graphs} in~\cite{LSz1},
are the natural dense version of $G(n,\ka)$; see Subsection~\ref{ss_rg}.

The only difference between kernels and graphons is that the
latter are bounded, while the former must be allowed to be unbounded
in order to model, for example, highly inhomogeneous real-world networks.
In many fundamental questions (for example those concerning the phase
transition), this difference is substantial.
The appearance of graphons or kernels
in the two different contexts described above suggests the existence
of interesting connections between these areas. One such connection
is described by Bollob\'as, Borgs, Chayes and Riordan~\cite{BBCR},
who study (sparse) random subgraphs of arbitrary dense graphs;
this has recently been extended by Bollob\'as, Janson and Riordan~\cite{BJRcm}.

We have several aims in this paper. First, we shall
review some of the results of Borgs, Chayes, Lov\'asz,
S\'os, Szegedy and Vesztergombi mentioned above.
Our main aim is then to take the first tentative steps towards a
general theory of metrics on sparse graphs; in particular, we shall
investigate to what extent these ideas can be carried over to the sparse
setting, and what can be said about the connection between the
metrics and the ideas of Bollob\'as, Janson and Riordan. As we shall
see, the difficulties that arise are considerably greater than in
the dense case; in fact, the difficulties increase as the graphs get
sparser. The {\em almost dense} case $e(G_n)=n^{2-o(1)}$ is already
rather different from the dense case; the {\em extremely sparse} case
$e(G_n)=\Theta(n)$, which will be studied in a companion paper~\cite{BRsparse2},
is {\em very} different indeed, having
many novel features. We shall prove
numerous results, but the picture we obtain is much less complete
than that obtained by Borgs et al in the dense case. In fact,
perhaps our most important aim is to identify some of the main
problems and conjectures whose resolution would enhance the theory
of metrics on sparse graphs.

An important tool in the study of metrics on spaces of dense graphs
is Szemer\'edi's Regularity Lemma. While there is a version of
Szemer\'edi's Lemma for sparse graphs (with $o(n^2)$ but $\omega(n)$
edges) satisfying a mild additional condition, there is no
satisfactory counting/embedding lemma for counting (or even finding)
small subgraphs using regular partitions. This is one of the
reasons why sparse graphs are much more difficult to handle than
dense ones. One of our main aims is to prove such a counting lemma
for certain subgraphs, greatly extending a result of Chung and
Graham~\cite{CG}.

The rest of the paper is organized as follows. The next section
is about dense graphs and kernels; we start by
briefly recalling some of the definitions and results of
Borgs, Chayes, Lov\'asz, S\'os, Szegedy and Vesztergombi whose
generalization we shall discuss, focussing in particular
on the cut metric.
Then, in Subsection~\ref{ss_equiv}, we show that these
results are closely connected to the question of when
two kernels are `equivalent'; we shall need this
notion of equivalence when we come to sparse graphs.

The rest of the paper concerns sparse
graphs, i.e., graphs with $n$ vertices and $o(n^2)$ edges:
in Section~\ref{sec_hsp}
we consider subgraph counts in sparse (but mostly not too sparse)
graphs, stating a conjecture that generalizes the main result
of Lov\'asz and Szegedy~\cite{LSz1},
and proving various partial results, concentrating especially 
on the {\em uniform} case, i.e., on sparse quasi-random graphs.
In Section~\ref{sec_Sz} we turn to Szemer\'edi's Lemma for sparse
graphs satisfying an appropriate `bounded density' assumption,
and the consequences for questions of convergence in the cut metric.

Sections~\ref{sec_compar} is the longest and
most important section of the paper. In it we discuss
the relationship between the cut metric and the count metric (to be defined)
in the sparse case. As well as proposing various conjectures extending the 
results of Borgs, Chayes, Lov\'asz, S\'os and Vesztergombi, we prove
several partial results, amounting to `sparse counting lemmas'
with various assumptions; these results, Theorem~\ref{cthm}
and its variants Theorems~\ref{cthm2} and~\ref{dthm},
are the most substantial results in the paper.

In Section~\ref{sec_partition} we briefly discuss another metric
considered by Borgs, Chayes, Lov\'asz, S\'os and Vesztergombi,
showing that for graphs that are sparse, but not too sparse, it is
equivalent to the cut metric. In the {\em extremely sparse}
case, considering graphs with bounded average degree, the partition
metric turns out to be much more useful than the cut metric.
This and a discussion of the many problems and interesting
open questions concerning metrics on extremely sparse graphs
will be the topic of a companion paper~\cite{BRsparse2}.

In Section~\ref{sec_end} we return briefly to the relationship between
metrics and random graph models, and close with some final
remarks summarizing our main results and conjectures.

Throughout the paper we use standard graph theoretic notation
as in~\cite{BB:MGT}. For example, $|G|$ and $e(G)$ denote
respectively the number of vertices and number of edges
of a graph $G$.

\section{Dense graphs}\label{dense}

There are many natural definitions of what it means for two graphs
to be `close', and corresponding metrics and notions of
Cauchy/fundamental sequences. These tend to be particularly natural
for `dense' graphs, with $\Theta(n^2)$ edges. Several of these
metrics have been studied by Borgs, Chayes, Lov\'asz, S\'os and
Vesztergombi~\cite{BCLSV:1,BCLSV:2}, who showed that they are
equivalent, and that there is a natural completion of the space of
graphs under any of these metrics. In this section we briefly recall
some of these definitions and results; we are not aiming to
give a comprehensive survey of
the results of these papers, discussing only those that will be relevant for us
here. Although most of the results mentioned in
Subsections~\ref{ss_sd}--\ref{ss_rg}
will be from Lov\'asz and Szegedy~\cite{LSz1}
and~\cite{BCLSV:1,BCLSV:2}, we shall not always adopt their notation
or terminology, or indeed follow their definitions exactly. 

Borgs, Chayes, Lov\'asz, S\'os and
Vesztergombi~\cite{BCLSV:1,BCLSV:2} consider {\em weighted graphs},
with weights on the edges and on the vertices. For the results
we shall describe, this makes essentially no difference. In what
follows, we consider only unweighted graphs; while much of what we
shall say presumably carries over to suitably weighted graphs, the
definitions for weighted graphs are not as natural in the sparse
case, and are likely to introduce more additional complications than
new insights.

\subsection{The subgraph distance}\label{ss_sd}

The basic starting point is to consider, for each fixed graph $F$,
the number of copies of $F$ in a large graph $G$, i.e., the number
$X_F(G)$ of subgraphs of $G$ isomorphic to $F$.
Recall that a {\em homomorphism} from a graph $F$ to a graph $G$
is a function $\phi:V(F)\to V(G)$ such that $\phi(x)\phi(y)\in E(G)$
whenever $xy\in E(F)$. Although
$X_F(G)$ (for example, the number of triangles in $G$) is the most
natural basic notion in this context, it turns out to be cleaner
to work with $\emb(F,G)$, the number of injective homomorphisms
or {\em embeddings} of $F$ into $G$.
Note that
\[
 \emb(F,G) = \aut(F) X_F(G),
\]
so $X_F(G)$ and $\emb(F,G)$ contain the same information. Working
with the latter avoids constant factors $\aut(F)$ in many formulae.

If $F$ has $k$ vertices, then for $n\ge k$ we have
$\emb(F,K_n)=n_{(k)} = n(n-1)\cdots(n-k+1)$,
so the natural normalization is to work with
\[
 s(F,G) = \frac{\emb(F,G)}{n_{(k)}} = \frac{X_F(G)}{X_F(K_n)} \in [0,1],
\]
where, as usual, $n=|G|$ is the number of vertices of $G$.
If $|F|>|G|$ then the above ratio is not defined, and we set $s(F,G)=0$.

Let $\F$ denote the set of isomorphism classes of finite graphs;
sometimes it will be convenient to enumerate $\F$ in an arbitrary
way, writing $\F=\{F_1,F_2,\ldots\}$. (More formally, we shall
take each $F_i$ to be a representative of an isomorphism class.)
The graph parameters
$s(F,\cdot)$, $F\in \F$, define a natural family of equivalent metrics
on $\F$, by mapping $\F$ into $[0,1]^\infty$ (or into $[0,1]^\F$).
Indeed, for any finite graph $G$, set
\[
 s(G) =(s_i(G))_{i=1}^\infty \in [0,1]^\infty,
\]
where $s_i(G)=s(F_i,G)$.
Let $d$ be any metric on $X=[0,1]^\infty$ which gives the product topology,
for example $d(s,t)=\sum_{i=1}^\infty 2^{-i}|s_i-t_i|$.
We may define the {\em subgraph distance} of two
graphs $G_1$, $G_2$ as
\[
 \de(G_1,G_2) = d(s(G_1),s(G_2)).
\]
It is easy to see that this defines a metric on $\F$: indeed,
given $G\in \F$, among graphs $F$ with $s(F,G)>0$, there is a unique
graph with $|F|+e(F)$ maximal, namely $G$. Thus the map $G\mapsto s(G)$
is injective. Furthermore, considering $s(E_{n+1},G)$, where $E_{n+1}$
is the empty graph with $n+1$ vertices, we see that the distance
between any graph $G$ with $n$ vertices and the set of graphs
with more than $n$ vertices is positive. It follows that the
metric space $(\F,\de)$ is discrete.

A sequence $(G_n)$ of graphs is Cauchy with respect to $\de$
if and only if, for each $F\in \F$, the sequence $s(F,G)$
converges. Such sequences are sometimes called `convergent',
although they do not converge in the metric space $(\F,\de)$.
Note that if $(G_n)$ is Cauchy then, since $(\F,\de)$ is discrete, 
either $(G_n)$ is eventually constant, or $|G_n|\to \infty$.

Many minor variations on the definition of $\de$ are possible. For example,
instead of considering the number of embeddings of $F$ into $G$, one can
consider the number $\hom(F,G)$ of
homomorphisms from $F$ to $G$.
If $|F|=k$ and $|G|=n$, then the number of non-injective homomorphisms from $F$ to $G$
is at most $\binom{k}{2}n^{k-1}=O(n^{k-1})$, so setting
\[
 t(F,G)=\hom(F,G)/n^k
\]
we have
\begin{equation}\label{st}
 t(F,G) = s(F,G) +O(n^{-1})
\end{equation}
for each $F$.
Hence, in this dense case, the parameters $s(F,\cdot)$ and $t(F,\cdot)$
are essentially equivalent. [There is a minor difference that, working
with homomorphisms, one ends up with a pseudo-metric: if $G$ is any graph
and $G^{(r)}$ is the {\em blow-up} of $G$ obtained by making $r$ copies
of each vertex, joined to all copies of its neighbours, then
$t(F,G^{(r)})=t(F,G)$ for all $F\in\F$ and $r\ge 1$.]
Also, one can pass easily
back and forth between subgraph counts and counts of induced subgraphs
using inclusion--exclusion.

One of the key properties of the metric $\de$ is that there is a natural description
of the (clearly compact) completion of $(\F,\de)$, in terms of {\em standard kernels}
(also called {\em graphons}).
Here a {\em kernel} is a symmetric measurable function from $[0,1]^2$
to $[0,\infty)$; a standard kernel is one taking values in $[0,1]$.
In other contexts, one considers more general bounded kernels, taking values in
$[0,M]$ or $[-M,M]$, $M>0$, or general {\em signed kernels} taking
values in $\RR$.
One can extend the definition of $s(F,G)$ (or of $t(F,G)$) to kernels in a natural way:
given a finite graph $F$ with vertex set $\{1,2,\ldots,k\}$,
let
\begin{equation}\label{sfk}
 s(F,\ka) = \int_{[0,1]^k} \prod_{ij\in E(F)}\ka(x_i,x_j)\prod_{i=1}^k \dd x_i.
\end{equation}
(Some authors use the notation $t(F,\ka)$ for the same quantity.)
This formula has a natural interpretation as the normalized
`number' of embeddings
of $F$ into a weighted graph with the uncountable vertex set $[0,1]$,
with edge weights given by $\ka$.
Of course, in this context there is no difference between embeddings
and homomorphisms.

Lov\'asz and Szegedy~\cite{LSz1} proved (essentially) the following result.
\begin{theorem}\label{ctok}
Let $(G_n)$ be a Cauchy sequence in $(\F,\de)$. Then either $(G_n)$
is eventually constant, or there is a standard kernel $\ka$
such that $s(F,G_n)\to s(F,\ka)$.
\noproof
\end{theorem}
Let us remark that the result proved in~\cite{LSz1} concerns $t$ rather than $s$, which makes
no difference, except that a separate case for eventually
constant sequences is then not needed. Here, the distinction is informative:
considering the parameters $s(E_k,G_n)$
for each $k$ shows that in the second case above we have $|G_n|\to\infty$.

Of course, \eqref{sfk} allows one to extend
the metric $\de$ to standard kernels, obtaining in the first
instance a pseudo-metric on the set of standard kernels.
There is a natural notion of equivalence for kernels,
which one can think of as a two dimensional version of the equivalence
relation on random variables given by $X\sim Y$ if $X$ and $Y$
have the same distribution; the details are somewhat technical,
and not essential for understanding the metrics discussed here,
so we postpone them to Subsection~\ref{ss_equiv}. We write
$\sim$ for this relation, and $\K$ for the set
of equivalence classes of standard kernels under $\sim$. Borgs, Chayes and Lov\'asz~\cite{BCL:unique}
have shown that
$\ka_1\sim\ka_2$ if and only if $\de(\ka_1,\ka_2)=0$ (see also
Theorem~\ref{th_kequiv}), so $\de$ induces a metric
on $\K$.
The metric space $(\K,\de)$ is complete (the result about Cauchy sequences
of graphs above applies just as well to standard kernels). Hence,
the completion of $(\F,\de)$ is obtained by adding to $\F$ the set $\K$
of all equivalence classes of standard kernels, and using
the map $s:\F\cup \K\to [0,1]^\infty$ to extend $\de$ to $\F\cup\K$.

There is a natural way to associate a standard kernel $\ka_G$ to a graph
$G$ with $n$ vertices: divide $[0,1]$ into $n$ intervals
$I_1,\ldots,I_n$ of equal length (we may and shall ignore the
question of which endpoints are included), and set $\ka_G$ to be $1$
on $I_i\times I_j$ if $ij\in E(G)$, and $0$ otherwise. One slight
advantage of using $t$ rather than $s$ is that
\[
 t(F,G)=s(F,\ka_G)
\]
for all graphs $F$ and $G$. However, the metric
obtained using $t$ is only a pseudo-metric, since graphs on
different numbers of vertices may correspond to the same kernel, for
example if one is a blow-up of the other.

We say that a kernel $\ka$ is of {\em finite type} if there is a
partition of $[0,1]$ into measurable
sets $A_1,\ldots,A_k$ so that $\ka$ is constant on each of the
rectangles $A_i\times A_j$. Note that $\ka_G$ is always of finite
type.

\subsection{The cut distance}

Borgs, Chayes, Lov\'asz, S\'os and Vesztergombi~\cite{BCLSV:1}
considered another natural metric on graphs or kernels, namely, the
{\em cut metric}, based on a norm used by Frieze and Kannan~\cite{FKquick}.
For any integrable
function $\ka:[0,1]^2\to\RR$, its {\em cut norm} $\cn{\ka}$ is defined by
\begin{equation}\label{cndefST}
 \cn{\ka} = \sup_{S,T\subset [0,1]} \left|\int_{S\times T} \ka(x,y) \dd x\dd y\right|,
\end{equation}
where the supremum is over all pairs of measurable
subsets of $[0,1]$.
It is easily seen that this defines a norm on $L^\infty([0,1]^2)$.
In fact, there are several variations of this definition: one can take
\begin{equation}\label{cndefSSc}
 \cn{\ka} = \sup_{S\subset [0,1]} \left|\int_{S\times S^\cc} \ka(x,y) \dd x\dd y\right|,
\end{equation}
where $S^\cc=[0,1]\setminus S$,
or one can take the supremum in \eqref{cndefST}
only over sets $S$, $T$ with $S\cap T=\emptyset$.
It is easy to check that these variations only affect
the norm up to an (irrelevant) constant factor (see~\cite{BCLSV:1}), so we
shall feel free to use whichever definition is most convenient in any given context.

There is yet another definition of $\cn{\ka}$ that is more natural from the point
of view of functional analysis, namely
\[
 \cn{\ka} = \sup_{\infn{f},\infn{g}\le 1} \int_{[0,1]^2} \ka(x,y)f(x)g(y) \dd x \dd y,
\]
where the supremum is taken over all pairs of measurable functions from $[0,1]$ to $[-1,+1]$.
Since the integral above is linear with respect to each of $f$ and $g$, the supremum
is attained at some functions taking values in $\{-1,+1\}$, and it follows immediately
that this version of the cut norm is again within a constant factor of that defined
by \eqref{cndefST}.
As noted in~\cite{BJRcm}, for example,
this last definition is the most natural from the point
of view of functional analysis: it is the dual of the projective tensor product norm in
$L^\infty\hat\otimes L^\infty$, and is thus the injective tensor
product norm in $L^1\check\otimes L^1$.
Equivalently, this is just the norm of the integral operator with kernel $\ka$,
treated as a map from $L^\infty$ to $L^1$.

Before turning to the cut metric we need one further definition.
Given a kernel $\ka$ and a measure-preserving map
$\tau:[0,1]\to [0,1]$,
let $\ka^{(\tau)}$ be the kernel defined by
\begin{equation}\label{pullback}
 \ka^{(\tau)}(x,y) = \ka(\tau(x),\tau(y)).
\end{equation}
If $\tau$ is a bijection, then
we call $\tau$ a {\em rearrangement} of $[0,1]$, and
$\ka^{(\tau)}$ a {\em rearrangement} of $\ka$.
(It is perhaps more natural to consider measure-preserving
bijections between two subsets of $[0,1]$ with measure 1; this makes no difference.)
Two kernels
$\ka_1$ and $\ka_2$ are {\em naively equivalent}
if one is a rearrangement of the
other, more precisely, if there is a rearrangement $\tau$ of $[0,1]$ such that
\begin{equation}\label{kequivN}
 \ka_1(x,y) = \ka_2^{(\tau)}(x,y) \quad \hbox{for a.e. }(x,y)\in [0,1]^2.
\end{equation}
In this case we write $\ka_1\approx\ka_2$, noting that $\approx$ is
an equivalence relation.

The {\em cut metric} $\dc$ on the set of standard kernels may be
defined as follows:
\begin{equation}\label{dcdef1}
 \dc(\ka_1,\ka_2) = \inf_{\ka_2'\approx \ka_2} \cn{\ka_1-\ka_2'}.
\end{equation}
Clearly, this defines a pseudo-metric on standard kernels; in particular,
if $\ka_1\approx\ka_2$, then $\dc(\ka_1,\ka_2)=0$. The
reverse implication does not hold; in fact,
$\dc(\ka_1,\ka_2)=0$ if and only if $\ka_1\sim\ka_2$,
where $\sim$ is the equivalence relation to be defined
in Subsection~\ref{ss_equiv}.
Hence, $\dc$ induces a metric on the set $\K$ of equivalence classes
of standard kernels under the relation~$\sim$.

As noted above, there is a standard
kernel $\ka_G$ naturally associated to each graph $G$,
although the map $G\mapsto \ka_G$ from $\F$ to $\K$ is not injective.
One extends the cut metric to a pseudo-metric on graphs by setting
\begin{equation}\label{dcdef}
 \dc(G_1,G_2) = \dc(\ka_{G_1},\ka_{G_2}),
\end{equation}
and to $\F\cup \K$ similarly.

For graphs $G_1$, $G_2$ on $n$ vertices, there is a much more natural
variant of their cut distance: let
$\dcG(G_1,G_2)$ be the smallest $\eps$ for which we can identify
the vertices of $G_1$ with those of $G_2$ such that for any
bipartition of the vertex set, the corresponding cuts in $G_1$ and
$G_2$ have sizes within $\eps n^2$.  In terms of kernels,
\begin{equation}\label{dcg}
 \dcG(G_1,G_2) = \min_{\ka\approx_n \ka_{G_2}} \cn{\ka_{G_1}-\ka},
\end{equation}
where $\ka_1\approx_n \ka_2$ if \eqref{kequivN} holds for some map
$\tau$ that simply permutes the intervals $I_n$
corresponding to the vertices, and we take \eqref{cndefSSc}
as the definition of the cut norm.
Note that the supremum implied by \eqref{cndefSSc} in the definition
\eqref{dcg}
is over all bipartitions of $[0,1]$, not just those corresponding
to bipartitions of the vertices; it is very easy to see that
this makes no difference: the supremum is attained at a vertex
bipartition.

Comparing \eqref{dcdef} and \eqref{dcg}, since the infimum in the former
is taken over a larger set, one trivially has
$\dc(G_1,G_2)\le \dcG(G_1,G_2)$.
Borgs, Chayes, Lov\'asz, S\'os and Vesztergombi~\cite{BCLSV:1} noted
that strict inequality is possible.
For example, taking \eqref{cndefSSc} as the definition of the cut norm,
let $G_1$ be a triangle, and let $G_2$ be the graph with $3$ vertices and one edge.
For any pairing of the vertices of $G_1$ with those of $G_2$,
the `worst' cut is the one in which the isolated vertex of $G_2$
is placed into one part and the other two vertices into the other part.
This cut has $2$ edges in $G_1$ but no edges in $G_2$,
so $\dcG(G_1,G_2)=2/9$. On the other hand, consider
the blow-ups $G_1^{(2)}$, a complete tripartite graph with two vertices
in each class, and $G_2^{(2)}$, a $C_4$
with two isolated vertices added. Pairing the vertices
of $G_1^{(2)}$ and $G_2^{(2)}$ by placing two opposite vertices of the $C_4$ in one
class of $G_1^{(2)}$, and the other vertices in different classes,
we realize $G_2^{(2)}$ as a subgraph of $G_1^{(2)}$ in such a way
that the $8$ edges of $G_1^{(2)}$ not present in $G_2^{(2)}$ form a non-bipartite
graph, so {\em every} cut cuts at most 7 of these extra edges. It follows that
$\dcG(G_1^{(2)},G_2^{(2)}) \le 7/6^2$. In fact, one can check that
with the vertices paired in this way the maximum difference between the
sizes of corresponding cuts in $G_1^{(2)}$ and $G_2^{(2)}$ is $6$, so
\[
 \dc(G_1,G_2) \le \dcG(G_1^{(2)},G_2^{(2)}) \le 6/6^2 = 1/6 < 2/9 = \dcG(G_1,G_2),
\]
showing that $\dc$ and $\dcG$ do not always agree.
For questions of convergence, however, the two metrics are equivalent:
as shown in~\cite{BCLSV:1},
\[
 \dc(G_1,G_2) \le \dcG(G_1,G_2) \le 32\dc(G_1,G_2)^{1/67}.
\]

At first sight it is not clear why the cut metric should be interesting:
after all, what is the significance of two graphs having almost the same number of
edges in all corresponding cuts? One very important consequence
of this property is that their subgraph counts are close,
as shown by the following simple lemma
from Borgs, Chayes, Lov\'asz, S\'os and Vesztergombi~\cite{BCLSV:1}.

\begin{lemma}\label{cs}
Let $\ka$ and $\ka'$ be two standard kernels. Then for every graph $F$
we have
\[
 |s(F,\ka)-s(F,\ka')|\le e(F) \cn{\ka-\ka'}.
\]
\end{lemma}
\begin{proof}
Before we embark on the proof,
we extend the definition of $s(F,\ka)$ slightly.
Fix the graph $F$, taking its vertex set to be $[k]=\{1,2,\ldots,k\}$,
as usual, and list the edges of $F$ as $\{i_1j_1,\ldots,i_mj_m\}$.
Given a sequence $(\ka_1,\ldots,\ka_m)$ of standard kernels,
set
\[
 s(F;\ka_1,\ldots,\ka_m) =
 \int_{[0,1]^k} \prod_{r=1}^m \ka_r(x_{i_r},x_{j_r})\prod_{i=1}^k \dd x_i.
\]
Thus $s(F,\ka)=s(F;\ka,\ldots,\ka)$. We claim that for any graph $F$
with $m$ edges and any standard kernels
$\ka_1,\ka_2,\ldots,\ka_m$ and $\ka_1'$, we have
\begin{equation}\label{dsmall}
 |s(F;\ka_1,\ka_2,\ldots,\ka_m) - s(F;\ka_1',\ka_2,\ldots,\ka_m)| \le \cn{\ka_1-\ka_1'}.
\end{equation}
Applying this $m=e(F)$ times, changing one kernel from $\ka$ to $\ka'$ each time,
the lemma follows.

It remains to prove \eqref{dsmall}, which is easy. Suppose without loss of generality
that the first edge is $12$, so $i_1=1$ and $j_1=2$. Our task is to bound
\[
 \Delta = \int_{[0,1]^k}  (\ka_1(x_1,x_2)-\ka_1'(x_1,x_2)) 
 \prod_{r=2}^m \ka_r(x_{i_r},x_{j_r}) \prod_{i=1}^k \dd x_i
\]
Collecting the terms in the product
that involve $x_1$ or $x_2$, we may write this product as $f_0(\bx) f_1(x_1,\bx)f_2(x_2,\bx)$,
where $\bx=(x_3,\ldots,x_k$) and each $f_i$ (being a product of standard kernels evaluated
at certain places)
takes values in $[0,1]$.
Now from \eqref{cndefST}, it is immediate that if $f$ and $g$ take values in $[0,1]$,
then $\left|\int \ka(x,y)f(x)g(y) \dd x \dd y\right|\le \cn{\ka}$.
Applying this with $\bx$ fixed, and then integrating over $\bx$, it follows that
$|\Delta|\le \cn{\ka_1-\ka_1'}$, as required.
\end{proof}
\begin{corollary}\label{cor_dcde}
Let $(G_n)$ be a sequence of graphs with $|G_n|\to\infty$, and let
$\ka$ be a standard kernel.
If $\dc(G_n,\ka)\to 0$ then $\de(G_n,\ka)\to 0$.
\end{corollary}
\begin{proof}
Let $\ka_n=\ka_{G_n}$, so by definition $\dc(G_n,\ka)=\dc(\ka_n,\ka)$.
By Lemma~\ref{cs}, for every $F$ we have
$s(F,\ka_n)\to s(F,\ka)$. But $s(F,\ka_n)=t(F,G_n)$, while
from \eqref{st} we have $s(F,G_n)=t(F,G_n)+o(1)$.
Thus $s(F,G_n)\to s(F,\ka)$ for each $F$, i.e., $\de(G_n,\ka)\to 0$.
\end{proof}

We have just seen that convergence in $\dc$ implies convergence
in $\de$; one of the main results of
Borgs, Chayes, Lov\'asz, S\'os and Vesztergombi, namely Theorem 2.6
in~\cite{BCLSV:1}, gives a converse of this.
This result states 
that the metrics $\de$ (defined using $t$ rather than $s$)
and $\dc$ are equivalent, in the sense that
$(G_n)$ is a Cauchy sequence for $d_{\rm sub}$ if and
only if it is a Cauchy sequence for $d_{\rm cut}$.
In the light of the various other results of Lov\'asz and
Szegedy~\cite{LSz1} and Borgs, Chayes, Lov\'asz, S\'os and
Vesztergombi~\cite{BCLSV:1}, this statement may be reformulated
in our notation as follows.
\begin{theorem}\label{dedc}
Let $(G_n)$ be a sequence of graphs or standard kernels with $|G_n|\to\infty$,
where we take $|G_n|=\infty$ if $G_n$ is a kernel, and let
$\ka$ be a standard kernel.
Then $\de(G_n,\ka)\to 0$ if and only if $\dc(G_n,\ka)\to 0$.
\noproof
\end{theorem}

An immediate consequence of this result is the following,
Corollary~3.10 in~\cite{BCLSV:1}.
\begin{corollary}\label{csk}
Let $\ka$ and $\ka'$ be two bounded
kernels. Then $s(F,\ka)=s(F,\ka')$ for
every $F$ if and only if $\dc(\ka,\ka')=0$.
\noproof
\end{corollary}
We shall return to a discussion of kernels at cut distance 0 shortly.

\subsection{Kernels and (quasi-)random graphs}\label{ss_rg}
As well as going from graphs to kernels, one can go from kernels to {\em random}
graphs in a very natural way, as in Section 2.6 of Lov\'asz and Szegedy~\cite{LSz1},
or as in Bollob\'as, Janson and Riordan~\cite{BJR} for the sparse case.
Indeed, given a standard kernel $\ka$
and an $n\ge 1$, let $G(n,\ka)$ be the random graph on $[n]$ defined as follows:
first let $x_1,\ldots,x_n$ be \iid\ with the uniform distribution
on $[0,1]$. Given the $x_i$, join each pair of vertices independently,
joining $i$ and $j$ with probability $\ka(x_i,x_j)$. 
The resulting graph is called a {\em $\ka$-random graph} by Lov\'asz and Szegedy~\cite{LSz1},
although they use $W$ as their default symbol for a kernel.
It is easy to check, for example by the second moment method,
that, for each $F$, the random variable $s(F,G(n,\ka))$ converges
(in probability and in fact almost surely) to $s(F,\ka)$ as $n\to\infty$.
Thus the sequence $G(n,\ka)$ converges almost surely to $\ka$ in the metric
$\de$ or $\dc$. Note that if $\ka$ is constant and takes the value $p$,
then we recover the usual Erd\H os--R\'enyi model $G(n,p)$: no confusion
should arise between the notation for the two models.
(In fact, it was Gilbert~\cite{Gilbert_Gnp} who introduced $G(n,p)$, while
Erd\H os and R\'enyi~\cite{ER_RG1} introduced a model, $G(n,m)$, that is essentially equivalent  
for many purposes. Since it was they who founded the theory of random graphs,
both models are often referred to as Erd\H os--R\'enyi models.)

It is natural to view a sequence $(G_n)$ converging to $\ka$ in $\de$
as a sequence of `inhomogeneous quasi-random graphs': when $\ka$
is constant, the convergence condition is equivalent to the standard
notion of quasi-randomness, introduced by Thomason~\cite{Tho} in 1987
(although he called it pseudo-randomness) and studied in great detail
by Chung, Graham and Wilson~\cite{CGW89} and many others.
The convergence of $G(n,\ka)$ to $\ka$ in $\de$ establishes
that sequences generated by the natural inhomogeneous random
model are also quasi-random, as one would hope. One of the most
pleasing features of this whole subject area is the interpretation
that inhomogeneous quasi-random graphs are completely general:
any sequence of (dense) graphs has such a subsequence.

To take an alternative viewpoint, we may think of standard kernels as 
uncountable infinite graphs, and a `typical'
random graph $G(n,\ka)$ as a good finite approximation to $\ka$. Then
the completion of $\F$ is obtained by adding these infinite graphs,
and the approximations $G(n,\ka)$ ($n$ large) are examples
of finite graphs close to a given infinite graph. Taking this viewpoint
it is natural {\em not} to identify a finite graph with a kernel.
For another, slightly different, point of view,
see Diaconis and Janson~\cite{SJ209}, where connections to certain
infinite random graphs are described.

\subsection{Equivalent kernels}\label{ss_equiv}
In the light of Corollary~\ref{csk}, it is clearly important
to understand which pairs of kernels have $\dc(\ka_1,\ka_2)=0$;
this is also important for understanding $\dc$ itself.
Fortunately, it turns that there is a natural notion of equivalence
for kernels which gives the answer.
Since this topic is only touched on in passing in
Borgs, Chayes, Lov\'asz, S\'os and Vesztergombi~\cite{BCLSV:1},
we shall go into some detail here.

Roughly speaking, we would like to say that two kernels are
equivalent if one is obtained from the other simply by relabelling
the `types' in $[0,1]$. It would seem that 
the notion $\approx$ of naive equivalence defined in \eqref{kequivN}
is thus the right one,
but a little thought shows that this is not the case; for this,
the random viewpoint is very helpful.

So far, as in~\cite{BCLSV:1}, we defined kernels only on $[0,1]^2$.
In view of the connection to random graphs discussed in the previous
subsection, it is {\em a priori} more natural to work with
a general probability space $(\Omega,\F,\mu)$ rather than $[0,1]$
with Lebesgue measure, defining a standard kernel
as a symmetric measurable function from the square of a probability space
to $[0,1]$.
(This is the approach taken in the sparse case
by Bollob\'as, Janson and Riordan~\cite{BJR}.)
However, almost all the time, we shall consider only kernels
on $[0,1]$; there are two reasons for doing so: firstly, graphs with $n$
vertices correspond
to kernels on the discrete space with $n$ equiprobable elements,
and $[0,1]$ is the natural limit of these spaces. Secondly,
all probability spaces that one would ever wish to work
with (all so-called `standard' probability spaces) are isomorphic to Lebesgue
measure on an interval, combined with (possibly) a finite or countable
number of atoms.
When studying kernels, the presence of atoms makes no difference:
for example, a kernel on a finite measure space corresponds in a natural
way to a piecewise constant kernel on $[0,1]$. Hence it makes very good sense
to consider only kernels on $[0,1]$. For a formal reduction
to the case of kernels on $[0,1]$ in the context of random graphs,
see Janson~\cite{SJ210}.

We may think of kernels as two-dimensional versions of random variables
(not to be confused with vector valued random variables). Two
random variables are equivalent if they have the same distribution. 
Equivalently, they are equivalent if they may be coupled so as to
agree with probability $1$. This is the definition we shall use
for kernels.

Working, for the moment, on general (standard) probability spaces,
and suppressing the $\sigma$-field of measurable
sets in the notation, let $(\Omega_1,\mu_1)$ and $(\Omega_2,\mu_2)$ be two
probability spaces. A {\em coupling}
of $(\Omega_1,\mu_1)$ and $(\Omega_2,\mu_2)$ is simply a probability
space $(\Omega,\mu)$ together with measure-preserving maps
$\sigma_i:\Omega\to\Omega_i$, $i=1,2$. Thus, if $X$ is a uniformly
random point of $(\Omega,\mu)$, then $\sigma_1(X)$ and $\sigma_2(X)$
are uniform on $(\Omega_1,\mu_1)$ and $(\Omega_2,\mu_2)$, respectively.
Let $\ka_i$ be a kernel on $(\Omega_i,\mu_i)$, $i=1,2$.
Then $\ka_1$ and $\ka_2$ are {\em equivalent} if there
is a coupling $(\Omega,\mu)$ of the underlying probability spaces such that
\[
 \ka_1(\sigma_1(x),\sigma_1(y)) = \ka_2(\sigma_2(x),\sigma_2(y))
\hbox{ for $(\mu\times\mu)$-almost every }(x,y)\in \Omega^2.
\]
In other words, extending the notation in \eqref{pullback} to arbitrary
spaces, we require $\ka_1^{(\sigma_1)}=\ka_2^{(\sigma_2)}$ a.e.;
we write $\sim$ for the corresponding relation. Although this definition
may seem a little complicated, as explained above it is in fact
very natural.

Note that $\ka_1\approx\ka_2$ implies $\ka_1\sim\ka_2$:
if $\ka_1=\ka_2^{(\tau)}$, then one couples $x\in [0,1]=\Omega_1$
with $\tau(x)\in \Omega_2$. (More formally, we may take
$\Omega=\Omega_1$, with $\sigma_1$ the identity and $\sigma_2=\tau$.)
It is easy to see that the reverse implication does not hold: for example,
consider the random variables $\La_1$, $\La_2$ on $[0,1]$ given
by $\La_1(x)=x$ and $\La_2(x)=2x-\floor{2x}$; these both have
the uniform distribution, but since one is $1$-to-$1$ and
the other $2$-to-$1$, there is no measure-preserving
bijection from one ground space to the other
transforming one into the other. Setting $\ka_i(x,y)=\La_i(x)\La_i(y)$,
one obtains kernels with $\ka_1\sim \ka_2$ but $\ka_1\not\approx\ka_2$.
(Recently, Borgs, Chayes and Lov\'asz~\cite{BCL:unique}
have shown that if one excludes this phenomenon of `twins',
then $\sim$ and $\approx$ are equivalent; we refer the reader there
for a precise statement.)

Returning to the special case
of kernels on $[0,1]$, essentially equivalent
to the general case, couplings have a very simple description.
All that matters is that, for a uniform point $X$
of $(\Omega,\mu)$, the distribution of $(\sigma_1(X),\sigma_2(X))$
should have uniform marginals. Thus, couplings correspond to
{\em doubly stochastic measures},
i.e., Borel measures $\mu$ on $[0,1]^2$ with
both marginals Lebesgue measure. In other words, we have $\ka_1\sim\ka_2$
if and only if there is a doubly stochastic measure $\mu$ such that
\begin{equation}\label{ed2}
 \ka_1(x,y)=\ka_2(u,v)\hbox{ for $(\mu\times\mu)$-a.e. $(x,u,y,v)\in [0,1]^4$}.
\end{equation}

At first sight, $[0,1]^2$ is the most natural space to use to couple two kernels
on $[0,1]$, but there is another natural choice. Since $[0,1]^2$ is isomorphic
as a probability space to $[0,1]$, we may construct the coupling on $[0,1]$!
Hence, $\ka_1\sim \ka_2$ if and only if there are measure preserving maps
$\sigma_1,\sigma_2:[0,1]\to[0,1]$ such that $\ka_1^{(\sigma_1)}=\ka_2^{(\sigma_2)}$
for (Lebesgue) a.e. $(x,y)\in [0,1]^2$. Putting this a little more symmetrically,
we see that $\ka_1\sim\ka_2$ if and only if
\begin{equation}\label{kup}
 \exists \ka,\sigma_1,\sigma_2\hbox{ such that }
\ka=\ka_1^{(\sigma_1)} \hbox{ a.e \quad and\quad} \ka=\ka_2^{(\sigma_2)} \hbox{ a.e},
\end{equation}
where $\ka$ is a kernel on $[0,1]$ and $\sigma_1$ and $\sigma_2$ are 
measure-preserving maps from $[0,1]$ to itself. Note that $\ka\sim\ka^{(\sigma)}$
for any kernel $\ka$ on $[0,1]$ and any measure-preserving map from $[0,1]$ to itself.

Since couplings rather than rearrangements give the proper notion of equivalence
for two kernels, it is natural to use couplings rather than rearrangements
in the definition of the cut metric. Indeed, 
Borgs, Chayes, Lov\'asz, S\'os and Vesztergombi~\cite{BCLSV:1} define
the cut metric on standard (or simply bounded) kernels as follows:
\begin{equation}\label{dcdef2}
\dc(\ka_1,\ka_2) = \inf_{\mu\in\M}
 \sup_{S,T}\left| \int_{S\times T}
 \bb{\ka_1(x,y)-\ka_2(u,v)}\dd\mu(x,u)\dd\mu(y,v) \right|,
\end{equation}
where $\M$ is the set of doubly stochastic measures on $[0,1]^2$, $S$ and $T$
run over measurable subsets of $[0,1]^2$, and the integral
is over $(x,u)\in S$ and $(y,v)\in T$.
As shown in~\cite{BCLSV:1}, the definitions \eqref{dcdef1} and \eqref{dcdef2}
coincide. (This is not hard to see -- either formula defines a function
that is continuous, indeed Lipschitz with constant~1, with respect to the
cut norm, and hence continuous with respect to the $L^1$ norm.
Since the finite-type kernels are dense in $L^1$,
it suffices to
check the equality of the two definitions
for finite-type kernels, which is straightforward. For
the details, see~\cite{BCLSV:1}.) Since \eqref{dcdef1} is much easier
to work with than \eqref{dcdef2}, we shall take
the former as our definition of $\dc$.

Although \eqref{dcdef1} is more convenient, there is a sense in which
\eqref{dcdef2} is the `right' definition. For example,
as we shall now show, the infimum in \eqref{dcdef2}
is always attained, unlike that in \eqref{dcdef1}. This is not
discussed in~\cite{BCLSV:1}, where it is of no particular significance.
Here, as in the bulk of the paper, unless otherwise specified,
all kernels are kernels on $[0,1]$,
i.e., symmetric Lebesgue-measurable functions from $[0,1]^2\to [0,\infty)$.
As noted above, it always suffices to consider kernels
on $[0,1]$. Recall that we call a kernel {\em standard} if it takes values in $[0,1]$.

\begin{lemma}\label{infatt}
Let $\ka_1$ and $\ka_2$ be two standard kernels. Then 
there is a doubly stochastic measure $\mu$ achieving the infimum
in \eqref{dcdef2}.
\end{lemma}
\begin{proof}
For $\mu\in\M$ set
\begin{equation}\label{dmu}
 d_\mu(\ka_1,\ka_2)=  \sup_{S,T}\left| \int_{S\times T}
 \bb{\ka_1(x,y)-\ka_2(u,v)}\dd\mu(x,u)\dd\mu(y,v) \right|,
\end{equation}
so our aim is to show that $\inf_{\mu\in \M} d_\mu(\ka_1,\ka_2)$
is attained. Before doing so, let us note that in the supremum one may restrict
the sets $S$ and $T$ in \eqref{dmu} to `nice' sets.
Let $\D$ denote the set of 
finite unions of products of (half-open) intervals.
Since $\mu$ is a finite Borel measure, for any measurable
$S,T\subset[0,1]^2$ and any $\eps>0$, there are
sets $S',T'\in \D$ with $\mu(S\Delta S'), \mu(T\Delta T')<\eps$.
Since $\ka_1-\ka_2$ is bounded by $\pm 1$, replacing $S$, $T$ by
$S'$ and $T'$ changes the value of the integral by at most $2\eps$.
It follows that the supremum in \eqref{dmu} may
be taken over $S$, $T\in \D$ without changing its value,
as claimed.

It is well known that $\M$ is (sequentially) compact in the topology
in which $\mu_n\to\mu$ if and only if $\mu_n(A)\to\mu(A)$
for every set $A\in \D$. Indeed, writing $\D_0$ for the
set of products of intervals with rational endpoints,
since $\D_0$ is countable any sequence in $\M$
has a subsequence $(\mu_n)$ such that $(\mu_n(A))$ converges
for all $A\in \D_0$. Using the doubly stochastic property
to bound the measure of a rectangle with one or more short sides,
convergence for all $A\in \D$ follows easily,
and one can check that the limiting values do define
a measure $\mu$.
Note that one cannot require $\mu_n(A)\to \mu(A)$ for every 
measurable $A$:
it is easy to construct sequences where $\mu$
is concentrated on, for example,
the diagonal $S=\{(x,x)\}$, with $\mu_n(S)=0$ for every $n$.

Let $(\mu_n)$ be a sequence of doubly stochastic measures for which
$d_{\mu_n}(\ka_1,\ka_2)\to \dc(\ka_1,\ka_2)$; such a sequence
exists by the definition \eqref{dcdef2} of $\dc(\ka_1,\ka_2)$.
From the remark above, $(\mu_n)$ has a subsequence
converging to some $\mu\in \M$ in the appropriate topology.
Restricting to this subsequence, we may assume that $\mu_n(A)\to \mu(A)$
for every $A\in \D$.

Let $S=S_1\times S_2$ and $T=T_1\times T_2$, where $S_1,S_2,T_1$ and $T_2$
are all intervals in $[0,1]$. We claim that
\begin{equation}\label{ecl}
 \int_{S\times T}\ka(x,y)\dd\mu_n(x,u)\dd\mu_n(y,v) \to
 \int_{S\times T}\ka(x,y)\dd\mu(x,u)\dd\mu(y,v)
\end{equation}
as $n\to\infty$, for any standard kernel $\ka$. Before proving
this, let us show that the lemma follows.

For any $\nu\in \M$, let
\[
 f(S,T,\nu)= \int_{S\times T}
 \bb{\ka_1(x,y)-\ka_2(u,v)}\dd\nu(x,u)\dd\nu(y,v),
\]
so $d_\nu(\ka_1,\ka_2)=\sup_{S,T} |f(S,T,\nu)|$.
Applying \eqref{ecl} with $\ka=\ka_1$ and $\ka=\ka_2$,
we see that
$f(S,T,\mu_n)\to f(S,T,\mu)$ holds whenever $S$ and $T$ are products
of intervals. By additivity, it thus holds whenever $S$ and $T$
are in $\D$.
Since $d_{\mu_n}(\ka_1,\ka_2)=\sup_{S,T} |f(S,T,\mu_n)|$, for
$S$, $T\in \D$ we thus have
\[
 f(S,T,\mu) =\liminf f(S,T,\mu_n) \le \liminf d_{\mu_n}(\ka_1,\ka_2)=
 \dc(\ka_1,\ka_2).
\]
As noted earlier, when defining $d_\mu(\ka_1,\ka_2)=\sup_{S,T}|f(S,T,\mu)|$,
we may take the supremum instead over $S,T\in \D$, so it follows
that $d_\mu(\ka_1,\ka_2)\le \dc(\ka_1,\ka_2)$.
Since $\dc(\ka_1,\ka_2)=\inf_{\mu'\in \M} d_{\mu'}(\ka_1,\ka_2)$, this
infimum is attained (at $\mu$), as claimed.

It remains to prove \eqref{ecl}. But this is easy: for any interval
$I\subset [0,1]$, let
$\mu_n^I$ be the measure on $[0,1]$ defined by
\[
 \mu_n^I(A)=\mu_n(A\times I),
\]
and define $\mu^I$ from $\mu$ similarly.
Recall that $\mu_n\to \mu$ on products of intervals. Thus
$\mu_n^I(A)\to \mu^I(A)$ whenever $A$ is an interval,
and hence whenever $A$ is a finite union of intervals.
Since $\mu_n,\mu\in \M$, we have that $\mu_n(A)$ and $\mu(A)$ are
both at most the Lebesgue measure of $A$.
It follows that $\mu_n^I(A)\to \mu^I(A)$ for any measurable
$A\subset [0,1]$, since for any $\eps$ we can approximate
$A$ by a finite union of intervals $A'$ whose symmetric
difference from $A$ has Lebesgue measure at most $\eps$.
It also follows that if $I$ and $J$ are two intervals,
and $A\subset [0,1]^2$ is Lebesgue measurable, then
\[
 (\mu_n^I\times \mu_n^J)(A) \to (\mu^I\times \mu^J)(A).
\]
Indeed, this follows by approximating $A$ by a finite union of products of intervals.
Considering level sets, we see that
\[
 \int f(x,y) \dd\mu_n^I(x)\dd\mu_n^J(y) \to
 \int f(x,y) \dd\mu^I(x)\dd\mu^J(y)
\]
for any bounded measurable function $f$. Taking
$f=\ka(x,y)1_{x\in S_1}1_{y\in T_1}$, $I=S_2$ and $J=T_2$,
this is exactly \eqref{ecl}, completing the proof.
\end{proof}

The special case of Lemma~\ref{infatt} where the distance is 0
is of particular interest.
\begin{corollary}\label{0att}
Let $\ka_1$ and $\ka_2$ be two standard kernels. Then $\dc(\ka_1,\ka_2)=0$
if and only if $\ka_1\sim \ka_2$.
\end{corollary}
\begin{proof}
Using \eqref{dcdef2} as the definition of $\dc$, if $\ka_1\sim\ka_2$
then we certainly have $\dc(\ka_1,\ka_2)=0$; see \eqref{ed2}.

Suppose then than $\dc(\ka_1,\ka_2)=0$. From Lemma~\ref{infatt}, there
is a $\mu\in \M$ such that $d_\mu(\ka_1,\ka_2)=0$.
Let $\nu$ be the signed measure on $[0,1]^4$ defined by
\[
 \dd\nu(x,u,y,v) =\bb{\ka_1(x,y)-\ka_2(u,v)}\dd\mu(x,u)\dd\mu(y,v).
\]
Then $d_\mu(\ka_1,\ka_2)=0$ says exactly that $\nu(S\times T)=0$
for all measurable $S,T\subset [0,1]^2$. Since $\nu$ is a signed
Borel measure, it follows immediately that $\nu$ is the zero measure.
Equivalently, 
$\ka_1(x,y)-\ka_2(u,v)=0$ for $(\mu\times\mu)$-a.e.
points $(x,u,y,v)$. Referring to \eqref{ed2} again, we see that
$\ka_1\sim\ka_2$.
\end{proof}

As we have seen, Corollary~\ref{0att} is a simple exercise in measure theory.
Using this corollary, and the equivalence of $\dc$ and $\de$ proved
by Borgs, Chayes, Lov\'asz, S\'os and Vesztergombi~\cite{BCLSV:1},
one obtains the following characterization of equivalent (standard) kernels. 
\begin{theorem}\label{th_kequiv}
Let $\ka_1$ and $\ka_2$ be two standard kernels. Then
$s(F,\ka_1)=s(F,\ka_2)$ holds for every finite graph $F$
if and only if $\ka_1\sim \ka_2$.
\end{theorem}
\begin{proof}
Immediate from Corollaries~\ref{csk} and~\ref{0att}.
\end{proof}

The analogue of Theorem~\ref{th_kequiv} for general (i.e., unbounded)
kernels is false,
even for `rank 1' kernels with all counts $s(F,\ka)$ finite.
Indeed, if $\ka(x,y)=f(x)f(y)$ for some
$f:[0,1]\to [0,\infty)$, then the quantities $s(F,\ka)$ are easily seen to
be products of moments of $f$, viewed as a random variable.
As is well known, there are non-negative random variables with
the same finite moments but different distributions; using
two such random variables, one can construct non-equivalent
unbounded kernels $\ka_1,\ka_2$ with $s(F,\ka_1)=s(F,\ka_2)<\infty$
for all $F$.

We have shown that it is not hard to deduce Theorem~\ref{th_kequiv}
from Theorem~\ref{dedc}. In fact, these results are equivalent!
The reverse implication is actually much easier.

\begin{proof}[Theorem~\ref{th_kequiv} $\implies$ Theorem~\ref{dedc}]
We write out the argument for a sequence of graphs; the treatment 
for kernels is essentially the same.
Let $(G_n)$ be a sequence of graphs with $|G_n|\to\infty$, and
let $\ka$ be a standard
kernel. From Corollary~\ref{cor_dcde}, if $\dc(G_n,\ka)\to 0$, then $\de(G_n,\ka)\to 0$;
it remains to prove the reverse implication.

As shown by Lov\'asz and Szegedy~\cite{LSz1} (see their Lemmas 5.1 and 5.2),
repeatedly applying even the weak Frieze--Kannan~\cite{FKquick}
form of Szemer\'edi's Lemma, it is easy
to prove that any sequence $(G_n)$ with $|G_n|\to\infty$ has a subsequence
converging in $\dc$ to some standard kernel $\ka'$. We shall not give the details
of this argument here as we shall prove a corresponding statement
in a more general setting in Corollary~\ref{cSz}.

Suppose then that $\de(G_n,\ka)\to 0$. Then by the observation above
there is a subsequence $(G_{n_k})$ that
converges in $\dc$ to some standard kernel $\ka'$. But then,
by Corollary~\ref{cor_dcde}, we have $\de(G_{n_k},\ka')\to 0$.
Since $\de(G_n,\ka)\to 0$ we must have $\de(\ka,\ka')=0$, i.e., $s(F,\ka)=s(F,\ka')$
for all $F$. Thus, by Theorem~\ref{th_kequiv}, we have $\ka\sim\ka'$,
so $\dc(\ka,\ka')=0$. Thus $\dc(G_{n_k},\ka)\to 0$.

We have shown that $(G_n)$ has a subsequence converging to $\ka$ in $\dc$.
This argument applies equally well to any subsequence of $(G_n)$,
and it follows immediately that the whole sequence converges, i.e.,
$\dc(G_n,\ka)\to 0$, as required.
\end{proof}

As we have just seen, Theorem~\ref{dedc}, one of the main results
of Borgs, Chayes, Lov\'asz, S\'os and Vesztergombi~\cite{BCLSV:1},
is equivalent to Theorem~\ref{th_kequiv}. As far as we are aware,
this observation is new.
Now Theorem~\ref{th_kequiv} is a fundamental
analytic fact about
bounded kernels: it says
that a bounded kernel is characterized up to equivalence by the quantities
$s(F,\ka)$, which are the natural analogues for a kernel of the moments
of a random variable. When the first version of this paper was written,
we thus had the following rather unsatisfactory situation: the only
known proof of the analytic fact Theorem~\ref{th_kequiv} 
was that given above, relying on the hard results of
Borgs, Chayes, Lov\'asz, S\'os and Vesztergombi~\cite{BCLSV:1}
about sequences of graphs. Fortunately, this situation
has now been resolved: Borgs, Chayes and Lov\'asz~\cite{BCL:unique} have given
a very clever direct proof of Theorem~\ref{th_kequiv}.
In fact, they proved a little more.

Recall from \eqref{kup} that $\ka_1\sim\ka_2$ means that
\[
 \exists \ka,\sigma_1,\sigma_2\hbox{ such that }
\ka=\ka_1^{(\sigma_1)} \hbox{ a.e \quad and\quad} \ka=\ka_2^{(\sigma_2)} \hbox{ a.e},
\]
where $\ka$ is a kernel on $[0,1]$ and $\sigma_1$ and $\sigma_2$ are 
measure-preserving maps from $[0,1]$ to itself.
Turning this `upside-down', let us write $\ka_1\sim'\ka_2$ if
\begin{equation}\label{reverse}
 \exists \ka,\sigma_1,\sigma_2\hbox{ such that }
\ka_1=\ka^{(\sigma_1)} \hbox{ a.e \quad and\quad} \ka_2=\ka^{(\sigma_2)} \hbox{ a.e}.
\end{equation}
In \eqref{reverse}, we require $\ka$ to be a kernel on $[0,1]$;
it makes no difference if we allow $\ka$ to be a kernel on an
arbitrary standard probability space.
Note that if $\ka_1\sim'\ka_2$, then
using the observation that $\ka\sim \ka^{(\sigma)}$ twice, we have $\ka_1\sim\ka_2$.

Borgs, Chayes and Lov\'asz~\cite{BCL:unique} proved the following result.
\begin{theorem}\label{unique2}
For two bounded kernels $\ka_1$, $\ka_2$, the following are equivalent.
(a) $s(F,\ka_1)=s(F,\ka_2)$ for every finite graph $F$,
(b) $\ka_1\sim\ka_2$ and (c) $\ka_1\sim'\ka_2$.
\end{theorem}
The important implication is that if $s(F,\ka_1)=s(F,\ka_2)$ for all $F$,
then $\ka_1\sim'\ka_2$. As noted above, this trivially implies $\ka_1\sim\ka_2$,
which in turn easily implies $s(F,\ka_1)=s(F,\ka_2)$.
The proof in~\cite{BCL:unique} is direct, but somewhat technical.

As shown above, Theorem~\ref{unique2}, which trivially implies Theorem~\ref{th_kequiv},
implies Theorem~\ref{dedc}. This gives a proof of Theorem~\ref{dedc} 
that is very different from that given by
Borgs, Chayes, Lov\'asz, S\'os and Vesztergombi~\cite{BCLSV:1}.

\bigskip
Our aim in the rest of this paper is to investigate the extent to which the various
results and observations above carry over to sparse graphs, graphs
with $n$ vertices and $o(n^2)$ edges. As we shall see,
this gives rise to many difficult questions, so we shall present
many more questions than answers.

\section{Subgraph counts for sparse graphs}\label{sec_hsp}

In this section we consider sparse graphs, where the number of edges
is $o(n^2)$ as the number $n$ of vertices goes to infinity. We shall
assume throughout that we have at least $\omega(n)$ edges, i.e., that
the average degree tends to infinity; often, we shall make much
stronger assumptions.
Given a function $p=p(n)$, one can adapt many of the notions
of Section~\ref{dense} to graphs with $\Theta(pn^2)$ edges. Indeed,
let
\[
 s_p(F,G) = \frac{\emb(F,G)}{p^{e(F)}n_{(|F|)}} = \aut(F)\frac{X_F(G)}{p^{e(F)}X_F(K_n)},
\]
noting that
\[
 s_p(F,G) = \frac{\emb(F,G)}{ \E\bb{\emb(F,G(n,p))}}.
\]
Also, let
\[
 t_p(F,G) = \frac{\hom(F,G)}{p^{e(F)}n^{|F|}}.
\]
If $p=1$, then we recover the definitions in Section~\ref{dense}.
Furthermore, if $0<p<1$ is constant,
then we can define a map $s$ as before, but now $s$ maps $\F$ into the compact
space $\prod_{F\in \F}[0,p^{-e(F)}]$, and everything proceeds as before.
More generally, changing $p$ by a constant factor will be irrelevant: just
as we can use $s_c$ for any $c$ to study $G(n,1/2)$, we may use $s_p$ to study $G(n,p/2)$ or $G(n,2p)$,
say, for any $p=p(n)$.

From now on, we suppose that $p=p(n)$ is some given function of $n$, with $p(n)\to 0$
as $n\to\infty$. We wish to work in a compact space, so we shall assume
that there are constants $c_F$, $F\in \F$, such that $s_p(F,G)\le c_F$ for all graphs
$G$ we consider. Enumerating $\F$ as $\{F_1,F_2,\ldots\}$, we may thus
define a map
\begin{equation}\label{ds1}
 s_p:\F\to X=\prod_{i=1}^\infty [0,c_{F_i}], \qquad  G\mapsto (s_p(F_i,G))_{i=1}^\infty,
\end{equation}
and, using any metric $d$ on $X$ giving the product topology,
an associated metric
\begin{equation}\label{ds2}
 \de(G_1,G_2) = d(s_p(G_1),s_p(G_2)).
\end{equation}
We suppress the dependence on $p$ in our notation for the metric to avoid
clutter. As in the dense case, we can extend $\de$ to bounded kernels $\ka$,
setting
\[
 \de(G,\ka) = d(s_p(G),s(\ka)) \hbox{\quad and\quad} \de(s(\ka_1),s(\ka_2))=d(s(\ka_1),s(\ka_2))
\]
for a graph $G$ and bounded kernels $\ka$, $\ka_1$ and $\ka_2$. Here, for
a kernel $\ka$, $s(\ka)$ is the vector with coordinates
defined by \eqref{sfk}.

Much of the time, we think of a sequence $(G_n)$ of finite graphs.
Throughout, we are only interested in sequences with $|G_n|\to\infty$.
For notational convenience we always assume that $|G_n|=n$; this make no difference
to our conjectures and results. As usual, we need not assume that $G_n$ is defined
for every $n\in \NN$, but only for an infinite subset of $\NN$. In this setting,
the assumption described above may be stated as follows.

\begin{assumption}[bounded subgraph counts]\label{AA}
For each fixed graph $F$, we have $\sup_n s_p(F,G_n)<\infty$.
\end{assumption}

In particular, if $(G_n)$ satisfies Assumption~\ref{AA} then, taking
$F=K_2$, we see that $e(G_n)=O(pn^2)$, so our graphs are
sparse. There is a stronger version of Assumption~\ref{AA} that is perhaps
even more natural:

\begin{assumption}[exponentially bounded subgraph counts]\label{AB}
There is a constant $C$ such that, for each fixed $F$,
we have $\limsup s_p(F,G_n)\le C^{e(F)}$ as $n\to\infty$.
\end{assumption}

In this case, changing $p$ by a constant factor, we may take $C=1$ if
we like.  This is not always the most natural normalization, however.
There is a reason for writing $\limsup$ in Assumption~\ref{AB}: for any graph
$G_n$ with $|G_n|=n$ and $n$ large, there will be some $F$ with
$s_p(F,G_n)$ very large.  Indeed, $G_n$ contains at least one
embedding of itself, so $s_p(G_n,G_n)\ge 1/(n!p^{e(G_n)})$, which
typically grows much faster than any constant to the power $e(G_n)$.

Turning to kernels, there is no longer any good reason to restrict
our kernels to take values in $[0,1]$: in the dense case, the maximum
possible `local density' of edges is $1$. Here, if we normalize
so that $G_n$ has $pn^2/2$ edges, say, local densities larger than
$p$ are certainly possible. We shall thus consider
general kernels, i.e., symmetric measurable functions from $[0,1]^2$
to $[0,\infty)$, rather than only standard kernels. We define $s(F,\ka)$ as before,
using \eqref{sfk}; in general, $s(F,\ka)$ may be infinite,
but we shall always assume it is finite for the graphs
$F$ and kernels $\ka$ we consider.

Although we allow unbounded kernels in general, it may be that they
give rise to difficulties (as they do in the general (very) sparse
inhomogeneous model of Bollob\'as, Janson and Riordan~\cite{BJR}).
Assumption~\ref{AB} corresponds to the limiting kernel (if it exists)
being bounded, as shown by Lemma~\ref{lbd} below.

Our main conjecture states that, if $p$ is large enough,
then, under Assumption~\ref{AB}, the equivalent of Theorem~\ref{ctok} holds.

\begin{conjecture}\label{q1a}
Let $p=p(n)=n^{-o(1)}$, and let $C>0$ be constant.
Suppose that $(G_n)$ is a sequence of graphs with $|G_n|=n$ such that, for every
$F$, $s_p(F,G_n)$ converges to some constant $0\le c_F\le C^{e(F)}$.
Then there is a bounded kernel $\ka$ such that $c_F=s(F,\ka)$ for every $F$.
\end{conjecture}
As noted above, without loss of generality we may take $C=1$.
As we shall observe later, it is very easy to see
that if $s_p(K_2,G_n)\to 0$ and $s_p(F,G_n)$ is
bounded for every $F$, then $s_p(F,G_n)\to 0$ for every $F$.
Thus we may assume that $s_p(K_2,G_n)$ is bounded away from zero,
and we may normalize in a different way  by assuming that
$s_p(K_2,G_n)=1$, i.e, that $e(G_n)=p\binom{n}{2}$.

Assumption~\ref{AB} is trivially stronger than Assumption~\ref{AA}. Thus, if
$(G_n)$  satisfies Assumption~\ref{AB}, then the sequence $s_p(G_n)$ defined
by \eqref{ds1} lives
in a compact product space, and has a convergent subsequence. Hence
there are real numbers $c_F\ge 0$, $F\in \F$, and a subsequence
$(G_{n_i})$ with $s_p(F,G_{n_i})\to c_F$ for every $F$, to which
Conjecture~\ref{q1a} applies. Conjecture~\ref{q1a} is thus a statement about the possible limit
points of the sequences $s_p(G_n)$.

It may well be that the restriction to bounded kernels is not necessary.

\begin{conjecture}\label{q1}
Let $p=p(n)=n^{-o(1)}$, and let $(G_n)$ be a sequence of graphs with
$|G_n|=n$ such that, for every $F$, we have $s_p(F,G_n)\to c_F$
for some $0\le c_F<\infty$. Then there is a kernel $\ka$ with
$c_F=s(F,\ka)$ for every $F$.
\end{conjecture}

We have stated the above conjectures under the assumption that
$p=n^{-o(1)}$; we shall call this the {\em almost dense} case. The
reason for this assumption is discussed further below. Let us note
that, in the almost dense case, for each fixed $F$ with $k$ vertices,
the denominator in the formula $\emb(F,G_n)/(p^{e(F)}n_{(k)})$ for $s_p(F,G_n)$
is asymptotically $p^{e(F)}n^k$, which is $n^{k-o(1)}$.
Since there are at most $n^{k-1}$ non-injective homomorphisms from $F$
to $G_n$, it follows that
$t_p(F,G_n)\sim s_p(F,G_n)$ as $n\to\infty$, so it makes no
difference whether we consider $s_p$ or $t_p$.
In general, this is
not true: for example, considering homomorphisms which map
all $t$ vertices on one side of $K_{t,t}$ into a single vertex,
we see that in any graph $G_n$ with $pn^2/2$ edges there are at least
$n(np)^t=(n^{2t}p^{t^2}) / (np^t)^{t-1}$
non-injective embeddings of $K_{t,t}$. If $np^t$ is bounded,
then this is comparable to (or larger than) the denominator
in the definition of $t_p(K_{t,t},G_n)$, and it follows that
$t_p(K_{t,t},G_n)-s_p(K_{t,t},G_n)$ is bounded away from zero.
Thus, for $t_p(K_{t,t},G_n)\sim s_p(K_{t,t},G_n)$ to hold with both quantities
bounded, we need $np^t\to\infty$. This condition holds for every $t$
only in the almost dense case $p=n^{-o(1)}$.

\subsection{Bounded and unbounded kernels}
The following simple observation illuminates the relationship between Conjectures~\ref{q1a} and~\ref{q1}.

\begin{lemma}\label{lbd}
Let $\ka:[0,1]^2\to [0,\infty)$ be a kernel, and $C\ge 0$ a constant.
Then we have $s(F,\ka)\le C^{e(F)}$ for every $F$ if and only
if $\ka\le C$ holds almost everywhere.
\end{lemma}
\begin{proof}
The result is trivial if $C=0$. Otherwise, rescaling, we may assume
that $C=1$. If $\ka\le 1$ almost everywhere, then $s(F,\ka)\le s(F,1)=1$
for every $F$. We may thus suppose that $\ka>1$ on a set of positive measure.
It follows that there is some $\eta>0$ such that $\ka>(1+\eta)^2$ on a set
$A$ of positive measure. Applying the Lebesgue Density Theorem to $A$,
there is some $\eps>0$ and some rectangle $R=[a,a+\eps]\times [b,b+\eps]\subset
[0,1]^2$ such that $\mu(A\cap R)\ge \mu(R)/(1+\eta)$. Thus, the average
value of $\ka$ on the set $R$ is at least $1+\eta$.
Let $\ka'$ be the kernel taking the value $1+\eta$ on $R$ and $0$ elsewhere.
Standard arguments from convexity show that, for each $t$,
\[
 s(K_{t,t},\ka) \ge s(K_{t,t},\ka') = \eps^{2t}(1+\eta)^{t^2}.
\]
Taking $t$ large enough, we find an $F=K_{t,t}$ for which $s(F,\ka)>1$.
\end{proof}

Lemma~\ref{lbd} shows that a kernel $\ka$ is bounded if and only if
the counts $s(F,\ka)$ grow at most exponentially in $e(F)$. It also shows
that, in Conjecture~\ref{q1a}, we need only consider kernels $\ka:[0,1]^2\to[0,C]$.

Let us say that a kernel has {\em finite moments} if $s(F,\ka)<\infty$ for all $F$.
There are unbounded kernels with finite moments:
the simplest way to construct such an example is
to consider the `rank 1' case, where $\ka(x,y)=f(x)f(y)$
for some $f:[0,1]\to[0,\infty)$. Indeed,
let $f$ be any function
from $[0,1]$ to $[0,\infty)$ with $\E(f^k)=\int_0^1 f(x)^k\dd x$ bounded
for every $k$; for example, let $f(x)=\log(1/x)$ for $x>0$.
Set $\ka(x,y)=f(x)f(y)$. 
If $F$ is a graph on $\{1,2,\ldots,k\}$ in which vertex $i$ has degree $d_i$,
then
\begin{multline*}
 s(F,\ka) = \int_{[0,1]^k} \prod_{ij\in E(F)} f(x_i)f(x_j) \prod_{i=1}^k\dd x_i \\
 = \int_{[0,1]^k} \prod_{i=1}^k f(x_i)^{d_i} \prod_{i=1}^k\dd x_i
 = \prod_{i=1}^k \E(f^{d_i}) < \infty.
\end{multline*}

The calculation above shows that
a rank one kernel $\ka(x,y)=f(x)f(y)$ has finite moments
if and only if $\norm{f}_p<\infty$ for every $p>1$, and hence
if and only if $\norm{\ka}_p<\infty$ for every $p>1$.  It is tempting to think that this
holds in general.
In one direction, for any kernel $\ka$ and any graph $F$ on $\{1,2,\ldots,k\}$, we may write
\[
 s(F,\ka) = \int_{[0,1]^k} \prod_{ij\in E(F)}\ka_{ij}(x_1,\ldots,x_d) \prod_{i=1}^k\dd x_i,
\]
where $\ka_{ij}(x_1,\ldots,x_d)=\ka(x_i,x_j)$.
Thus, by H\"older's inequality,
\[
 s(F,\ka) = \int \prod_{ij\in E(F)}\ka_{ij}
\le \prod_{ij\in E(F)} \norm{\ka_{ij}}_{e(F)} =
\prod_{ij\in E(F)} \norm{\ka}_{e(F)}.
\]
Hence, if $\norm{\ka}_p<\infty$ for every $p>1$, then $s(F,\ka)<\infty$ for every $F$.
The reverse implication does not hold, however, as shown by the following example.

\begin{example}
{\bf A kernel with finite moments but infinite $2$-norm.}
Let us define a sequence of independent random kernels $\ka_0,\ka_1,\ka_2,\ldots$,
as follows. For $r\ge 0$, let $\PP_r$ be the partition of $[0,1]$ into $2^{2^r}$
equal intervals, and let $\PP_r^2$ be the corresponding partition of $[0,1]^2$:
divide $[0,1]^2$ into $4^{2^r}$ squares in the obvious way,
and take as one part of $\PP_r^2$ the union of a square and its reflection
in the line $x=y$ (which may be the same square).
Our kernel $\ka_r$ will be constant on each element of $\PP_r^2$,
taking the value $2^r$ with probability $2^{-2r}$ and $0$ otherwise, with the values
on different parts independent.
Note that $\ka_0$ is simply the constant kernel with value~1.

Let $\ka(x,y)=\sum_{r=0}^\infty \ka_r(x,y)$. It is easy to see that with probability
1 the sum converges almost everywhere (for example, recalling
that $\mu$ denotes Lebesgue measure, use the fact that
$\E\mu\{\ka_r>0\}=2^{-2r}$ to deduce that, with probability 1,
$\mu\{\exists s>r: \ka_s>0\}$ tends to 0 as $r\to\infty$).
Also, for large $r$, $\norm{\ka_r}_2^2$ is concentrated around its mean of 
$(2^r)^22^{-2r}=1$. Hence, with probability 1 we have $\norm{\ka_r}_2^2\ge 0.99$
for infinitely many $r$. Using $(a+b)^2\ge a^2+b^2$ for $a,b\ge 0$, it follows
that $\norm{\ka}_2^2$ is infinite with probability 1; in particular, $\ka$
does not have all $p$-norms finite.

Turning to the finite moments property, let $F$ be any fixed graph, with $t$
vertices. Since $\ka\ge \ka_0=1$, we have $s(F,\ka)\le s(K_t,\ka)$, so we may
assume without loss of generality that $F=K_t$. Since
$\ka$ is random, $s(K_t,\ka)$ is a random variable. We may write its expectation as
\[
 \E_{\ka} \E_{\bx} \prod_{i<j} \ka(x_i,x_j) =  \E_{\bx} \E_{\ka} \prod_{i<j} \ka(x_i,x_j),
\]
where $\E_{\ka}$ denotes expectation over the random choice of $\ka$, and
$\E_{\bx}$ over the random choice of $(x_1,\ldots,x_t)$, a sequence of $t$ \iid\ uniform
elements of $[0,1]$.
Let us fix $\bx$ for the moment, assuming as we may that $x_i\ne x_j$ for $i\ne j$.
Let $\ell$ be the largest $r$ such that some pair $x_i$, $x_j$ lie in the same part of $\PP_r$,
so $0\le \ell<\infty$.
Let $\sigma=\sum_{r\le \ell}\ka_r$ and $\tau=\sum_{r>\ell}\ka_r$, so $\ka=\sigma+\tau$.
For $r>\ell$, the $\binom{t}{2}$ pairs $(x_i,x_j)$, $i<j$, all lie in different
parts of $\PP_r^2$, so the values of $\ka_r$ on these pairs are independent.
Since different $\ka_r$ are independent, it follows that the values of $\tau$
on the pairs are also independent.
Now $\norm{\sigma}_\infty \le \sum_{r=0}^\ell\norm{\ka_r}_\infty = 2^{\ell+1}-1$.
Thus,
\[
 \E_\ka \prod_{i<j} \ka(x_i,x_j) \le \E_\ka \prod_{i<j} (2^{\ell+1}+\tau(x_i,x_j))
 = \prod_{i<j} (2^{\ell+1}+\E_\ka \tau(x_i,x_j)).
\]
For any $x$ and $y$ we have $\E_\ka \ka_r(x,y)=2^r2^{-2r}=2^{-r}$, from which
it follows that $\E_\ka \tau(x,y) \le 2$, and hence, very crudely, that
\[
 \E_\ka \prod_{i<j} \ka(x_i,x_j) \le  \prod_{i<j} (2^{\ell+1}+2)  \le 2^{2\ell t^2}.
\]
It remains to take the expectation over $\bx$. Since $\PR(\ell=r)\le \binom{t}{2}2^{-2^r}$,
we find that
\[
 \E s(F,\ka) \le \sum_{r=0}^\infty \binom{t}{2}2^{-2^{r}} 2^{2r t^2} <\infty,
\]
noting that for any fixed $t$ the $2^{-2^r}$ term dominates. If follows that with probability
1 we have $s(F,\ka)<\infty$ for every $F$, giving a kernel with finite moments
but with $\norm{\ka}_2$ infinite. A simple modification, taking the probability
that $\ka_r$ takes the value $2^r$ on a given square to be
$2^{-(1+\eps)r}$ rather than $2^{-2r}$ gives, for each $\eps>0$, 
an example with $\norm{\ka}_{1+\eps}$ infinite.
\end{example}

\subsection{Non-uniform random graphs}\label{ss_nurg}
As in the dense case, there is a key connection between convergence
of the counts $s_p(F,G_n)$ and random graphs.
Given a kernel $\ka$, let $G_p(n,\ka)$ be the random graph on
$[n]$ obtained as follows: first choose $x_1,\ldots,x_n$ independently
and uniformly from $[0,1]$. Then, conditional on this choice, join
each pair $\{i,j\}$ of vertices independently, with probability
$\min\{p\ka(x_i,x_j),1\}$. If $p\ka$ is bounded
by $1$, then $G_p(n,\ka)$ is simply $G(n,p\ka$); we write the parameter $p$ as a subscript
to emphasize that it is part of the overall normalization: we think of
a sparse graph generated from the kernel $\ka$, rather than a `sparse
kernel' $p\ka$. If $p=1/n$, then $G_p(n,\ka)$ is a special
case of the general sparse inhomogeneous model of
Bollob\'as, Janson and Riordan~\cite{BJR}.

\begin{remark}\label{rconv}
In what follows, we shall consider many statements about the convergence
of various sequences of random graphs. As usual in the theory of random
graphs, the precise notion of convergence is not important: one thinks
of `a random graph' with certain asymptotic properties, although this
makes no formal sense. Formally, it is most natural to work throughout
with convergence in probability, but this would require us to consider
`in probability' versions of our various assumptions, for example
the (exponentially) bounded counts assumptions~\ref{AA} and~\ref{AB}.
In fact, it is easy to check that in all cases considered here,
the error probabilities decay fast enough to give almost sure convergence
for any coupling of the relevant probability spaces.
However, we shall not verify this explicitly, noting
that one can in any case ensure almost sure convergence by passing
to a suitable subsequence.
\end{remark}

\begin{lemma}\label{scconv-ad}
Let $p=p(n)=n^{-o(1)}$, and let $\ka$ be a kernel
with $s(F,\ka)<\infty$ for every $F$.
Then $s_p(F,G_p(n,\ka))\pto s(F,\ka)$ for each fixed graph $F$,
so $\de(G_p(n,\ka),\ka)\pto 0$.
In fact, the sequence $G_p(n,\ka)$ converges almost surely to $\ka$
in the metric $\de$.
\end{lemma}
\begin{proof}
It is very easy to check that, for every $F$,
$s_p(F,G_p(n,\ka))$ is concentrated around its mean $s(F,\ka)$: indeed,
the second moment of the number of copies of $F$ can be written
as a sum of terms $(1+o(1))n_{(|H|)}p^{e(H)}s_p(H,\ka)$, and the dominant
term is the unique one with the largest power of $n$, where $H$ is the
disjoint union of two copies of $F$. (The $1+o(1)$ correction
is only needed if $\ka$ is unbounded, and appears due to the
$\max\{1,\cdot\}$ in the edge probabilities.)
This proves the first part of the result.
Convergence in probability in $\de$ follows since
convergence in probability in a product topology is equivalent
to convergence in probability of each coordinate. For the final
statement, see Remark~\ref{rconv}.
\end{proof}

Lemma~\ref{scconv-ad} implies that if $\ka$ has finite moments,
then the sequence $G_n=G_p(n,\ka)$
has bounded subgraph counts (i.e., satisfies Assumption~\ref{AA})
with probability $1$. If $\ka$ is bounded, then $G_n$
has exponentially bounded subgraph counts with probability~$1$.

Using Lemma~\ref{scconv-ad}, it is easy to see
that we must allow unbounded kernels in Conjecture~\ref{q1}. Indeed,
set $\ka(x,y)=\log(1/x)\log(1/y)$ for $0<x,y\le 1$, say,
and let $p(n)=1/\log n$. Then the
random graphs $G_p(n,\ka)$ satisfy Assumption~\ref{AA} with probability 1, and
\[
 s_p(F,G_p(n,\ka)) \to s(F,\ka)< \infty
\]
holds with probability 1 for every $F$. Since $\ka$ is unbounded,
by Lemma~\ref{lbd} there is no $C$
with $s(F,\ka)\le C^{e(F)}$ for every $F$, so there is no bounded $\ka'$
with $s_p(F,G_p(n,\ka))\to s(F,\ka')$ for every $F$.

Note that if $p$ decreases too fast with $n$, then
$s_p(F,G_p(n,\ka))$ is no longer concentrated around its mean: for
example, this is the case if $\E\,\emb(F,G_p(n,k))$ does not tend to infinity.
This is the reason for the assumption $p=n^{-o(1)}$ in the various
conjectures and results above: otherwise, there will be some $F$ for
which the expected number of embeddings does not tend to infinity. Note also that, for
smaller $p$, when $s_p(F,\cdot)$ and $t_p(F,\cdot)$ are no longer
asymptotically equal, the former is the more natural parameter: for
a given $F$, the lower limit on $p$ below which the corresponding
parameter for $G_p(n,\ka)$ is no longer close to $s(F,\ka)$ is in
general much smaller for $s_p(F,\cdot)$ than for $t_p(F,\cdot)$. It
may well be, however, that the conjectures in this section (or
perhaps just their proofs) fail when the relevant parameters
$s_p(F,\cdot)$ and
$t_p(F,\cdot)$ are no longer asymptotically equal.

\subsection{Subgraph counts in the uniform case}

Using convexity, it is very easy to check that the only possible
kernel $\ka$ with $s(K_2,\ka)=s(C_4,\ka)=1$ is the uniform kernel,
with $\ka=1$ a.e. The following conjecture is thus a very special case
of Conjecture~\ref{q1}.

\begin{conjecture}\label{q2}
Let $p=p(n)=n^{-o(1)}$, and let $(G_n)$ be a sequence of graphs with
$|G_n|=n$, $e(G_n)=p\binom{n}{2}$, $s_p(C_4,G_n)\to 1$,
and $\sup_n s_p(F,G_n)<\infty$ for each $F$.
Then $s_p(F,G_n)\to 1$ for every $F$.
\end{conjecture}

Of course, there is a variant of Conjecture~\ref{q2} where we replace Assumption~\ref{AA} by Assumption~\ref{AB},
i.e., we demand that $\limsup_n s_p(F,G_n)\le C^{e(F)}$ for some $C<\infty$. In this uniform
context there is perhaps less reason to expect this to make a difference.

In the dense case, it is one of the basic results about quasi-random graphs
that $s_p(K_2,G_n)\to 1$ and $s_p(C_4,G_n)\to 1$ imply $s_p(F,G_n)\to 1$ for every $F$,
with no further assumptions; see Chung, Graham and Wilson~\cite{CGW89}. 
In the sparse case, 
this result extends easily to certain graphs $F$; here it turns out
to be simpler to work with $t_p(F,G_n)$ rather than $s_p(F,G_n)$.

\begin{lemma}\label{lcycles}
Let $p=p(n)$ with $pn^{1/2}\to\infty$, and let $(G_n)$ be a sequence
of graphs with $|G_n|=n$ such that
$t_p(K_2,G_n)\to 1$ and $t_p(C_4,G_n)\to 1$.
Then $t_p(C_k,G_n)\to 1$ for each $k\ge 5$.
\end{lemma}
\begin{proof}
Suppressing the dependence on $n$, let $A$ denote the adjacency matrix of
$G_n$, and let $\la_1\ge \la_2\ge\cdots \ge\la_n$ be the eigenvalues of $A$.
For $k\ge 3$ we have
\[
 \hom(C_k,G_n) = \sum_{v_1,v_2,\ldots,v_k\in V(G_n)}A_{v_1v_2}A_{v_2v_3}\cdots
 A_{v_kv_1} = \tr(A^k) = \sum_{i=1}^n \la_i^k,
\]
so
\begin{equation}\label{ck}
 t_p(C_k,G_n) = n^{-k}p^{-k}\sum_{i=1}^n \la_i^k =  \sum_{i=1}^n \mu_i^k,
\end{equation}
where $\mu_i=\la_i/(np)$ is the $i$th normalized eigenvalue of $G_n$.
In particular,
\begin{equation}\label{s4}
  \sum_i \mu_i^4\to 1.
\end{equation}
The maximum eigenvalue of the adjacency matrix of any graph
is at least the average degree, so
\[
 \mu_1 = (np)^{-1}\la_1 \ge (np)^{-1}(1+o(1)) (n^2p)/n = 1+o(1).
\]
From \eqref{s4} it follows that
$\mu_1\sim 1$ and that $\sum_{i\ge 2} \mu_i^4\to 0$.
Hence $\mu_2\le 1$ and $\mu_n\ge -1$ if $n$ is large enough, and then
\[
 \sum \mu_i^k =\mu_1^k + \sum_{i\ge 2}\mu_i^k
 \le \mu_1^k + \max\{\mu_2^{k-4},\mu_n^{k-4}\} \sum_{i\ge 2}\mu_i^4
 \le \mu_1^k + \sum_{i\ge 2}\mu_i^4
 =1+o(1).
\]
Using \eqref{ck} again, the result follows.
\end{proof}

Informally, when $pn^{1/2}\to\infty$, the parameters $s_p(C_k,G_n)$
and $t_p(C_k,G_n)$ are equivalent. More precisely, 
Lemma~\ref{lcycles} implies the analogous statement 
with all occurrences of $t_p$ replaced by $s_p$, but
this requires a little work to show.

The restriction on $p$ in Lemma~\ref{lcycles} was
not used in the proof. However,
if $G_n$ has average degree $\db$, then it contains
at least $n\binom{\db}{2}$ pairs of adjacent edges. 
Thus, writing $N_{i,j}$ for the number of common neighbours
of $i$ and $j$, the sum of $N_{i,j}$ over ordered pairs $i\ne j$
is at least $2n\binom{\db}{2}=n\db(\db-1)$.
Hence, the number of homomorphisms from $C_4$ to $G_n$
with a given pair of opposite vertices mapped to distinct vertices is
\[
 \sum_{i\ne j} N_{i,j}^2
  \ge \frac{1}{n(n-1)}\left(\sum N_{i,j}\right)^2 \ge \frac{n\db^2(\db-1)^2}{n-1}.
\]
The number of homomorphisms with a given pair of opposite vertices
mapped to the same vertex is simply the sum of the squares of the degrees
in $G_n$, which is at least $n\db^2$. Thus,
\begin{equation}\label{c4min}
 \hom(C_4,G_n) \ge \frac{n\db^2(\db-1)^2}{n-1} + n\db^2
\end{equation}
for any graph $G_n$ with $n$ vertices and average degree $\db$.
With $\db\sim pn\to \infty$, this gives $\hom(C_4,G_n)\ge (1+o(1))(n^4p^4+n^3p^2)$,
i.e., $t_p(C_4,G_n)\ge (1+o(1))(1+n^{-1}p^{-2})$.
Consequently, $t_p(C_4,G_n)\sim 1$ implies $pn^{1/2}\to\infty$.
When $pn^{1/2}\to\infty$, \eqref{c4min}
reduces to the well-known fact that, in this case, $e(G_n)\sim p\binom{n}{2}$ implies
that
\[
 t_p(C_4,G_n), \ s_p(C_4,G_n) \ge 1-o(1).
\]

In the dense case, Lemma~\ref{lcycles} extends to triangles. Indeed,
$\tr(A^2)$ counts the number of walks of length 2 in $G$, which is just
$2e(G)$. Thus
\[
 \sum\mu_i^2 =\frac{2e(G)}{n^2p^2}\sim p^{-1}.
\]
If $p$ is bounded away from zero then it follows that $\sum_{i\ge 2}\mu_i^2$
is bounded as $n\to\infty$. Since $\sum_{i\ge 2}\mu_i^4\to 0$, it follows
by the Cauchy--Schwarz inequality
that $\sum_{i\ge 2}\mu_i^3\to 0$, and hence that $s_p(C_3,G_n)\to 1$.

To obtain a result for triangles in the sparse case by this method,
one needs stronger assumptions. Defining $p$ by $e(G)=n^2p/2$,
if we assume that $t_p(C_4,G_n)=1+o(p)$, then arguing
as above we find that $\sum_{i\ge 2}\mu_i^4=o(p)$
and $\sum_{i\ge 2}\mu_i^4 \le p^{-1}$, so Cauchy--Schwarz does
give $\sum_{i\ge 2}\mu_i^3\to 0$. In general, many results
for quasi-random graphs extend to the sparse case
with similar modifications, where $o(1)$ error terms
are replaced by suitable functions of $p$; see, for example, the results
of Thomason~\cite{Tho,Tho2} on $(p,\alpha)$-jumbled graphs.
Our aim here is different; we wish to assume only convergence
in the relevant metric, making no assumption about the rate of convergence.

When $p\to 0$, the conditions of Lemma~\ref{lcycles}
do not guarantee the `right' number of triangles,
as our next two examples will show.

\begin{example}\label{vst}
{\bf Very sparse graphs with too few triangles.} 
Throughout this example we assume that $p_1(n)$ and
$p_2(n)$ are functions of $n$ satisfying
\begin{equation}\label{p1p2}
 p_2=(1-p_1^2)^{n-2}
\end{equation}
and $p_1,p_2=\Theta(\sqrt{\log n}/\sqrt{n})$. To be concrete,
we may take $p_2=\sqrt{\log n}/\sqrt{n}$, in which case
the corresponding $p_1$ satisfies $p_1\sim p_2/\sqrt{2}$.
Suppressing the dependence on $n$, let $G$ be the usual Erd\H
os--R\'enyi random graph $G=G(n,p_1)$, and let $H$ be the graph on
the same vertex set $[n]$ in which vertices $i$ and $j$ are joined
if and only if they do not have a common neighbour in $G$. From \eqref{p1p2},
each edge of $H$ is present with probability $p_2$; note
that the edges of $H$ are {\em not} present independently of one another.
For any set $E$ of $r=O(1)$ possible edges of
$H$, the edges of $E$ are all present if and only if no vertex of
$G$ is joined to both ends of some edge in $E$. Considering each
vertex of $G$ separately, we see that the probability of this event
is
\[
 (1-rp_1^2 +O(p_1^3))^{n+O(1)} =e^{-rp_1^2n+O(np_1^3)} \sim \left((1-p_1^2)^{n-2}\right)^r
 =p_2^r,
\]
where the $O(1)$ correction in the first exponent is to account for
vertices that are endpoints of one or more edges in $E$. In other
words, the probability that a bounded number of edges is present in
$H$ is asymptotically the corresponding probability for $G(n,p_2)$.

For $E_1, E_2\subset E(K_n)$, the event $E_2\subset E(H)$ is a down-set
in terms of $G$ (it says that certain pairs of edges of $G$ are not present),
so $E_2\subset E(H)$ and $E_1\subset E(G)$ are negatively correlated.
Hence, if $|E_2|=O(1)$, we have
\begin{equation}\label{e1e2}
 \PR\bb{\{E_1\subset E(G)\}\cap \{E_2\subset E(H)\}} \le (1+o(1))p_1^{|E_1|}p_2^{|E_2|}.
\end{equation}
Considering all ways of splitting a set $E$, it follows that
$\PR(E\subset G\cup H)\le (1+o(1))(p_1+p_2)^{|E|}$, and hence that
\begin{equation}\label{fgh}
 \E(s_p(F,G\cup H))\le 1+o(1)
\end{equation}
for any fixed graph $F$, where $p=p_1+p_2$.

Since $G$ and $H$ overlap in very few edges, and the numbers of edges
of $G$ and of $H$ are concentrated, we have $s_p(K_2,G\cup H)\to 1$ almost
surely. It
follows that $s_p(C_4,G\cup H)\ge 1-o(1)$ almost surely. Hence,
from \eqref{fgh},
$
 s_p(C_4,G\cup H) \pto 1,
$
and it is not hard to deduce that $t_p(C_4,G\cup H)\pto 1$.

On the other hand, there are by
definition no triangles with two edges in $G$ and one in $H$. Hence,
from \eqref{e1e2}, the expectation of $\emb(K_3,G\cup H)$ is at most
\[
 (1+o(1))n^3 (p_1^3+ 0 + 3p_1p_2^2+p_2^3),
\]
so $\E(s_p(K_3,G\cup H)) \le (p^3-3p_1^2p_2)/p^3+o(1)$.
Since $p_1,p_2$ and $p$ are all of the same order, this final fraction
is strictly less than 1, and our construction
gives almost surely
a sequence $G_n=G\cup H$ with $s_p(K_2,G_n)\to 1$, $s_p(C_4,G_n)\to 1$
but $s_p(C_3,G_n)\not\to 1$.
Since $\emb(C_3,G)=\hom(C_3,G)$ for any $G$,
we have $t_p(C_3,G_n)\sim s_p(C_3,G_n)\not\to 1$.
Choosing $p_1$ and $p_2$
satisfying \eqref{p1p2} so that $p_2\sim p_1/2$, we may
achieve $s_p(C_3,G_n)\to 5/9$. Alternatively,
choosing $p_1$ and $p_2$ suitably, we may
find a sequence with $s_p(C_3,G_n)\not\to 1$
for any $p=p(n)$ satisfying $pn^{1/2}\to\infty$
and $p=O(\sqrt{\log n}/\sqrt{n})$.
\end{example}

\begin{example}
{\bf Very sparse graphs with no triangles.}\label{noga}
In the context of finding explicit constructions
giving lower bounds on Ramsey numbers,
Alon~\cite{Alon} constructed
a sequence of graphs $G_n$ defined only for certain $n$, with the following
properties, where $d=d(n)\sim n^{2/3}/4$:
the graph $G_n$ is a $d$-regular Cayley graph, it is triangle
free and (which is irrelevant here) the largest
independent set has size $O(n^{2/3})$.
In proving the last property, Alon shows that all eigenvalues
other than $\la_1=d$ are uniformly bounded by $O(n^{1/3})$.
Setting $p=d/n$, so $t_p(K_2,G_n)=1$,
and writing $\mu_i$ for $\la_i/(np)$, as
in the proof of Lemma~\ref{lcycles}, one thus
has $\mu_1=1$ and $\mu_i=O(n^{-1/3})$ for $i\ne 2$,
so from \eqref{ck} it follows that $t_p(C_4,G_n)=1+O(n^{-1/3})=1+o(1)$.
This gives another example of a graph with almost the minimal
number of $C_4$s but too few (in this case no) triangles.
\end{example}

\begin{example}\label{tdense}
{\bf Denser graphs with too few triangles.}
Let $n=mk$ where
$m\to\infty$, and let $p=\sqrt{\log m}/{\sqrt m}$. Example~\ref{vst}
gives us a graph $G'$ of order $m$ with $t_p(K_2,G'),
t_p(C_4,G')\sim 1$ and $t_p(K_3,G')\le 0.9$, say, for all large
enough $m$. Let $G$ be the blow-up of $G'$ obtained by replacing
each vertex by $k$ vertices. Since $t_p(F,\cdot)$ is unchanged by
blow-ups, we have $t_p(K_2,G), t_p(C_4,G)\sim 1$ but $t_p(K_3,G)\le
0.9$, from which $s_p(K_2,G), s_p(C_4,G)\sim 1$ and (for $n$ large)
$s_p(K_3,G)\le 0.91$ follow immediately.

Although $p$ has not changed, the number of vertices has. Seen as a function
of $n$, we may choose $p=\sqrt{\log m}/{\sqrt m}$ for any $m$ dividing
$n$ with $m\to\infty$. Exact divisibility is not essential. Either by
using this fact, or by restricting to a subsequence, we see that
any given function $p(n)$ can be realized up to a factor of $(1+o(1))$,
provided $p(n)/(\sqrt{\log n}/\sqrt{n})\to \infty$ and $p(n)=o(1)$.
Hence, we may construct graphs with the right number of $C_4$s but
too few triangles for any such function $p(n)$.
\end{example}

At first sight Example~\ref{tdense} seems to contradict Conjecture~\ref{q2}, but
this is not the case. Indeed, for the graph $G'$ that we blow up,
\eqref{fgh} tells us that we do not have too many embeddings of any
fixed $F$. However, while $s_p\sim t_p$ for $p=n^{-o(1)}$, the final
$p$ we consider, and while blowing up preserves $t_p$, $G'$ is a
very sparse graph: although it has the same absolute density as the
final graph $G$, this density is much smaller than $|G'|^{-o(1)}$,
since $G'$ has many fewer vertices than $G$. It follows that the
homomorphism counts in $G'$ are {\em not} well behaved. In
particular, $G'$ contains around $m^4p^3$ non-injective
homomorphisms from $K_{2,3}$, which turns out to be much larger than
the number $m^5p^6$ of embeddings. It follows that $G$ contains too
many homomorphisms from, and thus embeddings of, $K_{2,3}$, i.e.,
that $s_p(K_{2,3},G)\to\infty$.

\begin{remark}\label{r_large}
Let us note in passing that the blowing-up argument above
shows that replacing the assumption $p=n^{-o(1)}$
in Conjecture~\ref{q2} (or Conjecture~\ref{q1a}) with a
stronger assumption such as $p(n)\ge 1/\log\log\log n$, say,
makes no difference. Indeed, if the conjecture fails,
and $(G_n)$ is a counterexample, then blowing up $G_n$
as above by replacing each vertex by $f(n)$ vertices
for some rapidly growing $f(n)$ gives a counterexample
for a different density function, where now the density
goes to zero extremely slowly as a function of the number
of vertices.
\end{remark}

One possible approach to producing a counterexample to Conjecture~\ref{q2} would
be to consider {\em circulant graphs}, i.e., graphs on the vertex
set $[n]$ in which whether or not $ij$ is an edge depends only on
$i-j$ modulo n. There is one circulant graph for each subset $A$ of
the integers modulo $n$ satisfying $0\notin A$ and $a\in A$ if and
only if $-a\in A$. All our conjectures thus imply corresponding
conjectures for subsets of $\Z_n$, the integers modulo $n$, in which
the symmetry condition is not likely to be relevant. Most subgraph
counts in the graph have a rather unnatural interpretation in terms
of the corresponding sets; the exception is cycles, where the number
of $k$-cycles in $G$ corresponds to ($n$ times the) number of
$k$-tuples in $A^k$ summing to 0. There is a result corresponding to
Lemma~\ref{lcycles} for subsets of $\Z_n$, proved in the same way
but using Fourier coefficients instead of eigenvalues.
Unfortunately, Examples~\ref{vst} and \ref{tdense} also carry over
to the set context, in a fairly straightforward way: instead of
blowing up the graph, we replace each element of $A$ by a block of
consecutive integers. This shows that any result of the kind we want
about subsets of $\Z_n$ must involve conditions other than
constraints on the number of tuples summing to 0.

In the sparse case, even when $p=n^{-o(1)}$,
it is not true that $s_p(K_2,G_n)\to 1$ and $s_p(C_4,G_n)\to 1$
together imply $s_p(F,G_n)\to 1$ for every $F$. We have just seen
one example, with $F=C_3$. There are also much simpler examples.

\begin{example}\label{clique}
{\bf Adding a dense part.}
Let $p=1/\log n$, say, and let $m=m(n)=n/(\log n)^c$ where $c>0$ is constant.
(We ignore rounding to integers.) Let $G'$ be any graph on $n-m$ vertices,
and let $G$ be the disjoint union of $G'$ and a complete graph on $m$ vertices.
Since $K_m$ contains roughly $m^{|F|}$ embeddings of any fixed $F$, we have
\[
 s_p(F,G) \sim s_p(F,G') + \frac{m^{|F|}}{p^{e(F)}n^{|F|}} = s_p(F,G') + (\log n)^{e(F)-c|F|}.
\]
Taking $G'=G(n-m,p)$ and $c=3/2$, say, we have
$s_p(K_2,G)\sim s_p(K_2,G')\sim 1$, 
$s_p(C_4,G)\sim s_p(C_4,G')\sim 1$, but $s_p(K_4,G)\sim 1+1=2$.
Note that $s_p(K_5,G)\to\infty$, so the assumptions of Conjecture~\ref{q2} are
not satisfied.
\end{example}

The above example is rather artificial: there are too many copies of $K_4$
(and of $K_5$), but these sit on a small number of vertices. However, the same effect
can be achieved by taking the union on the same vertex set of $G(n,p)$
and a disjoint union of $n/m$ copies of $K_m$. Also, we can use complete
bipartite graphs instead of complete graphs.

\begin{example}\label{blowup}
{\bf A blown-up random graph.}
Let $n=mk$, where $k=k(n)$ and $m=m(n)$ both tend to infinity.
(As usual, we ignore divisibility issues, or consider a sequence $n_i\to\infty$.)
Let $G_1$ be the random graph $G(m,p)$, where $p=p(n)$, and
let $G=G_1^{(k)}$ be formed by replacing each vertex of $G$ by an independent
set of size $k$, and each edge by a $k$-by-$k$ complete bipartite graph.
The number of edges of $G$ is $k^2e(G_1)$, which is asymptotically
$k^2m^2p/2=n^2p/2$, so $s_p(K_2,G)\to 1$ in probability and almost surely.
Similarly, for any fixed graph $F$, each embedding of $F$ into $G_1$
gives rise to $k^{|F|}$ embeddings into $G$; the expected
number of embeddings arising in this way is essentially the
expected number in $G(n,p)$, so whenever this expectation tends
to infinity, such embeddings will contribute $1+o(1)$ to $s_p(F,G)$.

There are other embeddings of $F$ into $G$, however, where some distinct
vertices of $F$ are mapped to the same vertex in $G_1$. For $C_4$, we
have roughly $m^2pk^4$ such embeddings within our complete bipartite graphs,
and roughly $2m^3p^2k^4$ from embeddings involving three vertices
of $G_1$. Provided $mp^2\to\infty$, we still have $s_p(C_4,G)\to 1$.

Fix an integer $t\ge 3$, and suppose now that $m=m(n)$ and $p=p(n)$ are chosen
so that $m$ and $k=n/m\to \infty$, and $mp^t\to c$ for some constant $0<c<\infty$;
for example, set $p=1/\log n$, $m=c(\log n)^t$ and $k=c^{-1}n/(\log n)^t$.
Note that $mp^2\to\infty$.
Then we have roughly $m^{2+t}k^{2+t}p^{2t}$ embeddings of $K_{2,t}$ into
$G$ coming from embeddings into $G_1$. But we also have
roughly $m^{1+t}k^{2+t}p^t$ embeddings into $G$ coming from maps from
$K_{2,t}$ into $G_1$ sending the two vertices on one side to the same
vertex. It is easy to check that these two are the dominant terms (mapping
the two vertices on one side to the same place we gain $t$ factors
of $1/p$ and lose one factor of $m$; any other identifications gain
fewer factors of $1/p$ per factor of $m$ lost), and it follows that
$s_p(K_{2,t},G) \to 1+1/c$.

Taking a `typical' sequence of random graphs constructed as above
gives an example with $s_p(K_2,G_n)\to 1$, $s_p(C_4,G_n)\to 1$ (and
indeed $s_p(K_{2,t'},G_n)\to 1$ for $2\le t'<t$), but
$s_p(K_{2,t},G_n)\to 1+1/c\ne 1$.
Once again, the assumptions of Conjecture~\ref{q2} are not satisfied,
this time because $s_p(K_{2,t+1},G_n)\to\infty$.
\end{example}

We have seen from the examples above that if $p(n)\to 0$, then
$s_p(K_2,G_n)$ and $s_p(C_4,G_n)\to 1$ do not themselves imply
that $s_p(F,G_n)\to 1$ for every $F$.
However, attempted counterexamples to Conjecture~\ref{q2} seem to be
doomed to failure by the the additional
assumption that $s_p(F,G_n)$ is bounded for every $F$.
In the next section we shall see that
we can make some progress towards proving Conjecture~\ref{q2}.

\subsection{Partial results in the almost dense, uniform case}
In the examples in the previous subsection,
each vertex is in about the same number of copies
of any fixed graph $F$, but there are relatively few ($o(n^2)$) pairs
that are in too many copies of $K_{2,t}$, for example.
It is easy to see that, under the assumptions of Conjecture~\ref{q2}, this cannot happen.
In fact, we can make a much more general statement. For this it is convenient
to work with homomorphism counts and $t_p(F,G_n)$ rather than embeddings and $s_p(F,G_n)$.
As noted earlier, in the almost dense case that we consider
in this subsection, i.e., when $p=n^{-o(1)}$, the quantities
$t_p(F,G_n)$ and $s_p(F,G_n)$ differ by $o(1)$.

Let $F$ be a fixed graph, and $F'$ a subgraph of $F$. Without loss
of generality, suppose that $V(F')=[\ell]\subset [k]=V(F)$.
Then any homomorphism $\phi_F:F\to G_n$ restricts
to a homomorphism $\phi_{F'}:F'\to G_n$. With $e(G_n)\sim p\binom{n}{2}$, we expect a typical
$\phi_{F'}$ to have around $n^{k-\ell}p^{e(F)-e(F')}$ extensions.
For each $n$, let us define a random variable $Z_n(F',F)$ as follows:
let $\phi_{F'}$ be chosen uniformly at random from among all homomorphisms from $F'$ into $G$
(if there are any), and let $Z_n(F',F)$ be the number of extensions
of $\phi_{F'}$ divided by $n^{k-\ell}p^{e(F)-e(F')}$.
(The reader may well prefer to picture copies of $F'$ and $F$ in $G_n$ rather
than homomorphisms. In fact, it is better to picture embeddings, i.e., 
labelled copies. There are essentially the same number of these as of homomorphisms.)
Since $\hom(F,G_n)$ is the sum over $\phi_{F'}$ of the number of extensions,
we have 
\[
 \hom(F,G_n)=\hom(F',G_n) \E(Z_n(F',F)) n^{k-\ell}p^{e(F)-e(F')}, 
\]
and hence
\[
 t_p(F,G_n) = t_p(F',G_n) \E(Z_n(F',F)).
\]
For $r\ge 2$, let $rF/F'$ denote the graph formed by the union of $r$ copies
of $F$ which all meet in the same subgraph $F'$, so $rF/F'$
has $|F'|+r(|F|-|F'|)$ vertices and $e(F')+r(e(F)-e(F'))$ edges.
A homomorphism from $rF/F'$ to $G_n$ consists of a homomorphism $\phi$ from
$F'$ to $G_n$ together with $r$ extensions of $\phi$ to homomorphisms
from $F$ to $G$, which may or may not be distinct. (They almost always will be.)
Since we have normalized by the right powers of $n$ and $p$, it follows that
\begin{equation}\label{tmom}
 t_p(rF/F',G_n) = t_p(F',G_n) \E(Z_n(F',F)^r).
\end{equation}

Let $\mu_F=\mu_F(n)=n^{|F|}p^{e(F)}$, which is asymptotically equal to
the expected number of homomorphisms
from $F$ into $G(n,p)$. Then, under the assumptions of any of
Conjectures~\ref{q1a}, \ref{q1} and \ref{q2}, it is easy to see
that for $F'\subset F$, any $o(\mu_{F'})$ copies of $F'$ meet $o(\mu_F)$ copies
of $F$. (Here `copies' may be subgraphs of $G_n$, embeddings,
or homomorphisms; it makes no difference.) Otherwise $t_p(2F/F',G_n)$
would not remain bounded. This rules out any construction of a potential
counterexample similar to those above; it also shows that if $t_p(K_2,G_n)\to 0$
and Assumption~\ref{AA} holds (i.e., $(G_n)$ has bounded subgraph counts),
then $t_p(F,G_n)\to 0$ for every $F$.

Conjecture~\ref{q2} states that infinitely many conclusions (one for each $F$) hold
under the same assumptions. We have already proved some of these conclusions,
with $F=C_k$, $k\ge 5$. Our next aim is to prove a corresponding result
for a much wider class of graphs.  In doing so, the following observation
will be useful.

\begin{lemma}\label{lmom}
Let $X_n\ge 0$ be a sequence of random variables with $\sup_n \E(X_n^k)<\infty$
for every $k\ge 1$. Then $\E(X_n^k)\to 1$ for every $k$ if and only if $X_n\pto 1$.
\end{lemma}
\begin{proof}
For the forward implication we have $\E(X_n)\to 1$ and $\E(X_n^2)\to 1$; applying Chebyshev's
inequality it follows that $X_n\pto 1$. The reverse implication is not much harder.
Suppose that $X_n\pto 1$, but that $\E(X_n^k)\not\to 1$ for some $k$.
For any $M$, the variables $X_n^k\IN{X_n\le M}$ are uniformly
bounded and converge in probability to~1,
so
$\E(X_n^k\IN{X_n\le M})\to 1$.
It follows that there is some $M(n)\to\infty$ such that
$\E(X_n^k\IN{X_n\le M(n)})\to 1$.
But then
$\E(X_n^k\IN{X_n> M(n)})\not\to 0$, so
\[
 \E(X_n^{k+1}) \ge \E(X_n^{k+1}\IN{X_n> M(n)})
 \ge M(n)\E(X_n^k\IN{X_n> M(n)})
\]
is unbounded, contradicting our assumptions.
\end{proof}

\begin{corollary}\label{cmom}
Under the assumptions of Conjecture~\ref{q2}, if $F'$ and $F$ are fixed
graphs with $F'\subset F$ and $t_p(F',G_n)\to 1$,
then $Z_n(F',F)\pto 1$ if and only if $t_p(rF/F',G_n)\to 1$ for every $r\ge 1$.
\end{corollary}
\begin{proof}
Apply Lemma~\ref{lmom} to the random variable
$Z_n(F',F)$, using \eqref{tmom} to evaluate its moments.
\end{proof}

We shall say that the distribution of $F$ is {\em flat over} that of $F'$ in $G_n$,
or simply that $F$ is {\em flat over} $F'$, if $Z_n(F',F)\pto 1$.

\begin{lemma}\label{lKst}
Under the assumptions of Conjecture~\ref{q2} we have $s_p(K_{s,t},G_n)\to 1$ for all $s,t\ge 1$.
Moreover, $K_{1,s}$ is flat over $E_s$, where $E_s$ is the empty subgraph
of $K_{1,s}$ induced by the vertices in the second part.
\end{lemma}
\begin{proof}
Let $d_1,\dots,d_n$ denote the degrees of the vertices of $G_n$, and $\db$ the average
degree. Fix $s\ge 1$. By convexity, we have
\[
 \hom(K_{1,s},G_n) =\sum_{i=1}^n d_i^s \ge n{\db}^s,
\]
which we can rewrite as $t_p(K_{1,s},G_n)\ge t_p(K_2,G_n)^s$.
Since $t_p(K_2,G_n)\to 1$ by assumption, this gives
\begin{equation}\label{lotss}
 t_p(K_{1,s},G_n) \ge 1+o(1).
\end{equation}

Specializing to $s=2$ for the moment, let $Z_n=Z_n(E_2,K_{1,2})$ be the random
variable describing the distribution of the number of common neighbours
of a random pair of vertices of $G_n$. For any empty graph $E_k$
we have $t_p(E_k,G_n)=1$. Hence,
from \eqref{tmom} and \eqref{lotss},
\[
 \E(Z_n) = t_p(K_{1,2},G_n) \ge 1+o(1).
\]
On the other hand, since $tK_{1,2}/E_2=K_{2,t}$,
\[
 \E(Z_n^2) = t_p(K_{2,2},G_n) = t_p(C_4,G_n) \to 1.
\]
Since $\E(Z_n^2)\ge \E(Z_n)^2$, it follows that $\E(Z_n)\to 1$ and
(by Lemma~\ref{lmom}) that $Z_n\pto 1$. In other words,
$K_{1,2}$ is flat over pairs of vertices.
By Corollary~\ref{cmom} it then follows that $t_p(K_{2,t},G_n)\to 1$ for every $t$.

Returning to general $s$, let $W_n=Z_n(E_s,K_{1,s})$.
From \eqref{lotss} we have $\E(W_n)=t_p(K_{1,s},G_n)\ge 1+o(1)$.
But we have just shown that $\E(W_n^2)=t_p(K_{2,s},G_n)\to 1$,
so $W_n\pto 1$, i.e., $K_{1,s}$ is flat over $E_s$.
Applying Corollary~\ref{cmom} again we thus have $t_p(K_{s,t},G_n)\to 1$ for every $t$,
as required.
\end{proof}

\begin{theorem}\label{g4}
Let $F$ be any fixed graph with girth at least $4$,
and let $F'\ne F$ be any induced subgraph of $F$. Under the assumptions of Conjecture~\ref{q2},
$F$ is flat over $F'$. Furthermore, $s_p(F,G_n), t_p(F,G_n)\to 1$
as $n\to\infty$.
\end{theorem}
\begin{proof}
Note first that the definition of $Z_n(F',F)$ makes perfect sense when $F'$
is the empty `graph' with no vertices; there is one homomorphism from $F'$ to
$G_n$, and $Z_n(F',F)$ is constant and takes the value $t_p(F,G_n)$.
Hence, $F$ is flat over the empty subgraph means exactly that $t_p(F,G_n)\to 1$.
Since $p=n^{-o(1)}$, we have $s_p(F,G_n)\sim t_p(F,G_n)$, so it suffices
to prove the first statement.

We prove the first statement
of the theorem by induction on $|F|$. If $|F|=1$, there is nothing to prove.
Suppose then that $F$ and $F'$ are given, with $|F|\ge 2$, and that the result
holds for all smaller $F$.

Suppose first that $F'=F-v$ for some vertex $v$ of $F$.
Let $E_s$ denote the subgraph of $F'$ induced by the neighbours of $v$,
noting that $E_s$ has no edges, as $F$ is triangle free.
Set $X_n = Z_n(E_s,F')$ and $Y_n= Z_n(E_s,K_{1,s})$.
Note that these random variables are defined on the same probability space: the 
elements of this space are simply $s$-tuples of vertices of $G_n$.
If $F'=E_s$, then $F'$ is trivially flat over $E_s$. If not,
then $F'$ is flat over $E_s$ by the induction hypothesis.
Hence, in either case, $\E(X_n^k)\to 1$ for every $k$.
By the last part of Lemma~\ref{lKst}, $K_{1,s}$ is flat over $E_s$,
so $\E(Y_n^k)\to 1$ for every $k$. It follows that
$\E((X_n-1)^k)\to 0$ and $\E((Y_n-1)^k)\to 0$ for all $k\ge 1$.
Hence, by the Cauchy--Schwarz inequality,
\[
 \E((X_n-1)^k(Y_n-1)^\ell) \le\sqrt{\E((X_n-1)^{2k})\E((Y_n-1)^{2\ell})} \to 0
\]
for all $k,\ell\ge 0$ with $k+\ell>0$. Writing
$\E(X_n^kY_n^\ell)=\E( (X_n-1 +1)^k(Y_n-1+1)^\ell)$ as $1$ plus a sum
of terms $\E((X_n-1)^{k'}(Y_n-1)^{\ell'})$, $k',\ell'\ge 0$, $k'+\ell'>0$,
it follows that $\E(X_n^kY_n^\ell)\to 1$ for any $k,\ell\ge 0$.

Any homomorphism $\phi_{F'}$ from $F'$ into $G_n$ is the extension
of a unique homomorphism $\phi_{E_s}$ from $E_s$ into $G_n$.
Furthermore, to extend $\phi_{F'}$ to $F$ we must choose for the image of $v$
a common neighbour of the vertices in the image of $\phi_{E_s}$.
Hence, the value of $Z_n=Z_n(F',F)$ on $\phi_{F'}$ is simply the value
of $Y_n$ on $\phi_{E_s}$.
Choosing $\phi_{F'}$ uniformly at random, to obtain the correct distribution
for $Z_n$, the probability of obtaining a particular restriction $\phi_{E_s}$
is proportional to the number of extensions of $\phi_{E_s}$ to $F'$, i.e., to $X_n$.
Thus the distribution of $Z_n$ is that of $Y_n$ `size biased' by $X_n$.
In particular,
\[
 \E(Z_n^k) = \frac{\E(X_nY_n^k)}{\E(X_n)} \sim 1/1=1.
\]
Taking $k=1,2$, it follows that $Z_n\pto 1$, i.e., that $F$ is flat over
$F'$, as required.

It remains to handle the case $|F|-|F'|\ge 2$. In this case, we can find an induced
subgraph $F''=F-v$ of $F$ with $F'\subset F''\subset F$.
Note that $t_p(F',G_n), t_p(F'',G_n)\sim 1$ by induction,
that $F''$ is flat over $F'$ by induction, and that $F$ is flat over $F''$
by the case treated above. In particular, we certainly have
\[
 t_p(F,G_n) = t_p(F'',G_n)\E(Z_n(F'',F)) \sim 1.
\]
Fix $\eps>0$. Let us call a copy of $F''$ (more precisely, a homomorphism from
$F''$ into $G_n$) {\em bad} if it has fewer than $(1-\eps)\mu_{F}/\mu_{F''}$
extensions to copies of $F$.
Since $F$ is flat over $F''$ and $t_p(F'',G_n)\sim 1$, there
are fewer than $\eps^2\mu_{F''}$ bad copies of $F''$ if $n$ is large enough.
Since each copy of $F''$ extends a unique copy of $F'$,
it follows that at most $\eps\mu_{F'}$ copies of $F'$ have
more than $\eps\mu_{F''}/\mu_{F'}$ extensions to bad copies
of $F''$.

Let $\BB_1$ denote the set of copies of $F'$ that have more than 
$\eps\mu_{F''}/\mu_{F'}$ extensions to bad copies
of $F''$, so $|\BB_1|\le \eps\mu_{F'}$ if $n$ is large.
Let $\BB_2$ denote the set of copies of $F'$ that have fewer than
$(1-\eps)\mu_{F''}/\mu_{F'}$ extensions to copies of $F''$.
Since $F''$ is flat over $F'$, we have $|\BB_2|\le \eps\mu_{F'}$
if $n$ is large enough, which we assume from now on.
If $\phi$ is a copy of $F'$ not in $\BB_1\cup \BB_2$, then
$\phi$ has at least $(1-2\eps)\mu_{F''}/\mu_{F'}$ extensions
to good copies of $F''$, which in turn have at least $(1-\eps)\mu_F/\mu_{F''}$
extensions to copies of $F$,
so the value of $Z_n(F',F)$ on $\phi$ is at least $(1-2\eps)(1-\eps)$.
Since there are $(1+o(1))\mu_{F'}$ copies of $F'$ in total,
the proportion of these copies in $\BB_1\cup\BB_2$ is at most $\eps+o(1)$.
Since $\eps>0$ was arbitrary, it follows
that the negative part of $Z_n(F',F)-1$ tends to zero in probability.
Since $\E(Z_n(F',F))=t_p(F,G_n)/t_p(F',G_n)\to 1$,
it follows that $Z_n(F',F)\pto 1$, i.e., that $F$ is flat over $F'$. 
\end{proof}

The reader may find many of the arguments above familiar from the dense case; for example,
the proof for $K_{2,t}$ is an absolutely standard convexity argument.
The key point is that many arguments for the dense case do not carry over.
In particular, we have shown that almost all, i.e., all but $o(n^2)$, pairs
of vertices have about the right number of common neighbours. 
In the dense case, it follows immediately that almost all (all but $o(pn^2)=o(n^2)$)
edges are in the right number of triangles, and hence that $t_p(K_3,G_n)\to 1$.
Similarly,
the proof above shows that any $F$ is flat over all its subgraphs in the dense case,
without restriction to girth at least $4$.
In the sparse case, there are only $o(n^2)$ edges, and there seems to be no simple
way to rule out the possibility that a large fraction, or even {\em all}, of
the pairs of vertices corresponding to edges fall in the $o(n^2)$ set with too few
common neighbours. Nevertheless, we conjecture that this cannot happen.
The simplest graph for which we cannot prove the conclusion of Conjecture~\ref{q2} is the triangle.

\begin{conjecture}\label{ctri}
Under the conditions of Conjecture~\ref{q2} we have $s_p(K_3,G_n)\to 1$.
\end{conjecture}

In fact, we do not even have a proof that $G_n$ must contain at least {\em one}
triangle for $n$ large enough!

\subsection{Extensions to lower densities.}\label{ss_e}
Let us return to the study of general subgraphs $F$, rather than simply triangles.
If true, the various conjectures above may extend to smaller values of $p$, but one must be
careful. Firstly, $s_p$ and $t_p$ no longer coincide, as noted above.
One should work with $s_p$, because these quantities behave in the right
way for $G_p(n,\ka)$, while $t_p$ does not.
A simple modification of the proof of Lemma~\ref{lKst}, considering
the distribution of the number of common neighbours
of a set of $s$ {\em distinct} vertices, shows that if $np^s\to\infty$,
then $s_p(K_2,G_n)\to 1$, $s_p(C_4,G_n)\to 1$ and $s_p(K_{s,t+1},G_n)$ bounded
together imply $s_p(K_{s,t},G_n)\to 1$.
Taking $p=n^{-\alpha}$, with $0<\alpha<1/2$ constant, there is no
corresponding result for $t_p$, even with $s=2$.
Indeed, if $t_p(K_2,G_n)=1$, then there are at least
$n^{t+1}p^t$ homomorphisms from $K_{2,t}$ into $G_n$ mapping the two vertices
in the smaller class to the same vertex. It follows that
$t_p(K_{2,t},G_n)$ will be unbounded for any $t>1/\alpha$.

Secondly, even working with $s_p$ rather than $t_p$,
we cannot in general hope to conclude in the analogue
of Conjecture~\ref{q1} that $s_p(F,G_n)\to s(F,\ka)$
for {\em all} fixed graphs $F$. For example, set $p=n^{-1/2}$
and consider the {\em polarity graphs} $G_n$ of Erd\H os and R\'enyi~\cite{ERpolarity},
defined (for suitable $n$)
by taking as vertices the points of the projective
plane over $GF(q)$, $q$ a prime power, and joining $x=(x_0,x_1,x_2)$
and $y=(y_0,y_1,y_2)$ 
if and only if $x_0y_0+x_1y_1+x_2y_2=0$ in $GF(q)$.
These graphs satisfy $e(G_n)\sim n^{3/2}/2=pn^2/2$ but contain no $C_4$s, 
and thus satisfy $s_p(K_2,G_n)\to 1$ and $s_p(C_4,G_n)= 0$.
Since $s(C_4,\ka)\ge s(K_2,\ka)^4$ for any $\ka$,
we cannot have $s_p(F,G_n)\to s(F,\ka)$ for $F=K_2$ and 
for $F=C_4$ in this case. More generally, whenever
$pn^{1/2}\not\to\infty$, then there are graphs $G_n$ with
$pn^2$ edges but too few $C_4$s, so we should only consider
the counts $s_p(C_4,G_n)$ if $pn^{1/2}\to\infty$.
This problem is not unique to $C_4$, so it seems that to extend
our conjectures for $p=n^{-o(1)}$ to sparser graphs,
we should modify them to refer only to a certain set of `admissible' subgraphs $F$,
depending on the function $p=p(n)$.

In fact, we should only consider subgraphs $F$ for which
the expected number $\mu_F\sim n^{|F|}p^{e(F)}$ of embeddings of $F$ into $G(n,p)$
is much larger than the number $(1+o(1))n^2p/2$ of edges,
at least if $pn^{1/2}\to\infty$. To see this, first
suppose that $n^{|F|}p^{e(F)}\sim An^2p$, for some constant $0<A<\infty$.
Form a graph $G'$ from $G=G(n,p)$ by adding $\eps n^2p/(2e(F))$ copies
$F_1,F_2,\ldots$
of $F$, chosen uniformly at random from all subgraphs of $K_n$ isomorphic to $F$.
After deleting the small number of duplicate edges, we have added around
$\eps n^2p/2$ edges, so $s_p(K_2,G')\sim 1+\eps$.
It is easy to check that the number of $C_4$s in $G'$ containing two or more edges
from one single $F_i$ is negligible and thus, considering
$C_4$s formed from all combinations of edges from $G(n,p)$ and 
from different $F_i$, that $s_p(C_4,G')\sim (1+\eps)^4$ whp.
Hence, the appropriate limiting kernel is the constant kernel
$\ka=1+\eps$. Copies of $F$ itself containing at most one
edge from each $F_i$ contribute $(1+\eps)^{e(F)}$ to $s_p(F,G')$,
but there are $\Theta(n^{|F|}p^{e(F)})$ extra copies
of $F$, namely the $F_i$ themselves. It follows that $s_p(F,G_n)\not\to 1$.
If $n^{|F|}p^{e(F)}=o(n^2p)$, then the argument is much simpler: adding
a few copies of $F$ to $G(n,p)$ does not change the number of edges or $C_4$s
significantly, but does change the number of copies of $F$.

We can go somewhat further: the construction in Example~\ref{vst}
shows that for $C_3$ to be admissible, the expected number of $C_3$s per edge
should be larger than $\log n$. A similar construction can be carried
out for any fixed $F$, and shows that, at least for suitable balanced $F$,
we should require $n^{|F|}p^{e(F)}/(np^2\log n)\to \infty$
for $F$ to be admissible. In general, for $F$ to be admissible,
we need all induced subgraphs $F'$ of $F$ to be admissible; otherwise,
the distribution of copies of $F$ over $F'$ cannot be flat as we
expect in the uniform case.

Returning to triangles,
in the light of the comments above, perhaps the strongest conceivable extension
of Conjecture~\ref{ctri} to smaller $p$ would be that if
$p=p(n)=\omega(\sqrt{\log n}/\sqrt{n})$, and
$s_p(K_2,G_n)\to 1$, $s_p(C_4,G_n)\to 1$, and $\sup_n s_p(K_{2,t},G_n)<\infty$
for each $t$,  then $s_p(C_3,G_n)\to 1$.
However, it may well be that
the graphs constructed by Alon~\cite{Alon} mentioned in Example~\ref{noga}
have $s_p(K_{2,t},G_n)\to 1$ for each $t$.
(This may also be true of Kim's random construction~\cite{Kim}
giving his famous lower bound on the Ramsey numbers $R(3,t)$.)
If so, blowing
these graphs up as in Example~\ref{tdense} would show that even
in the almost dense case, controlling the $K_{2,t}$ counts is not enough,
so one should control (at least) the $K_{s,t}$ counts for some larger $s$.
Returning to much sparser graphs, we then have to limit ourselves
to $p=p(n)$ for which $K_{s,t}$ is admissible, suggesting the following conjecture.

\begin{conjecture}
There are constants $s\ge 2$ and $a>0$ such that,
if $p=p(n)=\omega((\log n)^a n^{-1/s})$ and
$G_n$ is a sequence of graphs with $|G_n|=n$,
$s_p(K_2,G_n)\to 1$, $s_p(C_4,G_n)\to 1$, and $\sup_n s_p(K_{s,t},G_n)<\infty$
for each $t$, then $s_p(C_3,G_n)\to 1$.
\end{conjecture}

It may be that if the conjecture holds for a given $s$, it holds with $c=1/s$.
It may also be that one needs to control the counts for $K_{s,t}$
and at the same time to consider $p$ larger than $n^{-b}$ for some
$b<1/s$.

There is a potential pitfall in handling subgraph counts
when $p$ is smaller than $n^{-o(1)}$: in proving that $s_p(F,G_n)\to 1$
for various graphs $F$ above,
we made use of the assumption that $s_p(F',G_n)$ is bounded
for other graphs $F'$. In particular, with $F=K_{2,t}$, we used
this assumption for $F'=K_{2,t+1}$. It may be that $F'$ is admissible
whenever $F$ is (as is likely in this case: $K_{2,t}$
should be admissible as soon as $C_4$ is),
but perhaps not. In the latter case we may be forced to work with a larger
admissible set for which we impose the hypothesis of Conjecture~\ref{q1a} (or Conjecture~\ref{q1}),
and a smaller set for which we obtain the conclusion.
In any case, the (smaller) admissible set should have the following property:
if $\F_\alpha$ denotes the set of admissible graphs when $p=n^{-\alpha}$,
$\alpha>0$,
then the sets $\F_\alpha$ should increase as $\alpha$ decreases, and their union
should contain all finite graphs.
We shall return to this question in Section~\ref{sec_compar}, in particular
in Subsections~\ref{sec_em} and~\ref{sec_hm}, where we prove
results that are steps towards (non-uniform) versions of the various
conjectures in this section.

\section{Szemer\'edi's Lemma and the cut metric}\label{sec_Sz}

In the next section we shall discuss the relationship between the
cut and count metrics. As in
the dense case, a key tool in the study of the cut metric is some
variant of Szemer\'edi's Lemma~\cite{Szem}: this will be discussed in this
section. Unlike in the dense case, we need an assumption on the
graphs we consider to make this useful; roughly speaking, our
assumption is that no subgraph of $G_n$ containing a constant
fraction of the vertices has density more than a constant factor
larger than it should have. Several of the usual proofs of
Szemer\'edi's Lemma extend easily to the sparse case under this
assumption; this was noted independently by Kohayakawa and R\"odl;
see~\cite{KR-2003}. (The much earlier Theorem 2 of
Kohayakawa~\cite{K97} is slightly different.)

Throughout this section, $p=p(n)$ with $p=o(1)$ and $np\to\infty$.
(Often, $n^2p\to\infty$ is enough in the proofs, but see
Remark~\ref{r_sp}.) As before, $(G_n)$ always denotes a sequence of
graphs with $|G_n|=n$, which need not be defined for all $n$, but
only for some infinite set.

For disjoint sets $A$, $B$ of vertices of a graph $G=G_n$ with
$n$ vertices, we write $e_G(A,B)$ for the number of edges
of $G$ joining $A$ to $B$, and
\begin{equation}\label{dpdef}
  d_p(A,B) = \frac{e_G(A,B)}{p|A||B|}
\end{equation}
for the normalized density of $G$ between $A$ and $B$. It is
convenient to extend this definition to sets $A$ and $B$ that need
not be disjoint: in this case, we write $e_G(A,B)$ for the number of
ordered pairs $(i,j)$ with $i\in A$, $j\in B$ and $ij\in E(G)$; we
then define $d_p(A,B)$ as above. Note that $e_G(A,A)=2e(G[A])$.
We shall make the following assumption:

\begin{assumption}[bounded density]\label{AC}
There
is a constant $C$ and a function $n_0(\eps)$ such that, for every $\eps>0$
and $n\ge n_0(\eps)$, and any $A$, $B\subset V(G_n)$ with $|A|$, $|B|\ge \eps n$,
we have $d_p(A,B)\le C+\eps$.
\end{assumption}

It suffices to impose this assumption only when $A=B$, replacing $C$ by $C/2$
and $\eps$ by $\eps/2$.
Indeed, if $|A|, |B|\ge \eps n$, $n\ge n_0(\eps)$, and $d_p(A,B)> C+\eps$ then,
by averaging, we may find $A'\subset A$
and $B'\subset B$ with $|A'|=|B'|=\ceil{\eps n}$ such that $d_p(A',B')>C+\eps$.
Then $e_G(A'\cup B',A'\cup B')\ge 2e_G(A',B')> 2(C+\eps)|A'|^2 \ge (C/2+\eps/2)|A'\cup B'|^2$.

The condition above may be written more compactly as follows:
\begin{equation}\label{aC}
 \forall \eps>0:\ \limsup_{n\to\infty} \max\{d_p(A,B): A, B\subset V(G_n), |A|,|B|\ge \eps n\} \le C.
\end{equation}

Note that we shall often assume that \eqref{aC} holds for a particular value of $C$: in this
case, we say that $(G_n)$ has {\em density bounded by $C$}. This is the reason
for including the final $+\eps$ in Assumption~\ref{AC}.

It will be convenient to phrase the proof of Szemer\'edi's Lemma in
terms of kernels. In this sparse setting, the way in which we
associate a kernel to a graph is different from in the dense case.
Indeed, our aim is that the random graph $G(n,p)$ should approximate
the constant kernel taking value $1$. For this reason, to a graph
$G$ with $n$ vertices $1,2,\ldots,n$ we associate the kernel
$\kappa_G$ taking the value $1/p$ on each square $((i-1)/n,i/n]
\times ((j-1)/n,j/n]$ whenever $ij\in E(G)$, and zero elsewhere.
This association will often be implicit: for example, given a graph
$G$ and a kernel $\ka$, we write $\dc(G,\ka)$ for $\dc(\ka_G,\ka)$.

The following observation shows the importance of bounded density.
In the proof, and throughout this section, given a subset $A$ of the
vertices of a graph $G$, we shall often abuse notation by also
writing $A$ for the corresponding subset of $[0,1]$.

\begin{lemma}\label{kbd}
Let $p=p(n)$ be any function of $n$, let $\ka:[0,1]^2\to [0,C]$
be a kernel, and let $(G_n)$ be a sequence of graphs with $|G_n|=n$ and $\dc(G_n,\ka)\to 0$.
Then $(G_n)$ has density bounded by $C$.
\end{lemma}
\begin{proof}
Suppose that $(G_n)$ does not have density bounded by $C$. Then there is an $\eps>0$
such that, for infinitely many $n$, there are sets $A_n,B_n\subset V(G_n)$
with $|A_n|,|B_n|\ge \eps n$ and $d_p(A_n,B_n)\ge C+\eps$. Identifying
$A_n$ and $B_n$ with subsets of $[0,1]$, and writing $\mu$
for Lebesgue measure, we have
\[
 \int_{A_n\times B_n} \ka_{G_n} = d_p(A_n,B_n)\mu(A_n)\mu(B_n)\ge (C+\eps)\mu(A_n)\mu(B_n).
\]
Since $\ka$ is bounded by $C$, it follows that
\[
 \left|\int_{A_n\times B_n} \ka_{G_n}-\ka^{(\tau)}
 \right|\ge \eps\mu(A_n)\mu(B_n)\ge \eps^3
\]
for any rearrangement $\ka^{(\tau)}$ of $\ka$,
which contradicts $\dc(G_n,\ka)\to 0$.
\end{proof}

\subsection{Weakly regular partitions}\label{sec4.1}

If $G$ is a graph with vertex set $\{1,2,\ldots,n\}$,
and $\Pi=(P_1,\ldots,P_k)$ is a partition of $V(G)$, then we write $G/\Pi$
for the kernel on $[0,1]^2$ taking the value $d_p(P_a,P_b)$
on the union of the squares $((i-1)/n,i/n]\times (j-1)/n,j/n]$,
$i\in P_a$, $j\in P_b$.
We say that a partition $\Pi$ of a graph $G$ is {\em weakly $(\eps,p)$-regular}
if $\cn{\ka_G-G/\Pi}\le \eps$. Note that the normalizing function $p$ comes
in via the definition of the kernels $\ka_G$ and $G/\Pi$.

For a kernel $\ka$, the definitions are similar:
for $A$, $B\subset [0,1]$ we write $\ka(A,B)$ for the integral of $\ka$
over $A\times B$, and
\[
 d(A,B)=d_\ka(A,B)=\frac{\ka(A,B)}{\mu(A)\mu(B)}
\]
for the average value of $\ka$ on $A\times B$.
Then $d_p(A,B)$, defined
using $G$, is exactly $d(A,B)$, defined using $\ka_G$, so the kernel
$G/\Pi$ is obtained from $\ka_G$ by replacing the value at each
point by the average over the relevant rectangle $P_a\times P_b$.
For $\ka$ a kernel and $\Pi$ a partition of $[0,1]^2$, we define
$\ka/\Pi$ similarly. The partition $\Pi$ is {\em weakly $\eps$-regular}
with respect to $\ka$ if $\cn{\ka-\ka/\Pi}\le \eps$.

The next lemma is a a sparse equivalent of (a version of) the Frieze--Kannan `weak' form
of Szemer\'edi's Lemma from~\cite{FKquick}. As with many proofs of the various
forms of Szemer\'edi's Lemma, the proof of the dense result is not hard
to adapt to the sparse setting:
the only
additional complication is that one must make sure that the parts
of the partition remain large enough so that we can make use of the bounded density assumption.
In the following lemma, $p=p(n)$ is any normalizing
function with $pn^2\to\infty$. In principle, the various constants depend
on the choice of $p$, but this is not the case if we impose
an explicit lower bound on $p(n)$, such as the harmless bound $p\ge n^{3/2}$.

\begin{lemma}\label{lSzw}
Let $p=p(n)$ be any function with $0<p\le 1$ and $pn^2\to\infty$.
Let $\eps>0$, $C>0$ and $k\ge 1$ be given.
There exist constants $n_0$, $K$ and $\eta>0$, all depending on $\eps$, $C$ and $k$,
such that, if $G_n$
is any graph with $n\ge n_0$ vertices such that
\begin{equation}\label{wC}
 d_p(A,B)\le C \hbox{ whenever }|A|,|B|\ge \eta n,
\end{equation}
and $\Pi$ is any partition of $V(G)$ into $k$ parts $P_1,\ldots,P_k$ with
sizes as equal as possible, then there is a weakly $(\eps,p)$-regular partition $\Pi'$
of $V(G_n)$ into $K$ parts that refines~$\Pi$.
\end{lemma}
\begin{proof}
Reducing $\eps$ if necessary, we may assume that $\eps \le C$, say.
We assume without comment that $n$ is `large enough' whenever this is needed.

Let $\Pi_0=\Pi$. We shall inductively define a sequence $\Pi_t$ of partitions
of $V(G)$ into $k_t=2^tk$ parts, stopping either when we reach
some $\Pi_t$ that is weakly $(\eps/2,p)$-regular, or when $t\ge T=\ceil{16C^2/\eps^2}+1$.
Every part of $\Pi_t$ will have size at least $\gamma^t n/(2k)$, where $\gamma=\eps/(100C)\le 1/100$.
Note that $\Pi_0$ satisfies this condition.

Set $\eta=\gamma^T/(2k)$, and let $n_0$ be a large constant to be chosen later.
We shall write $\ka_t$ for the kernel $G/\Pi_t$, noting that, since
all parts of $\Pi_t$ have size at least $\eta n$, the kernel $\ka_t$
is bounded by $C$.

Given $\Pi_t$ as above, suppose that $\Pi_t$ is not weakly $(\eps/2,p)$-regular.
Then there is a cut $[0,1]=A \cup A^\cc$ exhibiting this, i.e., a set $A\subset [0,1]$ 
for which
$|\ka_G(A,A^\cc)-\ka_t(A,A^\cc)|\ge \eps/2$. Since both $\ka_G$
and $\ka_t$ correspond to weighted graphs on $V(G)=\{1,2,\ldots,n\}$, we may choose
the cut $A$ to correspond to a subset of $V(G)$: among all `worst' cuts, there
is a cut of this form.

Our aim is to modify $A$ slightly to obtain a set $B$ (which we may think of
as a subset of $V(G)$ or as a subset of $[0,1]$) and
then take two parts $P_i\cap B$ and $P_i\cap B^\cc$ of $\Pi_{t+1}$ for
each part of $\Pi_t$; in doing so, we must ensure that neither of these
parts is too small.
We modify the set $A$ to obtain $B$ in $k_t$ stages, one for each part $P_i$.
At each stage, we move a set $S$ of at most $\gamma|P_i|\ge \eta n$ vertices
from $A$ to $A^\cc$ or vice versa, to ensure that both $B$ and $B^\cc$
meet $P_i$ in at least $\gamma|P_i|$ vertices.
Since $\ka_t$ is bounded by $C$, this changes the value
of the cut $\ka_t(A,A^\cc)$ by at most $2C\gamma|P_i|/n$.

From \eqref{wC}, the set $S$ meets at most $Cpn\gamma |P_i|$ edges
of $G$: to see this, apply \eqref{wC} to $S$ and $V(G)$ if $|S|\ge \eta n$,
and to $S'$ and $V(G)$ otherwise, for any $S'\supset S$ with $\ceil{\eta n}$ vertices.
Hence, the value of the cut
$\ka_G(A,A^\cc)$ changes by at most $2C\gamma|P_i|/n$ when we move our set $S$
from one side of the cut to the other.
After all these changes, we have
\[
 |\ka_t(A,A^\cc)-\ka_t(B,B^\cc)|, \  |\ka_G(A,A^\cc)-\ka_G(B,B^\cc)|
 \le 2C\gamma\le \eps/8.
\]
It follows that
\begin{equation}\label{badcut}
 |\ka_G(B,B^\cc)-\ka_t(B,B^\cc)|\ge \eps/4.
\end{equation}

Let $\Pi_{t+1}$ be the partition obtained by intersecting each part of $\Pi_t$
with $B$ and $B^\cc$, noting that $\Pi_{t+1}$ has all
the required properties. Set $\ka_{t+1}=G/\Pi_{t+1}$, noting
that $\ka_{t+1}(B,B^\cc)=\ka_G(B,B^\cc)$, since $\Pi_{t+1}$ refines
the partition $(B,B^\cc)$. From
\eqref{badcut} it thus follows that
\[
 \on{\ka_{t+1}-\ka_t} \ge \cn{\ka_{t+1}-\ka_t} \ge \eps/4,
\]
with the final inequality witnessed by the cut $(B,B^\cc)$.
Hence, $\tn{\ka_{t+1}-\ka_t}^2 \ge \on{\ka_{t+1}-\ka_t}^2 \ge \eps^2/16$.
Since $\ka_t$ may be obtained from $\ka_{t+1}$ by averaging over rectangles,
$\ka_t$ and $\ka_{t+1}-\ka_t$ are orthogonal:
for any two parts $P_i$, $P_j$ of $\Pi_t$, the kernel $\ka_t$ is constant
on $P_i\times P_j$. Also, $\int_{P_i\times P_j}\ka_{t+1} = \int_{P_i\times P_j}\ka_G = \int_{P_i\times P_j}\ka_t$.
Thus $\int_{P_i\times P_j}\ka_t(\ka_{t+1}-\ka_t)=0$. Summing over $i$ and $j$
it follows that $\int \ka_t(\ka_{t+1}-\ka_t)=0$.
Thus,
\[
 \tn{\ka_{t+1}}^2 = \tn{\ka_t}^2 + \tn{\ka_{t+1}-\ka_t}^2 \ge \tn{\ka_t}^2+\eps^2/16.
\]
It follows by induction that $\tn{\ka_t}^2\ge t\eps^2/16$ as long as our
construction continues.
But, as noted above, $\ka_t$ is bounded by $C$, so our construction
must stop after at most $16C^2/\eps^2$ steps. Since this number is smaller than $T$,
we must stop at a weakly $(\eps/2,p)$-regular partition.

To complete the proof we modify the final partition $\Pi_t$ slightly.
Set $K=k\ceil{\gamma^{-T}}$, and note that, since $t\le T-1$, each part of $\Pi_t$
has size at least $\gamma^{-1}n/K$. First, adjust the parts slightly so that the
size of each is of the form $a\floor{n/K}+b\ceil{n/K}$, $a,b\in \Z^+$,
replacing the kernel $\ka_t$ by a new kernel $\ka'$ corresponding to the altered
partition $\Pi'$.
Arguing as above, $\cn{\ka_t - \ka'}\le 2C\gamma\le \eps/4$,
so, by the triangle inequality and weak $(\eps/2,p)$-regularity of $\Pi_t$, we have
\[
 \cn{\ka'-\ka_G}\le \cn{\ka_t-\ka_G} + \cn{\ka_t-\ka'}
 \le \eps/2+\eps/4=3\eps/4.
\]
Finally, we split each part randomly into parts of sizes exactly $\floor{n/K}$
and $\ceil{n/K}$, obtaining a partition $\Pi''$ into $K$
parts whose sizes are as equal as possible. We write $\ka''$
for the corresponding kernel. Since $\Pi'$ has $O(1)$ parts, and we
have $\Theta(pn^2)$ edges between any two parts with density at least $\eps/100$, say,
it follows from Chernoff's inequality that if $n$ is large enough,
which we enforce by choosing $n_0$ suitably, then with probability
at least $0.99$ the density $d_p(A,B)$ between every pair $(A,B)$ of
new parts $A$ and $B$ coming from parts $P_i$ and $P_j$ of $\Pi'$
with $d_p(P_i,P_j)\ge \eps/100$ is $d_p(P_i,P_j)(1+o(1))$.
Since the densities $d_p(P_i,P_j)$ are uniformly bounded by $C$,
it follows that with probability at least $0.99$
we have $\on{\ka''-\ka'}\le \eps/100$.
But then
\[
 \cn{\ka''-\ka_G}\le \cn{\ka''-\ka'}+\cn{\ka'-\ka_G}
 \le \on{\ka''-\ka'} +\cn{\ka'-\ka_G} \le \eps/100+\eps/2,
\]
so our final partition $\Pi''$ is indeed weakly $(\eps,p)$-regular.
\end{proof}

If for any reason we want a weakly $(\eps,p)$-regular partition into a
particular number $K$ of 
parts (which must be a multiple of the number in the original
partition if we are refining a given partition), the proof
above gives such a partition for any large enough $K$, indeed,
for any $K\ge k\ceil{\gamma^{-T}}$. Of course, $n_0$ then depends
on $K$.

\begin{remark}\label{r_sp}
The proof of Lemma~\ref{lSzw} works even
if $p$ is very small, say of order $1/n$. However, this is of no help -- it is impossible
for Assumption~\ref{AC} (the sequence version of \eqref{wC})
to be satisfied in this range, except in the trivial case where $e(G_n)=o(pn^2)$
(so $p$ is not the appropriate normalizing function).
Indeed, passing to a subsequence where $e(G_n)/(pn^2)$ is bounded away from zero,
picking any $\eps pn^2$ edges of $G_n$, and putting one endpoint of each edge
into $A$ and the other into $B$, we find sets $A$, $B$ with $|A|,|B|\le \eps pn^2$
but $e(A,B)\ge \eps pn^2$, which gives $d_p(A,B)\ge 1/(\eps p^2n^2)$, which tends
to infinity as $\eps\to0$.
\end{remark}

\subsection{Strongly regular partitions}\label{sec4.2}

Usually, when working with the cut metric, weak $\eps$-regularity
turns out to be just as good as the usual stronger
$\eps$-regularity. In the dense case, this is true also when
considering subgraph counts. However, for the subgraph counts we
consider in the next section, it turns out that we do in fact need
the usual form of $\eps$-regularity.

As usual, a pair $(A,B)$ of (not necessarily disjoint) subsets of $V(G)$
is an {\em $(\eps,p)$-regular pair} if
$|d_p(A',B')-d_p(A,B)|\le \eps$ whenever $A'\subset A$ and $B'\subset B$
satisfy $|A'|\ge \eps |A|$ and $|B'|\ge \eps|B|$.
A partition $\Pi=(P_1,\ldots,P_k)$ of $V(G)$ is {\em $(\eps,p)$-regular}
if the parts $P_i$ each have size $\ceil{n/k}$ or $\floor{n/k}$,
and all but at most $\eps\binom{k}{2}$ of the unordered
pairs $\{P_i,P_j\}$, $i\ne j$, are $(\eps,p)$-regular.
The definition (now simply of $\eps$-regularity) for a kernel is similar,
although here one partitions the interval $[0,1]$ into parts
with measure exactly $1/k$.

The following is (essentially) the sparse version of Szemer\'edi's Lemma
observed by Kohayakawa and R\"odl; see~\cite{KR-2003}, where a closely related
result is proved. For a proof, see also Gerke and Steger~\cite{GSsurvey}.
We shall include a proof here as we state the result in a slightly different
way (which makes no real difference), and the use of kernels allows one to 
phrase the proof a little more simply than in~\cite{KR-2003} or~\cite{GSsurvey}.

\begin{lemma}\label{lSzs}
Let $p=p(n)$ be any function with $0<p\le 1$ and $pn^2\to\infty$.
Let $\eps>0$, $C>0$ and $k\ge 1$ be given.
There exist constants $n_0$, $K$ and $\eta>0$, all depending on $\eps$, $C$ and $k$,
such that, if $G_n$
is any graph with $n\ge n_0$ vertices such that
\begin{equation}\label{sC}
 d_p(A,B)\le C \hbox{ whenever }|A|,|B|\ge \eta n,
\end{equation}
and $\Pi$ is any partition of $V(G)$ into $k$ parts $P_1,\ldots,P_k$ with
sizes as equal as possible, then there is an $(\eps,p)$-regular partition $\Pi'$
of $V(G_n)$ into at most $K$ parts that refines~$\Pi$.
\end{lemma}

\begin{proof}
Reducing $\eps$ and/or increasing $C$ if necessary, we may suppose
for convenience that $\eps\le 1$ and $C\ge 1$.

Set $\gamma=\eps^3/(100C)$.
This time we inductively define a sequence $\Pi_t$ of partitions
of $V(G)$ into $k_t$ parts, where $\Pi_0=\Pi$, $k_0=k$, and $k_{t+1}=k_t\ceil{k_t2^{k_t}/\gamma}$,
stopping either when we reach
some $\Pi_t$ that is $(\eps,p)$-regular, or when $t\ge T=\ceil{20C^2/\eps^5}+1$.
The parts of each $\Pi_t$ will have sizes as equal as possible.
Note that $\Pi_0$ satisfies this condition.

Set $\eta=1/(2k_T)$, and let $n_0$ be a large constant to be chosen later.
We assume throughout that $n\ge n_0$.
As before, we write $\ka_t$ for the kernel $G/\Pi_t$, noting that, since
all parts of $\Pi_t$ have size at least $\eta n$, the kernel $\ka_t$
is bounded by $C$.

The key (standard) observation
is the following. Let $A$ and $B$ be parts of $\Pi_t$, so $\ka_t$ is by definition
constant on $A\times B$, and let $A'\subset A$ and $B'\subset B$.
Let $\Pi'$ be any partition refining $\Pi$ such that each of $A'$
and $B'$ is a union of parts of $\Pi'$, and let $\ka'=G/\Pi'$ be the corresponding
kernel. Restricted to $A\times B$, the function $\ka'$ integrates
to $d_p(A,B)\mu(A)\mu(B)=\int_{A\times B}\ka_t$, since
$A$ and $B$ are unions of parts of $\ka'$. Hence, $\ka_t$ and $\ka'-\ka_t$
are orthogonal on this set. Using the fact that $A'$ and $B'$ are unions
of parts of $\ka'$, we see that $\int_{A'\times B'}\ka'=d_p(A',B')\mu(A')\mu(B')$,
which differs from the integral of $\ka_t$ over the same set
by $|d_p(A',B')-d_p(A,B)|\mu(A')\mu(B')$. It follows that $\tn{\ka'-\ka_t}^2$
is at least $\bb{d_p(A',B')-d_p(A,B)}^2\mu(A')\mu(B')$, and hence, using orthogonality, that
\begin{equation}\label{incr}
 \int_{A\times B} (\ka')^2 \ge \int_{A\times B}\ka_t^2 + \bb{d_p(A',B')-d_p(A,B)}^2\mu(A')\mu(B').
\end{equation}

Suppose then that $\Pi_t$ is not $(\eps,p)$-regular, and let
$A_1,\ldots,A_{k_t}$ denote the parts of $\Pi_t$.
Then there are at least
$\eps \binom{k_t}{2}$ pairs $\{A_i,A_j\}$ of parts of $\Pi_t$ that are not 
$(\eps,p)$-regular. For each, pick sets $A_{ij}\subset A_i$ and $A_{ji}\subset A_j$ witnessing this,
i.e., with $|d_p(A_{ij},A_{ji})-d_p(A_i,A_j)|\ge \eps$ and $|A_{ij}|\ge \eps |A_i|$, $|A_{ji}|\ge \eps |A_j|$.
Let $\Pi'$ be the partition whose parts are all atoms formed by the sets $A_i$ and the
sets $A_{ij}$ taken together, so $\Pi'$ refines $\Pi_t$, and each $A_{ij}$ is a union
of parts of $\Pi'$.
We could estimate the $L^2$-norm of $G/\Pi'$ using \eqref{incr},
but this will not be useful if some parts of $\Pi'$
are too small, so we first adjust the part sizes.

Define $\Pi_{t+1}$
by dividing
each $A_i$ into $k_{t+1}/k_t$ parts whose sizes are as equal as possible,
so that each part of $\Pi'$ differs from a union of parts of $\Pi_{t+1}$
in at most $n/k_{t+1}$ vertices: to do this, keep taking for a part of $\Pi_{t+1}$ a subset
of some part of $\Pi'$, until what is left of every part of $\Pi'$ is too small.
For each $i$, there are at most $k_t$ sets $A_{ij}$ inside $A_i$,
so $A_i$ is a union of at most $2^{k_t}$ parts of $\Pi'$. 
It follows that there is some union $A_{ij}'$ of parts of $\Pi_{t+1}$ with
\[
 |A_{ij}-A'_{ij}|\le 2^{k_t}n/k_{t+1} \le \gamma n/k_t^2.
\]
Arguing as in the proof of Lemma~\ref{lSzw}, it follows from \eqref{sC} 
that the symmetric difference $S_{ij}$ of $A_{ij}$ and $A'_{ij}$ meets
at most 
\[
 Cpn|S_{ij}| \le Cp\gamma n^2/k_t^2 \le \eps^3 p |A_i||A_j|/99
\]
edges of $G$, if $n$ is sufficiently large.
Since $|S_{ij}|\le \eps^3 |A_i|/100$, say, while $|A_{ij}|\ge \eps |A_i|$
and $|A_{ji}|\ge \eps |A_j|$, it follows crudely that
\[
 |d_p(A_{ij}',A_{ji}')-d_p(A_{ij},A_{ji})| \le \eps/2,
\]
which implies that
\[
 |d_p(A_{ij}',A_{ji}')-d_p(A_i,A_j)| \ge \eps/2.
\]
Now $A_{ij}'$ and $A_{ji}'$ are unions of parts of $\Pi_{t+1}$,
and these sets have size at least $\eps n/(2k_t)$.
Hence, from \eqref{incr},
\[
 \int_{A_i\times A_j} \ka_{t+1}^2 \ge \int_{A_i\times A_j} \ka_t^2 + \eps^4/(16k_t^2)
\]
for each of the at least $\eps \binom{k_t}{2}$ irregular pairs $\{A_i,A_j\}$.
Since $\int_{A_i\times A_j} \ka_{t+1}^2 \ge \int_{A_i\times A_j} \ka_t^2$ always holds,
it follows
that $\tn{\ka_{t+1}}^2\ge \tn{\ka_t}^2+\eps^5/20$.

If the construction above does not stop before step $T$, then by induction we have
$\tn{\ka_t}\ge t\eps^5/20$ for $0\le t\le T$. But each $\ka_t$ is bounded
by $C$, so $\tn{\ka_T}^2\le C^2$, giving a contradiction. Hence
the construction does stop before step $T$, giving an $(\eps,p)$-regular partition
with $k_t\le k_T$ parts.
\end{proof}

Note that Lemma~\ref{lSzs} implies (essentially) Lemma~\ref{lSzw}: it is easy to check that
an $(\eps,p)$-regular partition is, say, weakly $(10(C+1)\eps,p)$-regular, provided
the parts are large enough for \eqref{sC} to hold. However, one of course
obtains much worse bounds on the number of parts using the stronger notion of regularity.

\begin{remark}
Let us illustrate once again the difference between the dense
and sparse cases with a simple observation. Given a pair $(A,B)$ of sets
of vertices of a graph $G$,
let $C_4(A,B)$ denote the number
of homomorphisms from $C_4$ into the subgraph spanned by 
$A\cup B$ mapping a given pair of opposite vertices into $A$
and the other pair into $B$. Standard convexity arguments
show that $C_4(A,B)\ge d(A,B)^4|A|^2|B|^2$. The pair $(A,B)$
is {\em $(\eps,p)$-$C_4$-minimal} if $C_4(A,B)\le (d(A,B)^4+\eps p^4)|A|^2|B|^2$.
In the dense case (with $p=1$) it is well known and very easy to check
that $\eps$-regularity and $\eps$-$C_4$-minimality are essentially
equivalent: $\eps$-regularity implies $f(\eps)$-$C_4$-minimality,
and $\eps$-$C_4$-minimality implies $g(\eps)$-regularity, for some
$f(\eps),g(\eps)$ with $f(\eps),g(\eps)\to 0$ as $\eps\to 0$.

Let $\eps>0$ and $M$ be given. By counting $C_4$s it is easy to see that
there is a function $f(\eps)$ with $f(\eps)\to 0$ as $\eps\to 0$
such that,
if $n$ is large enough and $(A,B)$ is $\eps$-regular with $|A|=|B|=n$,
then we may partition $A$ and $B$ into sets $A_1,\ldots,A_M$ and $B_1,\ldots,B_M$
of almost equal sizes so that every pair $(A_i,B_j)$ is $f(\eps)$-regular.
Indeed, a random partition has this property with probability
tending to 1, since
by standard concentration results (for example, the Hoeffding--Azuma
inequality), the edge densities and `$C_4$-densities' of the pairs
$(A_i,B_j)$ are highly concentrated about the corresponding densities
for $(A,B)$.
It follows immediately that in the usual dense Szemer\'edi's Lemma~\cite{Szem}, we may
specify in advance the number of parts $K$ we would like our partition
to have, provided (as in the weak case) that $K$ is large enough given
$\eps$, and $n$ large enough given $\eps$ and $K$.

In the sparse case, the fact about random partitioning above is presumably
true, but the simple proof using $C_4$-counts fails totally. It is still
true that $(\eps,p)$-$C_4$-minimality implies $(f(\eps),p)$-regularity,
but the reverse implication fails. Indeed, whenever $p=p(n)\to 0$,
given any pair $(A,B)$,
we may add a small dense (say complete bipartite) subgraph with too
few edges to disturb regularity, but containing many more than $p^4|A||B|$
$C_4$s.
\end{remark}

\subsection{Szemer\'edi's lemma and convergence in the cut
norm}\label{sec4.3}

We start with a consequence of Lemma~\ref{lSzw} concerning the cut norm.

\begin{corollary}\label{cSz}
Let $(G_n)$ be a sequence of graphs satisfying Assumption~\ref{AC}. Then there
is a kernel $\ka:[0,1]^2\to [0,C]$ and a subsequence $(G_{n_i})$ of $(G_n)$
such that $\dc(G_{n_i},\ka)\to 0$. Moreover, we may label the vertices
of $G_{n_i}$ with $1,2,\dots,n_i$ so that $\cn{\ka_{G_{n_i}}-\ka}\to 0$.
\end{corollary}
\begin{proof}
We shall only sketch the proof as the argument is exactly the same as
that of Lov\'asz and Szegedy~\cite{LSz1} for the dense case. Note that given any $\eta>0$
and $\eps>0$,
our graphs $G_n$ satisfy the assumption \eqref{wC} of Lemma~\ref{lSzw}
with $C+\eps$ in place of $C$ whenever $n$ is large enough.

First, let us apply Lemma~\ref{lSzw} with $k=1$ and $\eps=\eps_1=1/2$, say, to obtain
a weakly $(\eps_1,p)$-regular partition $\Pi_{n,1}$ of $G_n$ into
$k_1=K$ parts, for all large enough $n$. We may relabel the vertices of each $G_n$
so that the parts of $\Pi_{n,1}$ are all intervals.
Each kernel $G_n/\Pi_{n,1}$ is characterized by a $k_1$-by-$k_1$ density matrix,
whose entries all lie in $[0,C+\eps_1]$. (Indeed, if $k_1$ happens to divide $n$, then
the kernel is exactly the kernel obtained from the matrix in the obvious way.)
Since these matrices live in a compact set, $[0,C+\eps_1]^{k_1^2}$, they have a convergent
subsequence. Passing to the corresponding subsequence of $G_n$, we then
have $G_n/\Pi_{n,1}\to \ka_1$ pointwise almost everywhere, and hence in $L^1$ and
in the cut norm. Since the partitions $\Pi_{n,1}$ are weakly $(\eps_1,p)$-regular,
we have $\cn{\ka_{G_n} - G_n/\Pi_{n,1}}\le \eps_1$. Passing far enough along
our subsequence, it follows that $\cn{\ka_{G_n} - \ka_1}\le 2\eps_1$.

Working within the subsequence defined above, apply Lemma~\ref{lSzw} again
with $\eps=\eps_2=1/4$, say, and $k=k_1$. For each $n$ we find a partition $\Pi_{n,2}$
refining $\Pi_{n,1}$, with $k_2=K(\eps_1,C,k_1)$ parts.
Relabelling vertices, we may assume
that each part of each $\Pi_{n,2}$ is an interval. (Note that we only reorder
the vertices within parts of $\Pi_{n,1}$.) As before, on a subsequence we
have $G_n/\Pi_{n,2}\to \ka_2$, for some kernel $\ka_2$ constant on squares
of side-length $1/k_2$.
Since $\Pi_{n,2}$ refines $\Pi_{n,1}$ for each $n$, it follows that the value
of $\ka_1$ on each $1/k_1$-by-$1/k_1$ square is exactly the average of $\ka_2$
over this set; to see this, let $n\to\infty$.

Iterating, we find kernels $\ka_1,\ka_2,\ldots$ each of which can be obtained
by averaging the next one, and graphs $G_{n_i}$ with $\cn{\ka_{G_{n_i}} -\ka_i}\le
2\eps_i=2^{1-i}$, say.
To complete the proof we simply observe that the sequence $(\ka_t)$ is a martingale
on the state space $[0,1]^2$. Since each $\ka_t$ is bounded by $C+\eps_t\le C+1$, by the Martingale
Convergence Theorem there 
is a kernel $\ka:[0,1]^2\to [0,C]$ with $\ka_t\to \ka$ pointwise almost everywhere,
and hence in $L^1$ and in the cut-norm. Then $\cn{\ka_{G_{n_i}} -\ka}\to 0$
as required.
\end{proof}

The corollary above says that any (suitable) sequence of graphs
has a subsequence converging to a kernel, and is a simple
consequence of Szemer\'edi's Lemma and the Martingale
Convergence Theorem. Together with Lemma~\ref{kbd}, it shows that Assumption~\ref{AC}
is the correct assumption to impose on sequences of graphs
when we seek limits that are bounded kernels $\ka:[0,1]^2\to[0,C]$.
Before turning to an application of Corollary~\ref{cSz},
let us note an even simpler consequence of the Martingale Convergence
Theorem.

\begin{lemma}\label{approx}
Let $\ka$ be a bounded kernel, and for $k\ge 1$, let $\ka_k$ be the piecewise
constant kernel obtained by dividing $[0,1]^2$ into $2^{2k}$ squares
of side $2^{-k}$, and replacing $\ka$ by its average over each square.
Then $\ka_k\to \ka$ pointwise almost everywhere and also in $L^p$ for any $p$.
\end{lemma}
\begin{proof}
The sequence $\ka_k$ is a bounded martingale on $[0,1]^2$, so pointwise
convergence is given by the Martingale Convergence Theorem. Since
the sequence $\ka_k$ is bounded by $\sup \ka$, convergence in $L^p$ follows
by dominated convergence.
\end{proof}

A consequence of Corollary~\ref{cSz} is that it allows us to compare
the two different versions of the cut metric. Recall that for graphs $G_1$, $G_2$,
we defined $\dc(G_1,G_2)$ by first passing to kernels taking
the values $0$ and $1/p$. If $G_1$ and $G_2$
have the same number of vertices, then
there is a more natural definition of their cut-distance, $\dcG(G_1,G_2)$,
defined in the same way but only allowing rearrangements
that `map whole vertices to whole vertices'. As in the dense
case, $\dc(G_1,G_2)$ and $\dcG(G_1,G_2)$ are defined by \eqref{dcdef} and
\eqref{dcg}, respectively; the difference between the sparse
and dense cases is in the normalization of $\ka_{G_i}$.
Writing $\dc^1$ and $\dcG^1$ for the metrics defined using $p=1$,
Borgs, Chayes, Lov\'asz, S\'os and Vesztergombi~\cite[Theorem 2.3]{BCLSV:1}
showed that these metrics are equivalent, proving that
\begin{equation}\label{cequiv}
 \dc^1(G_1,G_2)\le \dcG^1(G_1,G_2)\le 32\dc^1(G_1,G_2)^{1/67}.
\end{equation}
In fact, they proved \eqref{cequiv} for edge-weighted graphs,
as long as all edge weights lie in $[-1,1]$.
Unlike simple Lipschitz equivalence, which may also hold,
this does not directly carry over to the sparse setting: we have $\dc=p^{-1}\dc^1$
and $\dcG=p^{-1}\dcG^1$, so \eqref{cequiv} can be written as
\[
 \dc(G_1,G_2)\le \dcG(G_1,G_2)\le 32\dc(G_1,G_2)^{1/67}p^{-66/67},
\]
which is of little if any use here.
However, the equivalence of the two metrics in the sparse case
is not too hard to deduce from \eqref{cequiv}, using Corollary~\ref{cSz}.

\begin{lemma}\label{dcdcG}
For $i=1,2$, let $(G_n^{(i)})$ be a sequence of graphs satisfying
the bounded density assumption~\ref{AC}.
Then $\dc(G_n^{(1)},G_n^{(2)})\to 0$ if and only if
$\dcG(G_n^{(1)},G_n^{(2)})\to\nobreak 0$.
\end{lemma}
\begin{proof}
If $\dcG(G_n^{(1)},G_n^{(2)})\to 0$ then, since $\dc\le\dcG$, it follows trivially that
$\dc(G_n^{(1)},G_n^{(2)})\to 0$.

Suppose now that $\dc(G_n^{(1)},G_n^{(2)})\to 0$; our aim is to show
that $\dcG(G_n^{(1)},G_n^{(2)})\to 0$, so we may suppose that this is not
the case. Hence, passing to a subsequence, we may assume
that $\dcG(G_n^{(1)},G_n^{(2)})\ge \delta$
for some positive $\delta$ and all $n$ in our subsequence.

Applying Corollary~\ref{cSz} twice, the second time to a suitable subsequence,
we find kernels $\ka_1$, $\ka_2:[0,1]^2\to [0,C]$, and
subsequences of the sequences
$(G_n^{(i)})$, defined for the same values of $n$,
on which $\cn{\ka_{G_n^{(i)}}-\ka_i}\to 0$.
Since $\dc(G_n^{(1)},G_n^{(2)})\to 0$, it follows that $\dc(\ka_1,\ka_2)=0$.

For any $\eps>0$, by Lemma~\ref{approx}
we may find a $K$ and kernels $\ka_1',\ka_2':[0,1]^2\to[0,C]$
that are constant on squares of side $1/K$, with $\cn{\ka_i'-\ka_i}\le \eps$.
Since the kernels $\ka_i'$ may be thought of as weighted graphs, it would appear that we have
gone round in circles, but the point is that they are {\em dense} weighted graphs.
Regarding the kernels $\ka_1'/C$ and $\ka_2'/C$ as weighted graphs with
edge weights in $[0,1]$, we have 
\[
 \dc(\ka_1'/C,\ka_2'/C)= \dc(\ka_1',\ka_2')/C \le 2\eps/C,
\]
so \eqref{cequiv} gives
\[
 \dcG(\ka_1',\ka_2')=C\dcG(\ka_1'/C,\ka_2'/C) \le 32C(2\eps/C)^{1/67} = O(\eps^{1/67}).
\]
Hence, there is a rearrangement of $\ka_1'$ preserving intervals
that is close to $\ka_2'$ in the cut norm. Ignoring divisibility, adapting this rearrangement
to the graph $G_n^{(1)}$, $n$ much larger than $K$, 
and using $\cn{\ka_{G_n^{(i)}}-\ka_i'}\le \eps+o(1)$, it follows that 
that $\dcG(G_n^{(1)},G_n^{(2)})\le O(\eps)+O(\eps^{1/67})$.
Choosing $\eps$ small enough, the final bound is less than $\delta$,
contradicting our assumptions.
\end{proof}

Corollary~\ref{cSz} shows that one property of the cut metric
carries over to the sparse setting: for every suitable sequence
$(G_n)$, i.e., any sequence satisfying the
bounded density assumption~\ref{AC}, there is a kernel $\ka$
and a subsequence converging to $\ka$ in $\dc$. In the other
direction, as in the dense case, such a sequence is given by the
natural random construction.

\begin{lemma}\label{lrgcut}
Let $p=p(n)$ satisfy $np\to\infty$,
let $C>0$ be constant, let $\ka:[0,1]^2\to [0,C]$ be a bounded kernel,
and let $G_n=G_p(n,\ka)$.
Then $\dc(G_n,\ka)\to 0$ almost surely.
Also, the sequence $(G_n)$ satisfies the bounded density assumption~\ref{AC} with probability $1$. 
\end{lemma}
\begin{proof}
The second statement is essentially immediate from Chernoff's inequality,
constructing $G_n$ as a subgraph of
the Erd\H os--R\'enyi random graph $G(n,Cp)$; it
also follows from the first statement and Lemma~\ref{kbd}.

We now turn to the proof that $\dc(G_n,\ka)\to 0$.
Recall that $\ka$ is of finite type if $[0,1]$ may be partitioned
into sets $A_1,\ldots,A_k$ so that $\ka$ is constant on each rectangle
$A_i\times A_j$. We first suppose that $\ka$ is of finite
type. Rearranging $\ka$, and ignoring parts with measure zero, we
may assume that each $A_i$ is an interval with positive measure.
Recall that $G_n=G_p(n,\ka)$ is constructed by first choosing
the `types' $x_1,\ldots,x_n$ of the vertices independently and uniformly
at random from $[0,1]$. Let $n_i$ denote the number of vertices of type
$i$, noting that we have $n_i\sim \mu(A_i)n$ a.s.
Let us adjust the intervals $A_i$ slightly, replacing $A_i$ by
a set $A_i'$ ($=A_i'(n)$) with measure $n_i/n$. Let $\ka'=\ka'(n)$
be the adjusted kernel, taking on $A_i'\times A_j'$ the value
that $\ka$ takes on $A_i\times A_j$. Since, a.s., we adjust
the length of each $A_i$ by $o(1)$, the kernels $\ka'$ and $\ka$
differ on a set of measure $o(1)$. Since each is bounded, it follows that
\begin{equation}\label{kaka'}
 \cn{\ka-\ka'}\le \on{\ka-\ka'}\to 0
\end{equation}
a.s., as $n\to\infty$.

Given $x_1,\ldots,x_n$, let $G'$ be the weighted graph in which
each edge is present and has weight $w_{ij}=p\ka(x_i,x_j)$. Then, relabelling
the vertices so that those with $x_i=k$ correspond to the set $A_k'$,
we see that $\ka_{G'}=\ka'$. The graph $G_n$
may be constructed from $G'$ by simply selecting each edge $ij$
independently, with probability equal to its weight in $G'$. As noted earlier, for a kernel
corresponding to a (weighted) graph, the cut norm (defined
by~\eqref{cndefST}) is realized by a cut corresponding to a partition
of the vertex set, so
\[
 \cn{\ka_{G_n}-\ka'} = \cn{\ka_{G_n}-\ka_{G'}} = \max_{S\subset V(G_n)}
 \left|\frac{e_{G_n}(S,S^\cc)-\sum_{i\in S,\,j\in S^\cc}w_{ij}}{n^2p}\right|.
\]
Having conditioned on $x_1,\ldots,x_n$,
for each $S$ the random variable $X=e_{G_n}(S,S^\cc)$ has mean exactly
$\sum_{i\in S,\,j\in S^\cc}w_{ij}$. Furthermore, $\E(X)=O(n^2p)$. Since $X$ is a sum
of independent indicator variables, it follows from (for example) the Chernoff bounds,
that for any $\eps>0$ we have $\PR(|X-\E(X)|\ge \eps n^2p) \le \exp(-c_\eps n^2p)$
for some $c_\eps>0$. Since $n^2p=\omega(n)$, this probability decays superexponentially.
Since there are only $2^n$ sets $S$ to consider, we see that
$\PR(\cn{\ka_{G_n}-\ka'}\ge \eps)$ decays superexponentially as $n\to\infty$.
Since $\eps>0$ was arbitrary, using \eqref{kaka'} it follows that $\dc(G_n,\ka)\to 0$ a.s.

So far we assumed that $\ka$ was of finite type. Given an arbitrary $\ka$,
for each $\eps>0$ we can find a finite type approximation $\ka_\eps$ to $\ka$
with
\[
 \cn{\ka_\eps-\ka} \le \on{\ka_\eps-\ka} \le \eps;
\]
see, for example, Lemma~\ref{approx}.
One can couple the random graphs $G_n=G_p(n,\ka)$ and $G_n'=G_p(n,\ka_\eps)$
using the same vertex types $x_1,\ldots,x_n$ for each, in such a way
that the symmetric difference $G_n\diff G_n'$ has the distribution
of $G_p(n,\Delta\ka)$, where $\Delta\ka(x,y)=|\ka(x,y)-\ka_\eps(x,y)|$.
The expected number of edges of $G_p(n,\Delta\ka)$ is at most $n(n-1)p\on{\Delta\ka}/2$
(with equality if $p\Delta\ka\le 1$),
which is at most $n^2p\eps/2$. It is easy to check that the actual number is tightly
concentrated about the mean, so 
\[
 \dc(G_n,G_n') \le \on{\ka_{G_n}-\ka_{G_n'}} =\frac{2e(G_n\Delta G_n')}{n^2p} \le 2\eps
\]
holds with probability tending (rapidly) to $1$ as $n\to\infty$.
Using the finite-type case to show that $\dc(G_n',\ka_\eps)\to 0$
and the bound $\dc(\ka,\ka_\eps)\le \eps$,
and recalling that $\eps>0$ was arbitrary, the result follows.
\end{proof}

\section{Comparison between cut and count convergence}\label{sec_compar}

Throughout this section, we fix a function $p=p(n)$, and consider
sequences $(G_n)$ of graphs with $|G_n|=n$. In the dense case, with
$p(n)=1$ for all $n$, Borgs, Chayes, Lov\'asz, S\'os and
Vesztergombi~\cite{BCLSV:1} showed that such a sequence converges to
a kernel $\ka$ in $\dc$ if and only if it converges to $\ka$ in
$\de$; here we wish to investigate whether this result can be
extended to the sparse case. To do this, we first have to make sense
of the definitions. For $\dc$, as in the previous section, we simply
associate a kernel $\kappa_n$ to $G_n$ as before, with $\kappa_n$
taking the values $0$ and $1/p$. Then we use the usual definition of
$\dc$ for (dense) kernels to define $\dc(G_n,G_m)$ and
$\dc(G_n,\ka)$. In the light of Lemma~\ref{dcdcG}, for questions of
convergence the metrics $\dc$ and $\dcG$ are equivalent; we shall
use $\dc$ rather than $\dcG$ in this section.

\subsection{Admissible subgraphs and their counts}

If $p=n^{-o(1)}$, then we use \eqref{ds1} and \eqref{ds2} to define $\de$,
so convergence in $\de$ is equivalent to convergence of
$s_p(F,G_n)$ for every graph $F$.
For smaller $p$, as noted in Subsection~\ref{ss_e}, it makes sense only
to consider graphs $F$ in a certain set $\A$ of {\em admissible} graphs.
It is not quite clear exactly which graphs should be admissible (see
Subsection~\ref{ss_e}), so there are several variants of the definitions.
To keep things simple, we shall work here with one particular choice
for the set $\A$, depending on the function $p$. It may be that
the various conjectures we shall make, if true, extend to larger sets $\A$.

Recall that we write $\F$ for the set of isomorphism classes of finite (simple) graphs.
Given a loopless multi-graph $F$ and an integer $t\ge 1$, let $F_t$ denote the graph obtained
by subdividing each edge of $F$ exactly $t-1$ times, so $e(F_t)=te(F)$
and $|F_t|=|F|+(t-1)e(F)$.
Writing $\Fm$ for the set of isomorphism classes of finite loopless multi-graphs,
for $t\ge 2$ let
\[
 \F_t =\{F_t: F\in \Fm\},
\]
and set $\F_1=\F$ (not $\Fm$). Thus, for $t\ge 2$,
the family $\F_t$ is the set of simple graphs that may be obtained
as follows: starting with a set of paths of length $t$, identify subsets of the
endpoints of these paths in an arbitrary way, except that the two endpoints of the
same path may
not be identified. Note that any $F_t\in \F_t$ has girth at least $2t$.

Similarly, let $\F_{\ge t}$ be the set of simple graphs that may be obtained
as above but starting with paths of length at least $t$. Thus
$\F_{\ge 1}=\F$ and, for $t\ge 2$, $\F_{\ge t}$
is the set of graphs that may be obtained from some $F\in \Fm$ by subdividing
each edge at least $t-1$ times.
Note that $\F=\F_{\ge 1}\supset \F_{\ge 2}\supset \cdots $.
Let $\T$ denote the set of (isomorphism classes of) finite trees.

Throughout this subsection and the next
we suppose that there is some $\alpha>0$ such that
$np\ge n^{\alpha}$ for all large enough $n$. Equivalently,
there is some integer $t\ge 1$ such that
\begin{equation}\label{npt}
 n^{t-1}p^t\ge n^{-o(1)}.
\end{equation}
We shall set
\[
 \A=\T\cup \F_{\ge t}
\]
for the smallest such $t$, noting that if $p=n^{-o(1)}$ then
$t=1$, so all graphs are admissible.
(An alternative that would work just as well is to let $\A$ be the set
of all subgraphs of graphs in $\F_{\ge t}$, which includes $\T$.)
A key observation is that if $F\in \F_{\ge t}$ then (considering
the internal vertices on the paths making up $F$) we have
$|F| > e(F)(t-1)/t$. This also holds if $F\in \T$, or indeed
if $F$ is a subgraph of some $F'\in \A$.
It follows that if $F\subset F'\in \A$ then
\begin{equation}\label{Finf}
 \E\, \emb(F,G(n,p)) \sim n^{|F|}p^{e(F)} = n^{|F|-e(F)(t-1)/t}(n^{t-1}p^t)^{e(F)/t}
 = n^{\Theta(1)-o(1)} \to\infty.
\end{equation}

On the one hand, $\A=\T\cup \F_{\ge t}$ is small enough to satisfy the requirements
for admissibility discussed in Subsection~\ref{ss_e},
including \eqref{Finf}.
(There may be requirements we have missed, in which case
$\A=\T\cup \F_{\ge t}$ for some larger $t$ is likely to work.) On the other hand,
as we shall now see, 
this set $\A$ is large enough to ensure that the counts for $F\in \A$ determine a kernel,
up to the equivalence relation $\sim$ defined in Subsection~\ref{ss_equiv}.

\begin{theorem}\label{kdet}
Let $\ka_1$ and $\ka_2$ be two bounded kernels, and $t\ge 1$ an odd integer.
Suppose that
$s(F,\ka_1)=s(F,\ka_2)$ for every $F\in \F_t$. Then $\ka_1\sim \ka_2$.
\end{theorem}
\begin{proof}
Given a kernel $\ka$, let $\ka^t$ be the kernel defined by
\begin{equation}\label{kat}
 \ka^t(x,y) = \int_{[0,1]^{t-1}} \ka(x,x_1)\ka(x_1,x_2)\cdots\ka(x_{t-1},y)\dd x_1\cdots\dd x_{t-1}.
\end{equation}
In other words, roughly speaking, $\ka^t(x,y)$ counts the number of paths from $x$ to $y$
in $\ka$ with length $t$.
The key observation is that if $F$ is a graph, $\ka$ a kernel, and $t\ge 1$, then
\begin{equation}\label{kkat}
 s(F_t,\ka) = s(F,\ka^t).
\end{equation}
Indeed, $s(F_t,\ka)$ is defined as an integral over one variable for each
vertex of $F_t$. We may evaluate this integral by first fixing the variables
corresponding to vertices of $F$, then using \eqref{kat} once for each edge of $F$
to integrate over the remaining variables. What remains is exactly the integral
defining $s(F,\ka^t)$.

By assumption, $s(F,\ka_1)=s(F,\ka_2)$ for every $F\in \F_t$. Hence,
from \eqref{kkat}, we have $s(F,\ka_1^t)=s(F,\ka_2^t)$ for
{\em every} graph $F$, so, by Theorem~\ref{th_kequiv}
or Theorem~\ref{unique2}, $\ka_1^t\sim \ka_2^t$.
Hence, from \eqref{kup}, there is a kernel $\ka$ and measure-preserving
maps $\sigma_1,\sigma_2:[0,1]\to[0,1]$ such that
$(\ka_i^t)^{(\sigma_i)}=\ka$ a.e., for $i=1,2$.
Since $(\ka_i^{(\sigma_i)})^t=(\ka_i^t)^{(\sigma_i)}$, we thus
have $(\ka_1')^t=(\ka_2')^t$ a.e. for $\ka_i'=\ka_i^{(\sigma_i)}$.
Since $\ka_i'\sim \ka_i$, and our aim is to prove that $\ka_1\sim\ka_2$,
it suffices to prove that $\ka_1'\sim\ka_2'$. Hence, without loss of generality,
we may replace $\ka_i$ by $\ka_i'$, so
we have $\ka_1^t=\ka_2^t$ almost everywhere.
It is now a matter of simple analysis to deduce that $\ka_1=\ka_2$ a.e.

Given a bounded signed kernel, i.e., a bounded function $\ka:[0,1]^2\to \RR$
satisfying $\ka(x,y)=\ka(y,x)$,
let $T_\ka$ be the corresponding operator on $L^2([0,1])$,
defined by
\begin{equation}\label{tkdef}
 (T_\ka f)(x) = \int_0^1 \ka(x,y)f(y) \dd y.
\end{equation}
From the Cauchy--Schwarz inequality we have
\begin{eqnarray*}
 \tn{T_\ka f}^2 &=& \int_0^1 \left(\int_0^1 \ka(x,y)f(y)\dd y\right)^2 \dd x \\
 &\le& \int_0^1 \left(\int_0^1 \ka(x,y)^2\dd y \int_0^1 f(y)^2\dd y\right) \dd x \\
&=& \tn{f}^2 \int\int \ka(x,y)^2\dd x\dd y = \tn{f}^2 \tn{\ka}^2,
\end{eqnarray*}
so the operator norm of $T_\ka$ on $L^2$ satisfies
\begin{equation}\label{nka}
 \norm{T_\ka} \le \tn{\ka}<\infty.
\end{equation}

Now let $\ka$ be any bounded kernel, and $\eps>0$ a real number.
By Lemma~\ref{approx} there is some $k$ such that the kernel $\ka_k$
obtained by averaging $\ka$ over $2^{-k}$-by-$2^{-k}$ squares
satisfies $\tn{\ka-\ka_k}\le \eps$.
Writing $T_\ka=T_{\ka_k}+T_{\ka-\ka_k}$, the first term has finite rank, since
$T_{\ka_k}f$ is constant on intervals of length $2^{-k}$.
From \eqref{nka}, the second term has operator norm at most $\tn{\ka-\ka_k}\le \eps$.
It follows that the image of the unit ball under $T$ can be covered
by a finite number of balls of radius $2\eps$. Since $\eps$ was arbitrary, this shows
that $T_\ka$ is a compact operator.

Since $\ka$ is symmetric, we also have that $T_\ka$ self-adjoint.
Consequently, $T_{\ka_1}$ is a compact self-adjoint operator
on the Hilbert space $L^2(0,1)$, so by standard results (see, for example,
Bollob\'as~\cite{BBLinAnal})
there is an orthonormal basis of eigenvectors
of $T_{\ka_1}$, and all its eigenvalues are real.
It is easy to see that $T_{\ka_1^t}=(T_{\ka_1})^t$, so $T_{\ka_1^t}$
acts on the $\la$-eigenspace of $T_{\ka_1}$ by multiplication by $\la^t$.
Since $t$ is odd (so the map $\la\mapsto \la^t$ is injective), it follows
that $T_{\ka_1^t}$ has the same eigenspaces as $T_{\ka_1}$.
Turning this around, the action of $T_{\ka_1}$ on each eigenspace $E_\la$
of $T_{\ka_1^t}$ with eigenvalue $\la$ is to multiply by $\la^{1/t}$.
Thus, $T_{\ka_1}$ is uniquely determined by $T_{\ka_1^t}$.
In particular, since $\ka_1^t=\ka_2^t$ a.e., the operators $T_{\ka_1}$ and $T_{\ka_2}$
are equal, i.e., $\ka_1=\ka_2$ a.e., as required.
\end{proof}

Note that in Theorem~\ref{kdet} the restriction to odd $t$ is
essential, as shown by the following example.
\begin{example}
Let $\ka_1$ and $\ka_2$ be the two $2$-by-$2$ `chessboard' kernels defined by
\[
 \ka_1(x,y) = \left\{
\begin{array}{ll}
  1 & \hbox{if } x<1/2,\  y<1/2 \hbox{ or } x\ge 1/2,\  y\ge 1/2\\
  0 & \hbox{otherwise}
\end{array}\right .
\]
and
\[
 \ka_2(x,y) = \left\{
\begin{array}{ll}
  1 & \hbox{if } x<1/2,\  y\ge 1/2 \hbox{ or } x\ge 1/2,\  y< 1/2\\
  0 & \hbox{otherwise}
\end{array}\right .
\]
Thus, in the dense case, $\ka_1$ corresponds to the union of two disjoint
complete graphs on $n/2$ vertices, and $\ka_2$ to the complete $n/2$-by-$n/2$ bipartite
graph.
It is easy to check that for any graph $F$ we have $s(F,\ka_1)=2^{1-|F|}$,
while $s(F,\ka_2)=2^{1-|F|}$ if $F$ is bipartite, and $s(F,\ka_2)=0$ otherwise.
In particular, $s(F,\ka_1)=s(F,\ka_2)$ for all bipartite $F$, and
hence for all $F\in \F_t$, $t$ even.
\end{example}

As we saw from Lemma~\ref{kbd} and Corollary~\ref{cSz},
bounded density is a natural
condition to impose on our sequence $(G_n)$ when
dealing with $\dc$ for sparse graphs. In the previous sections, when
dealing with subgraph counts and $\de$, we imposed different conditions,
the closest being Assumption~\ref{AB}. Let us restate this here in the
appropriate form when $p$ need not be as large as $n^{-o(1)}$.

\begin{assumption}[exponentially bounded admissible subgraph counts]\label{ABp}
\ \\
There is a constant $C$ such that, for each fixed $F\in \A$,
we have $\limsup s_p(F,G_n)\le C^{e(F)}$ as $n\to\infty$.
\end{assumption}

Note that we impose a condition only for $F\in \A$.
When comparing $\dc$ and $\de$, we need to impose both Assumption~\ref{AC} (bounded
density) and Assumption~\ref{ABp}.
In the `almost dense' case, when we take $\A=\F$, then Assumption~\ref{ABp} implies
Assumption~\ref{AC}, with the same constant $C$. The argument is based on showing
that a not-too-small dense part of $G_n$ would contain too many $K_{t,t}$s
for some large $t$. Since the details are very similar to the proof of Lemma~\ref{lbd},
we omit them.

Unfortunately, in general neither of Assumptions~\ref{AC} and~\ref{ABp}
implies the other.
In one direction, this is easy to see: simply add a complete graph on $m$
vertices, where $m(n)$ is chosen so that $e(K_m)\sim m^2/2=o(n^2p)$.
This does not affect Assumption~\ref{AC},
but,
if $m$ is chosen large enough, will create too many copies of any fixed connected
graph $F$ with $|F|\ge 3$.
For the reverse direction, consider the following example.

\begin{example}\label{B'C}
Fix a real number $D>1$, and let
$\ka=\ka_D$ be the unbounded kernel defined as follows. First partition $[0,1]$
into intervals $I_1,I_2,\ldots$, so that $I_i$ has length $2^{-i}$.
Then set $\ka(x,y)=i^{2/D}$ if $x,y\in I_i$, and $\ka(x,y)=0$ otherwise.
Let $F$ be a connected graph with average degree at most $D$. Then,
since only terms where all vertices are in the same $I_i$ contribute, we have
\[
 s(F,\ka) = \sum_{i=1}^\infty 2^{-i|F|} i^{2e(F)/D}
 \le \sum_{i=1}^\infty \left(2^{-i}i\right)^{|F|} 
 \le \sum_{i=1}^\infty 2^{-i}i = 2.
\] 
Let $G_n=G_p(n,\ka)$ be the random graph defined from $\ka$ as before.
If every component of any admissible graph has average degree at most $D$,
then it is easy to check that with probability 1 the sequence $(G_n)$
satisfies Assumption~\ref{ABp} (with $C=2$). On the other hand, this sequence
does not satisfy Assumption~\ref{AC}, since, for every $i$,
there will be a subgraph 
of $G_n$ containing a positive fraction of the vertices with density
around $i^{2/D}$.
\end{example}
With the choice of $\A$ made here, whenever $p=p(n)$ does not satisfy
$p=n^{-o(1)}$ then only trees and graphs in some $\F_{\ge t}$, $t\ge 2$,
are admissible. All such graphs, and all their components, have average degree
less than $4$, so the example above shows that in this case,
Assumption~\ref{ABp} does not imply Assumption~\ref{AC}.

Example~\ref{B'C} also shows that, in contrast to the almost dense case (where all graphs
are admissible), in general we cannot tell from the admissible subgraph counts
whether a kernel is bounded.
For this reason, together with those discussed above, when comparing $\dc$ and $\de$
we impose both Assumptions~\ref{AC} and~\ref{ABp}.

\subsection{Conjectured equivalence between cut and count convergence}

Our main conjecture from Section~\ref{sec_hsp} was that, in the sparse case, if the subgraph
counts converge, they converge to those of a kernel. 
In the present setting, we consider counts for admissible subgraphs.
Fix $p(n)$ satisfying \eqref{npt}, and a set $\A$ of admissible graphs.
By default we take $\A=\T\cup \F_{\ge t}$ as in the previous subsection,
although the definitions make sense for other sets $\A$.
Throughout we impose Assumptions~\ref{AC} and~\ref{ABp}
for some fixed constant $C$.
Let $X=[0,\infty)^\A$, let $s_p:\F\to X$ be the map defined by
\[
 s_p(G_n) = (s_p(F,G_n))_{F\in \A} \in X
\]
for any graph $G_n$ with $n$ vertices, let $d$ be any metric
on $X$ inducing product topology, and define $\de$ by mapping to $X$
and then applying $d$; as usual, we suppress the dependence
on the normalizing function $p$. Note that $\de$ is in general a pseudo-metric
rather than a metric: there may be non-isomorphic
graphs $G$, $G'$ with $s_p(F,G)=s_p(F,G')$
for all $F\in\A$. As we only consider questions of convergence
for sequences $G_n$ with $|G_n|\to\infty$, this will not be relevant.

Let $\LL\subset X$ denote the set of possible limit points of
sequences $s_p(G_n)$, where $(G_n)$ satisfies our assumptions.

Recall that we write $\K$ for the space of kernels, that is, symmetric
measurable functions $\ka:[0,1]^2\to [0,C]$ quotiented by
equivalence. There is a natural map from $\K$ into $X$ given
by subgraph counts; we write $s$ for this map, which does not
depend on $p$ (except through the choice of $\A$).
Since $\A$ always contains some set $\F_{\ge t}$, and hence some $\F_{t'}$
with $t'$ odd, Theorem~\ref{kdet} tells us that this map is injective.

Our main conjecture is the following.
\begin{conjecture}\label{mqs}
With the assumptions and definitions above, we have
\begin{equation}\label{ceq}
 \LL \subset s(\K).
\end{equation}
\end{conjecture}
Note that if $t=1$ then $p=n^{-o(1)}$ and we recover Conjecture~\ref{q1a};
Conjecture~\ref{mqs} seems to be the natural extension of Theorem~\ref{ctok}
to functions $p=p(n)$ with $p\to 0$ but $np\ge n^{\alpha}$ for some $\alpha>0$.

Turning to the equivalent of Theorem~\ref{dedc},
we believe that in this setting
the notions of convergence given by $\de$ and $\dc$ are equivalent.
The most concrete way of saying this is as follows; again we
take $\A=\T\cup \F_{\ge t}$ by default, although it might be 
that the conjecture fails for this $\A$ but holds for some other $\A$.
\begin{conjecture}\label{cce}
Let $(G_n)$ be a sequence satisfying Assumptions~\ref{AC} and~\ref{ABp},
and let $\ka\in \K$.
Then $\dc(G_n,\ka)\to 0$ if and only if $s_p(G_n)\to s(\ka)$.
\end{conjecture}

In this form, the conjecture implies \eqref{ceq} (see below). Without assuming
\eqref{ceq}, it still makes sense to compare the notions of Cauchy sequences instead.
\begin{conjecture}\label{ccc}
Let $(G_n)$ be a sequence satisfying Assumptions~\ref{AC} and~\ref{ABp}. Then $(G_n)$
is Cauchy with respect to $\dc$ if and only if $(G_n)$ is Cauchy with respect
to $\de$.
\end{conjecture}

As we shall shortly see, Conjectures~\ref{cce} and \ref{ccc}
are equivalent.

Although we cannot prove the conjectures above, we can say something. Conjecture~\ref{cce},
for example, asserts two implications. Surprisingly, it is easy to show that,
if \eqref{ceq} holds, then
either of these implications (for all sequences, not just a particular sequence) implies
the other! To prove this we shall first show
that the random graph $G(n,\ka)$ behaves `correctly' with respect
to our definition of $\de$; the corresponding result for $\dc$ is
Lemma~\ref{lrgcut}.

\begin{lemma}\label{lrgsub} 
Fix $C>0$, let $\ka:[0,1]^2\to [0,C]$ be a bounded kernel, and let $G_n=G_p(n,\ka)$.
Then, with probability 1, the sequence $(G_n)$ satisfies
Assumption~\ref{ABp} and we have $s_p(G_n)\to s(\ka)$.
\end{lemma}
\begin{proof}[Outline proof]
The first statement follows from the second, since $s(F,\ka)\le C^{e(F)}$ holds
for every $F$, and in particular for $F\in \A$.

It is well known that if $F$ is a fixed graph, and $p'=p'(n)$ is a function of $n$,
then the number $X_F$ of subgraphs of $G(n,p')$ isomorphic to $F$
is concentrated about its mean if and only if $\E(X_{F'})\to \infty$ for
every subgraph $F'$ of $F$. (For early results of this type see
Bollob\'as~\cite{BBthr} and Ruci{\'n}ski~\cite{Ruc}; for more
recent, much stronger, results see Janson~\cite{J_papr}
and Janson, Oleszkiewicz and Ruci{\'n}ski~\cite{JOR}.)

Our choice of the set $\A$ ensures that this holds
for every $F\in \F$ with $p'=Cp$ (see \eqref{Finf}), proving the result if $\ka$
is constant. It is straightforward to adapt this result to finite
type $\ka$. It is easy to check that for the $F$ we consider,
any $o(n^2p)$ edges of $G_n\subset G(n,Cp)$
meet $o(n^{|F|}p^{e(F)})$ copies of $F$. Using this observation, one
can approximate the general case by the finite type case
as in the proof of Lemma~\ref{lrgcut}. We omit the details.
\end{proof}

Lemma~\ref{lrgsub} gives us a sequence tending in $\de$ to any $\ka\in \K$.
In other words, it shows that $\LL\supset s(\K)$. Hence, if \eqref{ceq} holds,
\begin{equation}\label{LK}
 \LL = s(\K).
\end{equation}

Let $\J\subset \K\times \LL$ denote the set of pairs $(\ka,\la)\in \K\times \LL$
such that there is a sequence $(G_n)$ satisfying our assumptions
with $\dc(G_n,\ka)\to 0$ and $s_p(G_n)\to \la$. Together, Lemmas~\ref{lrgcut}
and Lemma~\ref{lrgsub} tell us much
more than simply that $\LL\subset s(\K)$: they show that
the `diagonal' $\D=\{(\ka,s(\ka)): \ka\in\K\}$ is contained in $\J$.

At this point, we have established three basic facts:

\medskip
FACT 1: Every subsequence of $(G_n)$  has a subsequence converging
in $\de$ to some point
of $\LL$. This is trivial, since Assumption~\ref{ABp} ensures that
$s_p(G_n)$ lives in a compact subset of $X=[0,\infty)^\A$.

FACT 2: Every subsequence of $(G_n)$ has a subsequence converging in
$\dc$ to some kernel $\ka\in \K$. This is the first part of
Corollary~\ref{cSz}.

FACT 3: The map $s$ is an injection from $\K$ to $\LL$. As noted
above, this follows from Theorem~\ref{kdet}.

\medskip
Facts 1 and 2 tell us that the relationship between the notions of convergence
in $\dc$ and $\de$ is described by the set $\J$. Indeed, any subsequence
of $(G_n)$ itself has a subsequence in which we have convergence in both these
metrics, to some point of~$\J$.

Suppose for the moment that \eqref{LK} holds. There are two possibilities.

If $\J$ is precisely the diagonal $\D$, then the three facts above
easily imply that Conjectures~\ref{cce} and \ref{ccc} both hold.

If $\J\ne \D$, 
then there is some off diagonal point $(\ka_1,\la)$ in $\J$.
Since we are assuming \eqref{LK}, we have $\la=s(\ka_2)$
for some $\ka_2\in \K$.
From the definition of $\J$ there is a sequence $(G_n)$
satisfying our assumptions, with $\dc(G_n,\ka_1)\to 0$ and
$s_p(G_n)\to s(\ka_2)$.
Interleaving the sequence $G_n$ with the sequence $G_p(n,\ka)$, which converges
to $\ka$ in both $\dc$ and $\de$, taking $\ka=\ka_1$ or $\ka_2$,
we find a sequence which converges in one of $\dc$ or $\de$ but not in 
the other. Hence, neither implication in Conjecture~\ref{cce} or \ref{ccc}
holds, i.e., these conjectures fail as badly as possible.

In the light of the comments above, Conjecture~\ref{cce} has the
following rather vague reformulation as a question.

\begin{question}\label{qq}
Given a definition of `suitable' sequences $(G_n)$,
let $\C$ be the set of all graphs $F$ with the property that, whenever $\ka$
is a bounded kernel and $(G_n)$ is a suitable sequence
with $\dc(G_n,\ka)\to 0$, then $s_p(F,G_n)\to s(F,\ka)$.
Under what reasonable definition of `suitable' is the
set $\C$ large enough that the counts $s(F,\ka)$, $F\in \C$,
determine a kernel $\ka$ up to equivalence?
\end{question}

The point is that, if $\C$ is large enough, then the three facts
above hold with $\A=\C$, and we simply use $\C$ as the set
of graphs whose counts we use to define $\de$. Then, for
our `suitable' sequences, $\dc$ convergence
implies $\de$ convergence to the same kernel by definition,
so $(\ka,\la)\in \J$ implies $\la=s(\ka)$. Thus \eqref{ceq} (and hence
\eqref{LK}) holds, and
$\J=\D$, so $\de$ convergence also implies $\dc$ convergence.
Unfortunately, there is no obvious single choice for the
set of suitable sequences. One could hope that
sequences with bounded density would do, but this is not the case:
by adding a complete graph with many (but still $o(pn^2)$)
edges to $G(n,p)$, say,
it is easy to check that in this case $\C$ consists only of matchings.
Conjecture~\ref{cce} is more specific than Question~\ref{qq}, since
we define `suitable' by assuming $s_p(F,G_n)$ bounded
for $F$ in some set $\A$, and then require $\C\supset \A$.

\bigskip
If Conjecture~\ref{mqs} does not hold, then Conjectures~\ref{cce}
and \ref{ccc} cannot hold. Indeed, there is some $\la\in \LL$ not corresponding
to a kernel. Taking $G_n$ converging to $\la$ in $\de$, and then a subsequence
that converges in $\dc$, there is some $\ka$ with $(\ka,\la)\in \J$.
Interleaving a corresponding sequence $(G_n)$ with $G_p(n,\ka)$,
we find a sequence that converges in $\dc$ but not in $\de$.

Even if Conjecture~\ref{mqs} does not hold,
it is still possible that there is some relationship between cut and subgraph convergence:
it may be
that every sequence that is Cauchy with respect to $\de$, and hence converges
to some $\la\in \LL$, is Cauchy with respect to $\dc$, i.e., converges
to some $\ka\in \K$. This happens if and only if, for every $\la\in \LL$,
there is a unique $\ka\in \K$ such that $(\ka,\la)\in \J$.
This is not as implausible as it may sound.
Indeed, suppose Conjecture~\ref{cce} holds for some admissible set $\A_-$,
but that the definitions involved
make sense for a larger set $\A_+$. It may be that \eqref{ceq}
fails working with $\A_+$, because we are now allowing as admissible
some counts which need not converge to what we expect.
However, there is a restriction map from $\LL_+$ to $\LL_-$ forgetting
about the counts outside $\A_-$. Since \eqref{LK} holds for the smaller set
of admissible graphs, this would show that for the larger set
there is only one $\ka$ for each $\la$, but not vice versa.

\bigskip
In the next section we shall prove a form of Conjecture~\ref{cce}.
Before doing so, let
us briefly compare this conjecture with
the corresponding result of Borgs, Chayes, Lov\'asz, S\'os and Vesztergombi~\cite{BCLSV:1}
for the dense case.
In the dense setting, as here, Facts 1 and 2 above
are easy to prove. That all limiting counts come from kernels was shown
by Lov\'asz and Szegedy~\cite{LSz1}; this gives \eqref{LK}.
Surprisingly, the hard part is proving Fact 3,
that the counts (now meaning all counts)
determine the kernel, up to equivalence as defined in Subsection~\ref{ss_equiv}.
(For us this was easy, since we deduced the sparse
equivalent of this statement from the dense result, Theorem~\ref{th_kequiv}.) Once one knows
that the counts determine the kernel,
the `meta-argument'
above shows that $\dc$ convergence implies $\de$ convergence if and only
if the reverse implication holds. Since the forward implication is very easy
(see Corollary~\ref{cor_dcde}),
the result of~\cite{BCLSV:1} that the two metrics are equivalent follows.
This gives a proof of this result in which the only non-straightforward
step is showing that the counts $s(F,\ka)$ determine the kernel $\ka$
up to the appropriate notion of equivalence. One might expect this uniqueness
result to be easy, but this seems to be far from the case.
Recently, Borgs, Chayes and Lov\'asz~\cite{BCL:unique} gave a direct proof
of this result (which, as noted in Section~\ref{dense}, actually follows
from the results of~\cite{BCLSV:1}); their proof is far from simple.

\subsection{Partial results: embedding lemmas}\label{sec_em}

Our aim in this section is to prove a positive result, that under certain
circumstances, if $\dc(G_n,\ka)\to 0$, then $s_p(F,G_n)\to s_p(F,\ka)$
for certain graphs $F$. In the case where $\ka$ is of finite type,
this is simply a counting lemma: in this case, $G_n\to \ka$
says that $G_n$ can be partitioned into $(\eps,p)$-regular pairs
with densities given by $\ka$. In the uniform case,
Chung and Graham~\cite{CG} proved such counting lemmas for certain graphs
under certain assumptions.
The general case turns out to be rather different, but we shall
still use several of their ideas.

We start with the simplest case, where $F$ is a path. First we need
some definitions. As usual, in the proof it will be easier
to consider homomorphisms from $F$ to $G_n$ (i.e., walks in $G_n$) 
rather than embeddings. As we shall see later, this makes no difference.

For $G_n$ a graph and $X_0,\ldots,X_\ell$ subsets of $V(G_n)$, let
$G_n(X_0,X_1,\ldots,X_\ell)$ denote the number of $(\ell+1)$-tuples
$(v_i)$ with $v_i\in X_i$ and $v_iv_{i+1}\in E(G)$ for  $0\le i\le \ell-1$.
Identifying a subset of $V(G_n)$ with a subset of $[0,1]$ as before, 
for a kernel $\ka$ let
\[
 \ka(X_0,X_1,\ldots,X_\ell) = \int_{X_0\times\cdots\times X_\ell}
 \ka(x_0,x_1)\cdots\ka(x_{\ell-1},x_\ell) \dd x_0\cdots \dd x_{\ell}.
\]

\begin{lemma}\label{pem}
Let $C>0$ be constant, let $p(n)$ be any function of $n$ with $np\to\infty$,
and let $(G_n)$ be a sequence of graphs with $t_p(T,G_n)$ bounded
for each tree $T$. For every $\eps>0$ and $\ell\ge 1$
there is a $\delta=\delta_\ell(\eps)>0$ such that, whenever
$\ka:[0,1]^2\to [0,C]$ is a kernel with $\cn{\ka_{G_n}-\ka}\le\delta$,
then
\[
 \bm{G_n(X_0,X_1,\ldots,X_\ell) - n^{\ell+1}p^\ell\ka(X_0,X_1,\ldots,X_\ell)}
  \le \eps n^{\ell+1}p^\ell
\]
for any sets $X_0,X_1\ldots,X_\ell\subset V(G_n)$.
\end{lemma}
Roughly speaking, the lemma says that if $G_n\to \ka$ and $t_p(T,G_n)$
is bounded for each $T$, then $t_p(P_\ell,\ka)\to s(P_\ell,\ka)$.
The stronger assertion makes it simpler to prove the result by induction.
\begin{proof}
Renormalizing, we may assume without loss of generality that $C=1$. Let
us do so from now on.

The fact that $\delta$ is not allowed to depend on $\ka$ allows
us to assume without loss of generality that $\ka$ is piecewise
constant on squares of side $1/n$, i.e., that $\ka$
may be interpreted as a (dense) weighted graph
with vertex set $V(G_n)$. Indeed, the Frieze--Kannan form
of Szemer\'edi's Lemma shows that there is an integer $k$ such 
that, given any $\ka$, there is a $\ka'$ that is constant on squares
of side $1/k$ with $\dc(\ka,\ka')\le \delta$. Tweaking
$\ka'$ slightly if $k$ does not divide $n$, we obtain
a kernel $\ka''$ of the required form. Replacing $\delta$
by $2\delta$ as appropriate, the result for $\ka$ follows from the result
for $\ka''$. [Note that we implicitly assumed that $n$
is large here, meaning larger than some $n_0$ depending on $\eps$ and $\ell$.
We could simply assume this in the statement of the lemma,
but it can be achieved by subdividing vertices.
In fact, we could work with a kernel instead of a graph throughout the proof.]

Let
\[
 \Delta(X_0,\ldots,X_\ell) = 
 \frac{G_n(X_0,X_1,\ldots,X_\ell)}{n^{\ell+1}p^\ell} -
 \ka(X_0,X_1,\ldots,X_\ell),
\]
so our aim is to show that $|\Delta(X_0,\ldots,X_\ell)|\le \eps$
for all choices of the sets $X_i$. We shall show much more:
let $M=\max_T \sup_n t_p(T,G_n)$, where the maximum is over trees
with at most $2\ell+1$ vertices, noting that $M<\infty$.
We shall show
that if $\dc(G_n,\ka)\le \delta$, then, for any $1\le t\le \ell$
and any $X_0,\ldots,X_t\subset V(G_n)$ we have
\begin{equation}\label{ih}
 |\Delta(X_0,X_1,\ldots,X_t)| \le \eps_t,
\end{equation}
where $\eps_1=\delta$, and
\[
 \eps_t= 7\sqrt{\eps_{t-1}}+\sqrt{M}\eps_{t-1}^{1/4}
\]
for $t\ge 2$.
Since $\eps_\ell$ tends to zero as $\delta\to 0$,
taking $\delta$ small enough we have $\delta=\eps_1\le \eps_2\le \cdots \eps_\ell \le \eps$,
so to complete the proof of the lemma it suffices to prove \eqref{ih}
for this choice of $\delta$.

We shall prove \eqref{ih} by induction on $t$. For $t=1$, the result is immediate
from the definition
of the cut norm: indeed, $\Delta(X_0,X_1)$ is one of the
quantities appearing in the supremum defining this norm.
Suppose now that $2\le t\le\ell$, and that \eqref{ih} holds
with $t$ replaced by $t-1$.

For $v\in V(G)$ and $X_1,\ldots,X_r\subset V(G)$, set
\[
 \ka(v,X_1,\ldots,X_r) = \int_{X_1\times\cdots\times X_r}
 \ka(x,x_1)\ka(x_1,x_2)\cdots\ka(x_{r-1},x_r) \dd x_1\cdots \dd x_r,
\]
where $x$ is any point of the interval of length $1/n$
corresponding to the vertex $v$, and let
\begin{equation}\label{Dv}
 \Delta(v,X_1,\ldots,X_r) = 
 \frac{G_n(\{v\},X_1,\ldots,X_r)}{n^r p^r} -
 \ka(v,X_1,\ldots,X_r).
\end{equation}
Note that
\begin{equation}\label{dsum}
 \Delta(X,X_1,\ldots,X_r) = \frac{1}{n}\sum_{v\in X}\Delta(v,X_1,\ldots,X_r).
\end{equation}
Fix $X_0,\ldots,X_t\subset V(G_n)$, and set $\eta=\sqrt{\eps_{t-1}}$.
Let $B_1$ be the set of $v\in X_1$ with $\Delta(v,X_2,\ldots,X_t)>\eta$.
Then, from \eqref{dsum}, $\Delta(B_1,X_2,\ldots,X_t)\ge \eta |B_1|/n$.
But by the induction hypothesis, $\Delta(B_1,X_2,\ldots,X_t)\le \eps_{t-1}=\eta^2$.
Hence, $|B_1|\le \eta n$. Arguing similarly, and using $\eps_{t-1}\ge \eps_1$,
we see that the set $B$ of vertices $v\in X_1$ for which
either $|\Delta(v,X_2,\ldots,X_t)|\ge \eta$ or $|\Delta(v,X_0)|\ge \eta$
holds has size at most $4\eta n$.

If $v\in X_1\setminus B$, then we have roughly the right number of
walks through $v$, i.e.,
\[
 G_n(X_0,\{v\},X_2,\ldots,X_t) = G_n(\{v\},X_0) G_n(\{v\},X_2,\ldots,X_t)
\]
is close to $np \ka(v,X_0) n^{t-1}p^{t-1}\ka(v,X_2,\ldots,X_t)$.
More precisely, using the fact that
$\ka$ is pointwise bounded by $C=1$ to bound the $\ka$ terms in the
last expression by $1$, for $v\in X_1\setminus B$ we have
\begin{equation}\label{good}
 |\Delta(X_0,v,X_2,\ldots,X_t)| \le 3\eta,
\end{equation}
where the left hand side is defined by analogy with \eqref{Dv}.

It remains to consider $v\in B$. For $i=1,2$, let
\[
 \sigma_i = \sum_{v\in B} G_n(X_0,\{v\},X_2,\ldots,X_t)^i,
\]
noting that $\sigma_1\le \sqrt{|B|\sigma_2}$ by the Cauchy--Schwarz inequality.
Let $T$ be the tree with $2t$ edges formed by identifying the second vertices
of two paths of length $t$. Then $\sigma_2$ counts a subset of the homomorphisms from $T$ into $G_n$, so
\[
 \sigma_2 \le \hom(T,G_n) = n^{2t+1}p^{2t} t_p(T,G_n) \le Mn^{2t+1}p^{2t}.
\]
Since $|B|\le 4\eta n$ it follows that
\[
 \sigma_1\le \sqrt{|B|\sigma_2} \le 2\sqrt{M\eta} n^{t+1}p^t.
\]
Since $\ka$ is bounded by $1$, we have 
$\ka(X_0,B,X_2,\ldots,X_t)\le \mu(B)\le 4\eta$, so
\[
 |\Delta(X_0,B,X_2,\ldots,X_t)|\le 2\sqrt{M\eta}+4\eta.
\]
Together with the bound \eqref{good} for $v\in X_1\setminus B$
and (the equivalent of) \eqref{dsum},
this implies that
\[
 |\Delta(X_0,X_1,\ldots,X_t)|\le 7\eta +2\sqrt{M\eta} = \eps_t,
\]
as required. This completes the proof of \eqref{ih} by induction,
and thus the proof of the lemma.
\end{proof}

Note that the argument above works just as well for an arbitrary
fixed tree rather than a path: we pick some leaf $v$ to play the role
of $x_0$; the unique neighbour of $v$ then plays the role
of $x_1$. This gives us a counting
lemma for trees.

\begin{corollary}\label{cT}
Let $(G_n)$ be a sequence of graphs with $t_p(T,G_n)$ bounded for
every tree $T$, and suppose that $\dc(G_n,\ka)\to 0$,
where $\ka$ is a bounded kernel. Then for each
tree $T$ we have $t_p(T,G_n)\to s(T,\ka)$ as $n\to\infty$.
\noproof
\end{corollary}

Chung and Graham~\cite{CG} proved a version of this result (for
paths rather than trees) with $\ka$ constant, under the assumption
that the maximum degree of $G_n$ is at most $Cpn$. This maximum
degree assumption of course gives $t_p(T,G_n)\le C^{e(T)}$, so it is
stronger than the bounded tree counts assumption of Lemma~\ref{pem}.
In some sense, the maximum degree condition is much stronger, but it
turns out that our global assumption is just as good for questions
involving subgraph counts. The reason that Lemma~\ref{pem} is more
complicated than the corresponding simple result in~\cite{CG} is
that $\kappa$ is not uniform, not our weaker assumption.

We stated earlier that, in the sparse case, the parameter
$s_p(F,\ka)$ should be preferred to $t_p(F,\ka)$, even though $t_p$
tends to be easier to work with. Nevertheless, in the case of trees,
these parameters are equivalent, as shown by the following
observation.

\begin{lemma}\label{stT}
Let $p(n)$ be any function of $n$ with $np\to \infty$, and let $(G_n)$
be a sequence with $s_p(T,G_n)$ bounded for every tree $T$.
Then, for each tree $T$, we have $t_p(T,G_n)\sim s_p(T,G_n)$.
In particular, $t_p(T,G_n)$ is bounded.
\end{lemma}
\begin{proof}
Fix a tree $T$ with $k$ vertices. It suffices to show that the number $N_T$ of
non-injective homomorphisms from $T$ to $G_n$ satisfies $N_T=o(n^k p^{k-1})$
as $n\to\infty$.
Now the image of any non-injective homomorphism $\phi$ from $T$ to $G_n$ is a 
connected subgraph $H$ of $G_n$ with $\ell$ vertices, where $1\le \ell\le k-1$.
Any such subgraph contains a tree $T'$ with $\ell$ vertices, so for each $\ell$ there
are (crudely) at most $\sum_{|T'|=\ell} \emb(T',G_n)$ possibilities
for vertex set of $H$, where the sum is over all trees $T'$ with $\ell$ vertices.
Since there are at most $k^\ell$ homomorphisms $\phi$ with image
a given set of $\ell$ vertices, we thus have
\[
 N_T \le \sum_{\ell=1}^{k-1} k^\ell \sum_{|T'|=\ell} \emb(T',G_n).
\]
Since $\emb(T',G_n)=n_{(|T|')}p^{e(T')}s_p(T',G_n)$, the final
term is $O(n^\ell p^{\ell-1})$ by assumption. It follows that
$N_T=O(n^{k-1}p^{k-2}) = o(n^k p^{k-1})$, as claimed.
\end{proof}
Lemma~\ref{stT} allows us to restate Corollary~\ref{cT} in terms
of the parameter $s$.

\begin{theorem}\label{th_Temb}
Let $(G_n)$ be a sequence of graphs with $s_p(T,G_n)$ bounded for
every tree $T$, and suppose that $\dc(G_n,\ka)\to 0$,
where $\ka$ is a bounded kernel. Then for each
tree $T$ we have $s_p(T,G_n)\to s(T,\ka)$ as $n\to\infty$.
\noproof
\end{theorem}

Theorem~\ref{th_Temb} may be regarded as an embedding lemma for trees.
Our next aim is to prove a much more general result.
Chung and Graham showed that, in the uniform case,
if the number of paths
of length $\ell-1$ between any two vertices is at most a constant
times what it should be, then almost all pairs of vertices
are joined by almost the right number of paths of length $\ell$,
and hence $G_n$ contains asymptotically the expected number of copies
of any $F\in \F_\ell$. This result is much harder than the paths
result, even in the uniform case. 
Although we shall use the key idea of Chung and Graham,
the proof does not carry over in a simple way.
In the following result, we work with $t_p$ rather than $s_p$
for simplicity; we return to this later.

\begin{theorem}\label{cthm}
Let $C>0$ and $\ell\ge 3$ be fixed, and let $p=p(n)$ be any
function of $n$. Let $(G_n)$ be a sequence of graphs with
$\sup_n t_p(F,G_n)<\infty$ for each $F\in \T\cup\F_\ell\cup\{C_{2\ell-2}\}$,
and suppose that $\dc(G_n,\ka)\to 0$ for some kernel
$\ka:[0,1]^2\to [0,C]$. Then
$t_p(F,G_n)\to s(F,\ka)$ for each $F\in \F_\ell$.
\end{theorem}
\begin{proof}
Note that by Lemma~\ref{kbd}, the sequence $(G_n)$ has density
bounded by $C$, i.e., it satisfies Assumption~\ref{AC}.
Renormalizing, we shall assume without loss of generality that $C=1$.

Fix $\eps>0$, and a graph $F_\ell\in \F_\ell$.
Let $\eta>0$ be a small constant to be chosen below (depending
on $\eps$, $\ell$ and $F_\ell$). 
By Lemma~\ref{lSzs} there is some $K$ such that
for $n$ large enough, which we assume from
now on, $G_n$ has an $(\eta,p)$-regular partition
$\Pi=(P_1,\ldots,P_k)$ for some $k=k(n)\le K$. Passing to a subsequence
of $(G_n)$, we may assume that $k$ is constant.
As usual, we shall ignore rounding to integers,
assuming that each $P_i$ contains exactly $n/k$ vertices.

Passing to a subsequence (again), we may assume that
for all $i$ and $j$ the sequence $d_p(P_i,P_j)$
converges to some $\ka'(P_i,P_j)\in [0,1]$.
Relabelling if necessary so that $P_i$ consists of vertices $v$
with $in/k<v\le (i+1)n/k$, and identifying vertices with corresponding
subsets of $[0,1]$ as usual, we may view $\ka'$ as a kernel on $[0,1]^2$.

If $n$ is large enough, which we assume, then each $d_p(P_i,P_j)$ 
is within $\eps$ of $\ka(P_i,P_j)$. It follows that $\dc(G_n/\Pi,\ka')
\le \on{G_n/\Pi-\ka'} \le\eps$. Under our bounded density assumption~\ref{AC},
strong regularity
implies weak regularity (for suitably transformed parameters),
so choosing $\eta$ small enough we have
$\dc(G_n,G_n/\Pi)\le \eps$. Hence, choosing $n$ large enough,
$\dc(\ka,\ka')\le \dc(\ka,G_n)+\dc(G_n,G_n/\Pi)+\dc(G_n/\Pi,\ka')\le 3\eps$.
Hence, by Lemma~\ref{cs},
for any fixed $F$ we have
\[
 |s(F,\ka)-s(F,\ka')| = O(\eps),
\]
so it suffices to show that $t_p(F_\ell,G_n)$ is close to $s(F_\ell,\ka')$
rather than to $s(F_\ell,\ka)$. To avoid clutter in the notation, from now on 
we write $\ka$ for the finite type kernel $\ka'$ defined above;
the original $\ka$ plays no further role in the proof. Recall that
$\ka$ (formerly known as $\ka'$) is bounded by $1$.
For $u\in P_i$ and $v\in P_j$
we shall abuse notation by writing $\ka(u,v)=\ka(P_i,P_j)$
for the value of $\ka$ at any point of $[0,1]^2$ corresponding to $(u,v)$.
Recall that $|d_p(P_i,P_j)-\ka(P_i,P_j)|\le \eps$ for all $i,j$.

For $v,w\in V(G_n)$ and $t\ge 1$, let $w_t(v,w)$ denote the number of walks
of length
$t$ in $G_n$ starting at $v$ and ending at $w$; we suppress
the dependence on $G_n$ in the notation. Let $\ka^t(v,w)$
denote the normalized `expected' number of such walks, if $G_n$ behaved
like the random graph $G_p(n,\ka)$.
Let $U\subset V^2$ be the set of pairs $(v,w)$ such that
$w_\ell(v,w) \le (\ka^\ell(v,w)-\eps) n^{\ell-1}p^\ell$.
We call the pairs $(v,w)\in U$ {\em underconnected}, since they
are joined by `too few' walks of length $\ell$.
We shall show that
\begin{equation}\label{aim}
 |U| = \bm{\{(v,w): w_\ell(v,w) \le (\ka^\ell(v,w)-\eps)n^{\ell-1}p^\ell\}}
  \le \eps n^2
\end{equation}
if $\eta$ is chosen suitably,
and then $n$ is taken large enough. Before doing so, let us
note that this implies the result.

By Lemma~\ref{pem}, if we choose $\eta$ small enough, then
the total number of walks of length $\ell$
in $G_n$ is within $\eps n^{\ell+1}p^\ell$
of the expected number in $G_p(n,\ka)$, namely
$\on{\ka^\ell}n^{\ell+1}p^\ell$. If \eqref{aim} holds,
then if we count only a maximum of $\ka^\ell(v,w)n^{\ell-1}p^{\ell}$
walks for each pair $(v,w)$ of endpoints, we still count at least
$(1-\eps)(\on{\ka^\ell}-\eps)n^{\ell+1}p^\ell$ walks, so there are at
most $3\eps n^{\ell+1}p^\ell$ walks uncounted,
using $\on{\ka^\ell}\le 1$.
Writing
$W$ for the set of {\em overconnected} pairs $(v,w)\in V^2$ with
$w_\ell(v,w)\ge (\ka^\ell(v,w)+\sqrt{\eps})n^{\ell-1}p^\ell$, it follows
that
\begin{equation}\label{Waim}
 |W|\le 3\sqrt{\eps} n^2.
\end{equation}
In other words, almost all pairs of
vertices are joined by almost the right number of walks.

Recall that we fixed a graph $F_\ell\in \F_\ell$. Let $F_\ell$ be obtained
by subdividing the edges of a loopless multi-graph $F$ with vertex set
$u_1,\ldots,u_r$, so
\begin{equation}\label{hft}
 \hom(F_\ell,G_n) = \sum_{v_1,\ldots,v_r\in V(G_n)} \prod_{u_iu_j\in E(F)} w_\ell(v_i,v_j),
\end{equation}
where the factors in the product corresponding to multiple edges of $F$
are of course repeated.
Given $u_iu_j\in E(F)$, let $2F_\ell/E_2$ be the graph formed from two
copies of $F_\ell$ by identifying the vertices corresponding to $u_i$
and identifying the vertices corresponding to $u_j$.
Since $2F_\ell/E_2\in \F_\ell$,
we have $t_p(2F_\ell/E_2,G_n)$ bounded. It follows by the Cauchy--Schwarz
inequality
that the number of homomorphisms from $F_\ell$ into $G_n$ mapping
$u_i$ and $u_j$ to a pair in $U\cup W$ is small, in fact of order
$\eps^{1/4}n^{|F_\ell|}p^{e(F_\ell)}$; the argument is as in the
proof of Lemma~\ref{pem}.

Since the comment above applies to any edge $u_iu_j$ of $F$,
the contribution to the sum in \eqref{hft} from terms in which one
or more pairs $(v_i,v_j)$ fall in $U\cup W$ is small. But in the remaining
terms,
$w_\ell(v_i,v_j)$ is well approximated by $\ka^\ell(v_i,v_j)n^{\ell-1}p^\ell$,
and it follows
that $t_p(F_\ell,G_n)$ is close to $s(F_\ell,\ka)$: the difference
is bounded by some function of $|F_\ell|$ and $\eps$.
In short, we have shown that to prove the theorem, it suffices to prove
\eqref{aim}, i.e., that there are few
underconnected pairs.

{F}rom now on, we forget the original graph $F_\ell$, and
aim to prove \eqref{aim}, recalling that $\ka$ is a fixed
finite-type kernel and that $G_n/\Pi$ is (pointwise)
within $\eps$ of $\ka$, where $\Pi=(P_1,\ldots,P_k)$
is our $(\eta,p)$-regular partition of $G_n$.
It will be convenient to assume that
$\eps$ is fairly small. In particular, we shall assume that
$\eps\le 1/40$. 

Recall that all but at most $\eta k^2$ pairs in our partition $(P_i)_1^k$
are $(\eta,p)$-regular.
Since {\em all} pairs have density at most $1+\eps\le 2$, the irregular pairs
contain at most $2\eta n^2p$ edges. By assumption $t_p(T,G_n)$ is bounded for
each tree $T$, and in particular for the trees formed from two paths
by identifying an edge from each, so using Cauchy--Schwarz again
a small set of edges meets only a small fraction of the
walks of length $\ell$ in $G_n$. In particular, the number of walks
of length $\ell$ containing
one or more edges from irregular pairs is $O(\sqrt{\eta}n^{\ell+1}p^\ell)$.
Taking $\eta$ small enough, we may assume that this quantity is less than $\eps^2 n^{\ell+1}p^\ell/10$, say.
It follows that in proving \eqref{aim},
we may delete all edges in irregular pairs, i.e., we may
assume that every pair is regular: if \eqref{aim} holds for the resulting graph $G_n'$
and kernel $\ka'$ with $\eps/2$ in place of $\eps$, then \eqref{aim} holds for our original
graph $G_n$ and kernel $\ka$.

The lower bound in the proof of Lemma~\ref{pem} used only closeness
of the graph and kernel in the cut norm, not the bounds on various
tree counts. This argument can thus be applied locally to sequences
of parts of our partition.
Abusing notation, let us write $P_0,P_1,\ldots,P_{\ell-1}$
for an arbitrary sequence of $\ell$ parts of our partition, with repetition
allowed. For any subsets $X_i\subset P_i$, we find that
there are at least
\[
 p^{\ell-1} \prod_{i=0}^{\ell-1} |X_i| \, \prod_{i=0}^{\ell-2}\ka(P_i,P_{i+1})
 - \gamma \frac{n^\ell}{k^\ell} p^{\ell-1}
\]
walks $v_0v_1\cdots v_{\ell-1}$ with $v_i\in X_i$, where $\gamma=\gamma(\eta,\ell)$
tends to 0 as $\eta\to 0$.
We choose $\eta$ small enough that $\gamma\le \eps^{12}$.
Taking $X_i=P_i$ for $i>0$,
and summing over all choices for the intermediate parts,
a consequence of this is that if $P_0$ and $P_{\ell-1}$
are any two parts, and $X_0$ is any subset of $P_0$, then
there are at least
\begin{equation}\label{defb}
 (\ka^{\ell-1}(P_0,P_{\ell-1}) |X_0|/|P_0|-\gamma) n^\ell p^{\ell-1}/k^2
\end{equation}
walks of length $\ell-1$ from $X_0$ to $P_{\ell-1}$.

Let us call a walk of length $\ell-1$ in $G_n$ {\em bad} if there are 
at least $M n^{\ell-2} p^{\ell-1}$ walks in $G_n$ with the same endpoints,
where $M$ is a constant to be chosen in a moment, depending
on $\eps$ but not on $\eta$; otherwise, the walk is {\em good}. Each bad walk may 
be extended to at least $M n^{\ell-2} p^{\ell-1}$ homomorphic
images of $C_{2\ell-2}$. By assumption, $t_p(C_{2\ell-2},G_n)$ is bounded,
so it follows that there are $O(n^{\ell}p^{\ell-1}/M)$ bad walks.
In particular, choosing the constant $M$ large enough, we may assume
that there are at most $\eps^9 n^\ell p^{\ell-1}/3$ bad walks.

Suppose for a contradiction that \eqref{aim} does not hold,
i.e., the set $U$ of underconnected pairs of vertices has size at least
$\eps n^2$. Our first aim is to select a pair $(P,P')$ of parts
of our partition such that there are many underconnected
pairs $(u,v)$ in $P\times P'$,
but not too many bad walks start in $P$.
Since $|U|\ge \eps n^2$ by assumption, there are at least $\eps k/2$ parts
$P$ with
\begin{equation}\label{mb}
 |U\cap (P\times V)|\ge \eps n^2/(2k).
\end{equation}
On the other hand, there are at most $\eps k/3$ parts $P$
with the property that more than $\eps^8 n^\ell p^{\ell-1}/k$
bad walks start in $P$ (otherwise there would be too many bad walks).
Hence there exists a part $P$ for which \eqref{mb} holds,
with at most $\eps^8 n^\ell p^{\ell-1}/k$ bad walks starting in $P$.
Fix such a $P$. From \eqref{mb} and averaging, there is 
a part $P'$ such that
\begin{equation}\label{pp'}
 |U\cap (P\times P')| \ge \eps n^2/(2k^2) = \eps |P||P'|/2.
\end{equation}
From now on, fix such a $P'$.

Let us say that a pair $(u,P'')$ with $u\in P$ and $P''$ a part of
our partition is {\em deficient} if there are fewer than
$(\ka^{\ell-1}(P,P'')-\sqrt{\gamma}) n^{\ell-1}p^{\ell-1}/k$
walks of length $\ell-1$ from $u$ to $P''$,
where $\gamma$ is as in \eqref{defb}.
For a given $P''$, at most $\sqrt{\gamma} n/k$ vertices $u\in P$
form a deficient pair with $P''$: otherwise, the set $X_0$ of such
vertices would have more than $\gamma n^\ell p^{\ell-1}/k^2$ fewer
walks to $P''$ than it should have, contradicting \eqref{defb}.
Hence, there are at most $\sqrt{\gamma}n$ deficient pairs.
Let $D\subset P$ be the set of vertices $u$ in more than $\gamma^{1/4}k$
deficient pairs.
Then $|D|\le \sqrt{\gamma}n/(\gamma^{1/4}k)=\gamma^{1/4}|P|$.

Let us say that a pair $(u,P'')$ with $u\in P$ and $P''$ a part
of our partition is {\em compromised} if there are more than
$\eps^3 n^{\ell-1}p^{\ell-1}/k$ bad walks from $u$ to $P''$.
Since at most $\eps^8 n^\ell p^{\ell-1}/k$ bad walks start in $P$,
there are at most $\eps^5 n$ compromised pairs.
Let $C$ be the set of $u\in P$ in more than $\eps^3 k$ compromised
pairs; then $|C|\le \eps ^2 n/k=\eps^2 |P|$.

Let $S\subset P$ be the set of vertices $u$ for which there are at 
least $\eps|P'|/4$ vertices $v\in P'$ with $(u,v)\in U$.
By \eqref{pp'} we have
\[ 
 \eps |P||P'|/2 \le |U\cap (P\times P')| \le |S||P'|+\eps|P||P'|/4,
\]
so $|S|\ge \eps|P|/4>(\gamma^{1/4}+\eps^2)|P|$.
Thus $|S| > |D| + |C|$, and there is some $u$
in $S\setminus (D\cup C)$. Fix such a $u$ for the rest of the proof,
and let $U_u$ denote the set of $v\in P'$ for which $(u,v)$ is underconnected.

At this point we have chosen a vertex $u\in P$, a part $P'$, and
a set $U_u\subset P'$ with the following
properties:

(i) for each $v\in U_u$, there are at most $
(\ka^\ell(u,v)-\eps)n^{\ell-1}p^\ell = (\ka^\ell(P,P')-\eps)n^{\ell-1}p^\ell$
walks of length $\ell$ from $u$ to $v$.

(ii) $|U_u|\ge \eps |P'|/4$,

(iii) there are at most $\gamma^{1/4} k \le \eps^3 k$
deficient pairs $(u,P'')$,

(iv) there are at most $\eps^3 k$ compromised pairs $(u,P'')$.

From (i) and (ii) above, there are at least $m=\eps^2 n^\ell p^\ell /(4k)$
`missing walks' from $u$ to $U_u$: the number of walks of length $\ell$
from $u$ to $U_u$ falls short of the expected number in $G_p(n,\ka)$
by at least $m$.
Let $P''$ be any part of our partition. By a {\em $u$-$U_u$ walk via $P''$}
we mean a walk of length $\ell$ from $u$ to $U_u$ whose second last vertex
lies in $P''$;
the expected number of such walks is 
$N_{P''} = \ka^{\ell-1}(P,P'')\ka(P'',P') |U_u|n^{\ell-1} p^\ell/k$.
Note that $\sum_{P''}N_{P''}$ is simply the expected number of walks
from $u$ to $U_u$.
Let $m_{P''}$ be the number of `missing walks via $P''$', i.e.,
the difference between $N_{P''}$ and the number of $u$-$U_u$ walks
via $P''$, or zero if there are at least $N_{P''}$ such walks.
The total number of missing walks is at most the sum of the
numbers $m_{P''}$, so
\[
 \sum_{P''} m_{P''} \ge m \ge \eps^2 n^\ell p^\ell /(4k).
\]
Let us say that $P''$ is {\em useful}
if $m_{P''}\ge \eps^2 n^\ell p^\ell/(8k^2)$, so the contribution
to the sum above from non-useful parts $P''$ is at most half
the right hand side.
Recalling that we have normalized so that $\ka$ is bounded by $1$,
and that $\eps<1/40$,
for each $P''$ we have $m_{P''}\le N_{P''}\le n^{\ell}p^\ell/k^2$;
it follows that there are at least $\eps^2 k/8\ge 5\eps^3 k$ useful
parts $P''$.

Using (iii) and (iv) above, it follows that
there is a part $P''$
which is useful, but neither deficient nor compromised. 
Fix such a part $P''$.

Recall that a walk of length $\ell-1$ from $u$ to $w\in P''$
is {\em good} if it is not bad, i.e., if 
\begin{equation}\label{goodw}
 w_{\ell-1}(u,w)\le N = M n^{\ell-2} p^{\ell-1}.
\end{equation}
Since $\gamma^{1/4}\le \eps^3$, and
$P''$ is neither deficient nor compromised, there are at least
\[
 (\ka^{\ell-1}(P,P'')-2\eps^3) n^{\ell-1} p^{\ell-1}/k
\]
good walks from $u$ to $P''$. On the other hand, there are many
missing walks via $P''$. With this setup, we are finally ready to apply
the key idea of Chung and Graham~\cite{CG}, which is to partition the set $P''$
into subsets according to the approximate number of walks from $u$
to the relevant vertex, and then use regularity to show
that there are about the right number of walks from $U_u$
to each such subset. In fact, there is a slick way of doing this.

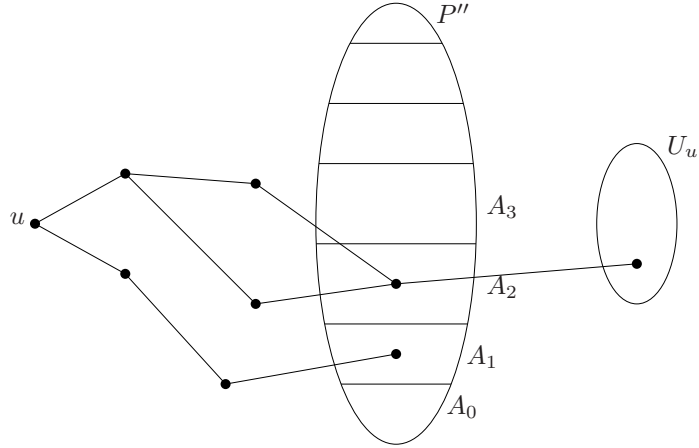
\begin{figure}[htb]
\centering
\input{pathcount.pstex_t}
\caption{The set $P''$ is subdivided into sets $A_i$, with $w_{\ell-1}(u,v)=i$
for each $v\in A_i$. Each edge from $A_i$ to $U_u$ contributes
$i$ walks from $u$ to $U_u$ via $P''$.} \label{fig_pc}
\end{figure}

For $i\ge 0$, let $A_i$
be the set of vertices $v\in P''$ with $w_{\ell-1}(u,v)=i$;
see Figure~\ref{fig_pc}. Also, let $A_i^+=\bigcup_{j\ge i}A_i$.
Then,
\[
 w_{\ell-1}(u,P'')= \sum_{i\ge 0} i|A_i| = \sum_{i\ge 1} |A_i^+|.
\]
More importantly,
$\sum_{i=1}^N |A_i^+|$ is at least the number of {\em good} walks from $u$
to $P''$, so
\begin{equation}\label{abig}
 \sum_{i=1}^N |A_i^+| \ge (\ka^{\ell-1}(P,P'')-2\eps^3) n^{\ell-1} p^{\ell-1}/k.
\end{equation}
Since $(P'',P')$ is $(\eta,p)$-regular with (normalized) density
$\ka(P'',P')\le 1$, if $A\subset P''$ and $B\subset P'$
then $e(A,B)\ge p\ka(P'',P')|A||B|-\eta p (n/k)^2$ (this is trivially true if
one of $A$ or $B$ has size less than $\eta n/k$).
Since each edge from $U_u$ to $A_i$ forms the final edge of exactly $i$
walks from $u$ to $U_u$, the number of walks from $u$ to $U_u$ via $P''$
is given by
\begin{eqnarray*}
 \sum_{i \ge 1} i e(A_i,U_u) 
 &\ge& \sum_{i=1}^N e(A_i^+,U_u) \\
 &\ge& \sum_{i=1}^N p\ka(P'',P')|A_i^+||U_u| -  \eta p (n/k)^2 \\
 &\ge& (\ka^{\ell-1}(P,P'')-2\eps^3)\ka(P'',P')|U_u| n^{\ell-1} p^\ell/k - \eta Np(n/k)^2,
\end{eqnarray*}
where we used \eqref{abig} in the last step.
The main term is simply the expected number of walks from $u$ to $U_u$ 
via $P''$,
so the conclusion is that there are at most
\begin{equation}\label{miss2}
 2\eps^3\ka(P'',P')|U_u|n^{\ell-1}p^\ell/k + \eta Np(n/k)^2
\end{equation}
missing walks from $u$ to $U_u$ via $P''$. The two terms
above may be bounded above by $2\eps^3 n^\ell p^\ell /k^2$
and, recalling \eqref{goodw}, $\eta M n^\ell p^\ell/k^2$, respectively.
Choosing $\eta\le \eps^3/M$ we thus have at most
$3\eps^3 n^\ell p^\ell/k^2$ missing walks via $P''$, i.e.,
$m_{P''}\le 3\eps^3 n^\ell p^\ell/k^2$, which
contradicts the fact that $P''$ is useful. This contradiction completes
the proof.
\end{proof}

Note that the argument above does not extend to $\ell=2$, and
not only because $C_2$ makes no sense. The problem is that
we cannot define $N$ as in \eqref{goodw} (this quantity
is now $o(1)$), but must take $N=1$ instead, and then
the second term in \eqref{miss2} is too large.

The proof of Theorem~\ref{cthm} actually gives rather more with
almost no extra work.

\begin{theorem}\label{cthm2}
Let $C>0$ and $\ell\ge 3$ be fixed, and let $p=p(n)$ be any
function of $n$. Let $(G_n)$ be a sequence of graphs with
$\sup_n t_p(F,G_n)<\infty$ for each $F\in \T\cup\F_{\ge\ell}
\cup\{C_{2\ell-2}\}$,
and suppose that $\dc(G_n,\ka)\to 0$ for some bounded kernel
$\ka$. Then
$t_p(F,G_n)\to s(F,\ka)$ for each $F\in \T\cup\F_{\ge\ell}$.
\end{theorem}
\begin{proof}
The conclusion for $F\in \T$ follows from Corollary~\ref{cT}.

Fix $F\in \F_{\ge\ell}$ and $\eps>0$, and let $L$ be the length
of the longest induced path in $F$.
Noting that for $t>\ell$ we have $C_{2t-2}\in \F_{\ge \ell}$,
the hypotheses of Theorem~\ref{cthm} are satisfied
with $\ell$ replaced by any $t$ in the range $\ell\le t\le L$.
The proof of that result thus shows that if $\eta$ is chosen small enough,
then when we take an $(\eta,p)$-regular partition of $G_n$ with associated
kernel $\ka'$, almost all pairs $(v,w)$ of vertices are joined by
almost the `right' number of walks of each length $t$, $\ell\le t\le L$.
More precisely, writing $\ka$ for $\ka'$
as in the proof of Theorem~\ref{cthm},
and writing $U_t$ for the set of pairs $(v,w)$
with $w_t(v,w)\le (\ka^t(v,w)-\eps)n^{t-1}p^t$ and
$W_t$ for the set of pairs with
$w_t(v,w)\ge (\ka^t(v,w)+\sqrt{\eps})n^{t-1}p^t$,
the proof of Theorem~\ref{cthm} shows that $|U_t|\le \eps n^2$
for $\ell\le t\le L$, and (hence) that $|W_t|\le 3\sqrt{\eps}n^2$
for each $t$ in this range.
Using the analogue of \eqref{hft} in which each term $w_\ell(\cdot,\cdot)$
is replaced by an appropriate term $w_t(\cdot,\cdot)$, as before
we can use the Cauchy--Schwarz inequality to show that the contribution
to $t_p(F,G_n)$ from terms with some pair $(v_i,v_j)$ in the
small set $\bigcup_t U_t\cup W_t$ is small (of order $\eps^{1/4}$),
and it follows as before that if $\eta$ is small enough,
then $|t_p(F,G_n)-s(F,\ka)|$ is bounded
by some function of $F$ and $\eps$, giving the result.
\end{proof}

Let us note for later reference that, in one way, the assumptions
of Theorems~\ref{cthm} and~\ref{cthm2} are weaker than they may
first appear. Let $F$ be a loopless multigraph with vertex
set $u_1,u_2,\ldots,u_k$, and let $F_\ell\in \F_\ell$ be obtained
by subdividing each edge of $F$ exactly $\ell-1$ times.
Then \eqref{hft} may be rewritten as
\[
 \hom(F_\ell,G_n) = n^k\E\left( \prod_{u_iu_j\in E(F)} w_\ell(v_i,v_j)\right),
\]
where the expectation is over the uniform choice of $(v_1,v_2,\ldots,v_k)\in V(G_n)^k$.
Applying H\"older's inequality,
$\E(\prod_{i=1}^r X_i)\le \left(\prod \E(|X_i|^r)\right)^{1/r}$, with
$r=e(F)$, it follows that
\begin{multline}\label{Hol}
 \hom(F_\ell,G_n)^r \le n^{kr} \prod_{u_iu_j\in E(F)} \E(w_\ell(v_i,v_j)^r)
 = n^{kr} \E\bb{w_\ell(v_1,v_2)^r}^r \\
 = n^{kr-2r} \hom(H_{r,\ell},G_n)^r,
\end{multline}
where $H_{r,\ell}\in \F_\ell$
is the `theta graph' consisting of $r$ internally
vertex disjoint paths
of length $\ell$ joining the same pair of vertices.
The normalizing factors work out correctly, so we have
\begin{equation}\label{thol}
 t_p(F_\ell,G_n) \le t_p(H_{r,\ell},G_n).
\end{equation}
Hence, the condition that $t_p(F,G_n)$ remain bounded for
every $F\in F_\ell$ is equivalent to the condition that $t_p(F,G_n)$
is bounded for $F=H_{r,\ell}$, $r=1,2,\ldots$.

Arguing similarly, for any $F\in F_{\ge \ell}$ we may bound
$t_p(F,G_n)$ in terms of the quantities $t_p(H_{r,\ell'},G_n)$,
where $\ell'$ ranges over the lengths of the paths making up $F$.
Hence, to show that $t_p(F,G_n)$ is bounded for all $F\in F_{\ge\ell}$,
it suffices to prove the same condition for the graphs
$H_{r,\ell'}$, $r\ge 1$, $\ell'\ge \ell$.
Note that these latter conditions are simply moment conditions
on the numbers of walks of various lengths joining a random
pair of vertices of $G_n$.

\medskip
In the case where the limiting kernel $\ka$ is of finite type,
Theorem~\ref{cthm2} may be seen as a form of counting lemma.
In this case, it is easy to strengthen the result to count homomorphisms
from $F$ into $G_n$ with each vertex mapped to a specified part
of the partition of $G_n$ corresponding to the finite type kernel $\ka$,
obtaining a result similar in form to Lemma~\ref{pem}. 
Such a (strengthened) finite type case of Theorem~\ref{cthm}
or Theorem~\ref{cthm2} is very much easier to prove than the general case:
there is no need to apply Szemer\'edi's Lemma, and the proof of the
result of Chung and Graham~\cite{CG} for the uniform case goes through
without much modification. One might hope that, using Szemer\'edi's Lemma,
the full generality of Theorem~\ref{cthm2} would follow easily
from the finite type case, but this is not true. The problem
is that our assumptions are inescapably global: we assume, for instance,
that the number of copies of $C_{2\ell-2}$ in $G_n$
is bounded by a multiple of the expected number of copies. When we
take an $(\eps,p)$-regular partition, this gives no useful information
about the number of copies of $C_{2\ell-2}$ in each regular pair:
we have a bound that is of the form $Mk^{2\ell-2}$ times the expected
number of copies, where $k$ is the number of parts. To apply the finite
type case, we would need a bound independent of $k$.
For this reason there seems to be no easy way around the work
in the proof of Theorem~\ref{cthm}.

\medskip
Theorem~\ref{cthm2} may be seen as some progress towards a proof
of some form of Conjecture~\ref{cce}. More precisely, it is almost
an answer to Question~\ref{qq}: the only problem is that
for Theorem~\ref{cthm2} we work with $t_p$ rather than $s_p$.
We shall return to this in detail in a moment. However, even
ignoring this, Theorem~\ref{cthm2} is a little disappointing
in some ways. Let $\A=\T\cup\F_{\ge\ell}$.
Assuming boundedness of $t_p(F,G_n)$
for $F\in \A\cup\{C_{2\ell-2}\}$, we obtain convergence of the
counts $t_p(F,G_n)$ for $F\in \A$. The extra assumption for $F=C_{2\ell-2}$
is somehow annoying. This is perhaps clearest if we consider the range
where $p$ is fairly large, say $n^{-o(1)}$. In this case $s_p\sim t_p$,
and it makes sense to assume boundedness of all counts $s_p(F,G_n)$.
However, since $C_2$ does not make sense, the smallest value
of $\ell$ for which we can apply Theorem~\ref{cthm2} is $\ell=3$,
and we obtain convergence of the counts $s_p(F,G_n)$ for $F\in \F_{\ge 3}$.
In comparison, Theorem~\ref{g4} shows that with the counts $s_p$ bounded,
and $s_p(C_4,G_n)\to 1$, which should roughly correspond
to convergence to the uniform kernel $\ka=1$, we obtain
$s_p(F,G_n)\to s(F,\ka)=1$
for all $F\in \F_{\ge 2}$, rather than just for $F\in \F_{\ge 3}$.

In fact, Theorem~\ref{g4} gives much more: it gives convergence for all
$F$ with girth at least $4$. Chung and Graham~\cite{CG} asked
whether an analogous result holds for sparse graphs under
the appropriate assumptions (what they call `$\ell$-quasi randomness', which
corresponds roughly to the assumptions of Theorem~\ref{cthm}
with $\ka$ constant), with girth at least $4$ replaced by girth at least $2\ell$.
In our language, they asked whether (when $\ka=1$) the conclusion of Theorem~\ref{cthm}
can be extended to all $F$ with girth at least $2\ell$. Unfortunately,
the answer is no for a trivial reason,
namely that there are graphs $F$ with arbitrarily large
girth and arbitrarily large average degree. Taking $p=n^{-\alpha}$
for some $0<\alpha<1$, and $d$ large enough, for any graph $F$
with average degree $d$ the expected number of copies of $F$
in $G_n=G(n,p)$ is $o(1)$, so the normalizing constant in the definition
of $t_p(F,G_n)$ is $o(1)$. Since $\hom(F,G_n)$ is an integer,
we cannot have $t_p(F,G_n)\to 1$ in this case.

\subsection{Embeddings or homomorphisms?}\label{sec_hm}

In this subsection we
return to the use of $t_p$ rather than $s_p$ in Theorems~\ref{cthm}
and~\ref{cthm2}. Although this simplifies the proof, it is unsatisfactory
for a reason we shall now explain. We start by discussing the analogous
problem with the corresponding result of Chung and Graham~\cite{CG},
their Theorem 8. We shall use the following fact,
proved by Blakley and Roy~\cite{BlakleyRoy} in a slightly more general form
in the context of symmetric matrices.
\begin{theorem}\label{th_paths}
Let $G$ be a graph with $n$ vertices and average degree $d$. Then
$G$ contains at least $nd^\ell$ walks of length $\ell$.\noproof
\end{theorem}
Recall that we write $w_t(u,v)$ for the number of walks of length
$t$ from $u$ to $v$. Chung and Graham~\cite{CG} impose the condition
that $w_{\ell-1}(u,v)<c_0 p^{\ell-1}n^{\ell-2}$ holds for {\em
every} pair of vertices $u$, $v$, where $c_0$ is a constant: they
call this {\em condition} $U(\ell)$. In other words, the number of
walks from $u$ to $v$ is at most a constant times what it should be.
Normalizing so that $G_n$ contains exactly $pn^2/2$ edges, Chung and
Graham note that $U(\ell)$ can only hold if
$p=\Omega(n^{-1+1/(\ell-1)})$: otherwise, the expected number of
walks of length $\ell-1$ from a random $u$ to a random $v$ is much
less than 1, so $w_{\ell-1}(u,v)$ must sometimes be much larger than
its expectation.

In fact, $U(\ell)$ cannot hold unless $p$ is quite a bit larger,
but for the `wrong' reason: taking $\ell$ odd for simplicity, let $\ell=2k+1$.
Considering walks of length $\ell-1$ formed by tracing a walk of length
$k$ forwards and then backwards, we see that if $G_n$ has $pn^2/2$ edges, then
\begin{equation}\label{vv}
 \sum_v w_{\ell-1}(v,v) \ge \hom(P_k,G_n) \ge n(np)^k,
\end{equation}
where the second inequality is Theorem~\ref{th_paths}.
Thus there is some $v$ with $w_{\ell-1}(v,v)\ge (np)^k$,
and it follows that $U(\ell)$ can only hold if $p=\Omega(n^{-1+2/(\ell-1)})$,
so Theorem 8 of~\cite{CG} can only be applied for $p$ in this
range. Note that this is an essential problem: this result
counts homomorphisms (Chung and Graham use the notation
$\#\{H\subset G\}$ for $\hom(H,G)$), and the bound on $w_{\ell-1}(u,v)$
is definitely used with $u=v$. Indeed, as we shall see, the conclusion
fails if $p=o(n^{-1+2/(\ell-1)})$.

Turning to Theorem~\ref{cthm}, the condition that $t_p(C_{2\ell-2},G_n)$
remain bounded corresponds roughly to the condition $U(\ell)$: indeed,
the former says exactly that
\begin{equation}\label{tpcl}
 \sum_{u,v} w_{\ell-1}(u,v)^2 = O(n^{2\ell-2}p^{2\ell-2}),
\end{equation}
which follows immediately from $U(\ell)$. It turns out that the
problem described above does not arise with \eqref{tpcl} -- in this
second moment (rather than uniform) condition, the few pairs
with $u=v$ matter less. Indeed, it is easy to check that in $G(n,p)$,
for example, \eqref{tpcl} holds as long as $p=\Omega(n^{-1+1/(\ell-1)})$.
[The expected number of homomorphisms from $C_{2\ell-2}$ whose
image is a tree with $k$ edges is $O(n(np)^k)=O(n(np)^{\ell-1})$,
and the expected number whose image is a graph with $k$
vertices containing a cycle is $O(n^kp^k)=O((np)^{2\ell-2})$.]
However, the same problem arises in a different place.

As before,
let $H_{k,\ell}\in \F_\ell$ be the `theta graph' formed by $k$ paths
of length $\ell$ joining the same pair $(s,t)$ of vertices,
with the paths internally vertex disjoint. 
Suppose that $\ell$ is even. Writing $w_t(v)=w_t(v,V(G_n))$
for the number of walks of length $t$ in $G_n$ starting at $v$,
normalizing still so that $e(G_n)=pn^2/2$, 
and considering homomorphisms from $H_{k,\ell}$ to $G_n$ mapping $s$ and $t$
to a common vertex $v$, we have
\[
 \hom(H_{k,\ell},G_n) \ge \sum_v w_{\ell/2}(v)^k
 \ge n \left(\frac{1}{n} \sum_v w_{\ell/2}(v)\right)^k
 \ge n(np)^{k\ell/2},
\]
where the second inequality is from convexity and the last
from Theorem~\ref{th_paths}.
Since $|H_{k,\ell}| = 2+k(\ell-1)$ and $e(H_{k,\ell})=k\ell$,
it follows that $t_p(H_{k,\ell},G_n) \ge n^{k-1}(np)^{-k\ell/2}$.
Suppose that $p\le n^{-1+2/\ell-\eps}$ for some $\eps>0$.
Then taking $k$ large enough we see that $t_p(H_{k,\ell},G_n)\to\infty$,
so neither the assumptions nor the conclusion of Theorem~\ref{cthm}
can hold. When $\eps$ is small, this value of $p$
is much larger than that
above which the number of subgraphs of $G(n,p)$ isomorphic to $H_{k,\ell}$
is well behaved.

The calculations above illustrate the problem with working with $t_p$:
we count certain trees as copies of $H_{k,\ell}$, for example,
and the number of these trees exceeds the number of embeddings
of $H_{k,\ell}$ in a wide range of densities in which Theorem~\ref{cthm}
might otherwise apply. For this reason, if we could replace $t_p$
by $s_p$ throughout the statement of the theorem, we would obtain
a much stronger and more satisfactory result: not only would
it count embeddings, which is what we are really interested
in,  but it would apply to a much
larger family of graphs, for example, to random graphs with much
lower densities. Unfortunately, the proof breaks down in various
places if we simply replace $t_p$ by $s_p$. However, the next result 
is a major step in this direction.

Given vertices $v$, $w$ of a graph $G_n$, suppressing
the dependence on $G_n$, let us write
$p_\ell(v,w)$ for the number of {\em paths} of length $\ell$ from $v$ to $w$,
so $p_\ell(v,w)\le w_\ell(v,w)$.

\begin{theorem}\label{dthm}
Let $C>0$ and $\ell\ge 3$ be fixed, and let $p=p(n)$ be any
function of $n$. Let $(G_n)$ be a sequence of graphs satisfying
the following three conditions:
\begin{equation}\label{dc1}
 \sup_n s_p(F,G_n)<\infty \hbox{ for each } F\in \T,
\end{equation}
\begin{equation}\label{dc2}
\sum_u\sum_{v\ne u} p_{\ell-1}(u,v)^2 = O\bb{n^{2\ell-2}p^{2\ell-2}},
\end{equation}
and
\begin{equation}\label{dc3}
\sum_u\sum_{v\ne u} p_\ell(u,v)^k = O\bb{n^{2+k(\ell-1)}p^{k\ell}},
\end{equation}
for each fixed $k\ge 1$.
Suppose also that $\dc(G_n,\ka)\to 0$ for some kernel
$\ka:[0,1]^2\to [0,C]$. Then
$s_p(F,G_n)\to s(F,\ka)$ for each $F\in \F_\ell$.
\end{theorem}

Before turning to the proof of this result, let us make some remarks
on the conditions above. Firstly, in \eqref{dc1} it makes no
difference whether we write $s_p$ or $t_p$, by Lemma~\ref{stT}.

Condition \eqref{dc3} is almost the same as the condition
$s_p(H_{k,\ell},G_n)=O(1)$. Indeed, $\emb(H_{k,\ell},G_n)$ is simply
the sum over distinct $u$ and $v$ of the number of $k$-tuples
of {\em internally vertex disjoint} paths from $u$ to $v$,
so \eqref{dc3}, which bounds the same sum
without the restriction to disjoint paths,
is formally stronger than $s_p(H_{k,\ell},G_n)=O(1)$.
Since there are (typically) many paths from $u$ to $v$ 
in the range of $p$ for which \eqref{dc2} may hold,
it seems very likely that, assuming the other
conditions of Theorem~\ref{dthm}, $s_p(H_{k,\ell},G_n)=O(1)$ implies
\eqref{dc3}, so \eqref{dc3} could be replaced by this more pleasant
condition. However, we do not have a proof of this.

Similarly, condition \eqref{dc2} is closely related to
$s_p(C_{2\ell-2},G_n)=O(1)$, and could perhaps be replaced by this
weaker condition. This is less clear, however, as Theorem~\ref{dthm}
can be applied for $p$ small enough that the typical number of paths
of length $\ell-1$ between a given pair of vertices is $O(1)$.

Instead of \eqref{dc2} we can always impose the stronger
condition $t_p(C_{2\ell-2},G_n)=O(1)$; these conditions
are probably equivalent in the present setting.
The corresponding statement for \eqref{dc3} and 
the stronger assumption $t_p(H_{k,\ell},G_n)=O(1)$
is not true; see the discussion of the behaviour
of $t_p(H_{k,\ell},G_n)$ in the paragraphs preceding
Theorem~\ref{dthm}.

Finally, let us note that \eqref{dc3} gives us control over
$s_p(F_\ell,G_n)$ for all $F_\ell\in \F_\ell$, not just
for $F_\ell=H_{k,\ell}$. Let $F_\ell$ be obtained by subdividing a graph
$F$ with vertex set $u_1,u_2,\ldots,u_k$.
Then
\[
 \emb(F_\ell,G_n) \le \sum_{v_1,v_2,\ldots,v_k} \prod_{u_iu_j\in E(F)}
 p_\ell(v_i,v_j),
\]
where the sum is over all $n_{(k)}$ $k$-tuples of distinct vertices
of $G_n$.
Applying H\"older's inequality as in the proof \eqref{Hol}
of \eqref{thol}, but in a probability space with $n_{(k)}$
elements rather than $n^k$, we find that
\[
 \emb(F_\ell,G_n) \le n_{(k)} \E(p_\ell(v_1,v_2)^{e(F)}),
\]
where the expectation is over the choice of a random pair $(v_1,v_2)$
of {\em distinct} vertices of $G_n$. Condition
\eqref{dc3} bounds the final expectation; as usual the normalizing
factors work out, and we see that if \eqref{dc3} holds
for every $k$ then $s_p(F_\ell,G_n)=O(1)$ for every $F_\ell\in \F_\ell$.

\begin{proof}[Outline proof of Theorem~\ref{dthm}]
Since the proof is a relatively simple modification
of that of Theorem~\ref{cthm}, we shall give only an outline,
concentrating on the differences.

The first change we make is that we work with paths rather
than walks, replacing the quantities $w_t(u,v)$, $t=\ell-1,\ell$,
appearing in the proof of Theorem~\ref{cthm}
with the corresponding quantities $p_t(u,v)$. By Lemma~\ref{stT},
all but a vanishing fraction of the walks in $G_n$ of a given
length are paths, so \eqref{aim}, for example, implies
the same statement with $w_\ell(v,w)$ replaced by $p_\ell(v,w)$.
Of course, \eqref{aim} was proved using the assumption
$t_p(C_{2\ell-2},G_n)=O(1)$, whereas we now have the weaker
assumption \eqref{dc2}. However, following through the proof
it is easy to see that if we count paths instead of walks,
then \eqref{dc2} suffices. (The key point is that \eqref{dc2}
suffices to bound the number of {\em bad paths},
i.e., paths between endpoints $u$, $v$
with $p_{\ell-1}(u,v)> Mn^{\ell-2}p^{\ell-1}$.)

Let us fix (a small) $\eps>0$ and a graph $F_\ell\in \F_\ell$. We
also fix an integer $N$ to be chosen
later, depending only on $\eps$ and $F_\ell$. Finally, let $\eta$
be a small positive constant depending on $\eps$, $F_\ell$ and $N$.
For reasons that will become clear later, we first partition $V(G_n)$
into $N$ almost equal parts $Q_1,\ldots,Q_N$. Then we take
an $(\eta,p)$-regular partition $(P_i)$ with each $P_i$ contained
in some $Q_j$. For the moment we ignore the partition $(Q_i)$.

As before, passing to a subsequence we assume that the densities
$d_p(P_i,P_j)$ converge to a finite-type kernel $\ka$.
Let $S\subset V\times V$
be the set of pairs of vertices joined by the `wrong' number
of paths of length~$\ell$:
\[
 S =  \{(v,w): v\ne w,\ 
 | p_\ell(v,w) -\ka^\ell(v,w)n^{\ell-1}p^\ell| \ge \eps n^{\ell-1}p^\ell\}.
\]
If $\eta$ is chosen small enough then the proofs of \eqref{aim}
and \eqref{Waim}
carry though counting paths instead of walks,
and  (replacing $\eps$ by $\eps^2/10$), the equivalents
of \eqref{aim} and \eqref{Waim} imply that
\begin{equation}\label{ss}
 |S|\le \eps n(n-1).
\end{equation}
We proceed from here to our bound on $s_p(F_\ell,G_n)$ in two steps. First
we count something that is not quite an embedding of $F_\ell$.

Let $F_\ell$ be obtained from
the loopless multigraph $F$ 
by subdividing each edge $\ell-1$ times,
and let $u_1,\ldots,u_k$ be the vertices of $F$, which we also
regard as vertices of $F_\ell$.
By a {\em semiembedding} of $F_\ell$ into $G_n$ we mean a homomorphism
from $F_\ell$ into $G_n$ that maps the vertices $u_1,\ldots,u_k$
to distinct vertices of $G_n$, and each of the $e(F)$ $u_i$--$u_j$
paths of length $\ell$ that make up the graph $F_\ell$ into
a {\em path} in $G_n$. 
Clearly, every embedding is a semiembedding; the only additional condition
on an embedding is that the paths in $G_n$ are internally vertex
disjoint.

Let $\emb^+(F_\ell,G_n)\ge \emb(F_\ell,G_n)$ denote the number of semiembeddings
of $F_\ell$ into $G_n$. Then, from the definition of a semiembedding, we have
\begin{equation}\label{seform}
 \emb^+(F_\ell,G_n) = \sum_{v_1,\ldots,v_k} \prod_{u_iu_j\in E(F)} p_\ell(v_i,v_j),
\end{equation}
where the sum is over all $n_{(k)}$ sequences $(v_1,\ldots,v_k)$
of distinct vertices of $G_n$ and,
as usual, any multiple edges in $F$ give rise to multiple factors in the product.

As before we, we can rewrite the formula above as an expectation over
a random choice of $(v_1,\ldots,v_k)$. Normalizing correctly for a change,
let $X_{ij}$ be the random variable $p_\ell(v_i,v_j)/(n^{\ell-1}p^\ell)$, so
\[
 s_p^+(F_\ell,G_n) = \frac{\emb^+(F_\ell,G_n)}{n_{(|F_\ell|)}p^{e(F_\ell)}}
\sim \frac{\emb^+(F_\ell,G_n)}{n_{(|F|)}n^{|F_\ell|-|F|}p^{e(F_\ell)}}
  = \E \left(\prod_{u_iu_j\in E(F)} X_{ij}\right).
\]
Equation \eqref{ss} says, roughly speaking, that each $X_{ij}$
is with high probability close to `what it should be', which is a random variable
depending on $\ka$, the kernel corresponding
to the partition $(P_1,\ldots,P_k)$ of $G_n$.
We should like to deduce that the expectation
of the product is close to what it should be.

Let $Z$ be the set of $k$-tuples $(v_1,\ldots,v_k)$ with the
$v_i$ distinct such that $(v_i,v_j)\in S$ for some $1\le i<j\le k$.
Regarding $Z$ as an event in our probability space,
\[
 \PR(Z) \le \binom{k}{2} \PR((v_1,v_2)\in S) \le \eps\binom{k}{2},
\]
from \eqref{ss}. H\"older's inequality thus gives
\[
 \E \left(1_Z \prod_{u_iu_j\in E(F)} X_{ij}\right)
 \le \left(\E\left(1_Z^{e(F)+1}\right) \prod_{u_iu_j\in E(F)} \E\left(X_{ij}^{e(F)+1}\right)\right)^{1/(e(F)+1)},
\]
where $1_Z$ is the indicator function of the event $Z$.
Now, for each $i$ and $j$, we have
\[
 \E(X_{ij}^{e(F)+1}) 
 = \frac{1}{n(n-1)}\sum_u\sum_{v\ne u} \left(\frac{p_\ell(u,v)}{n^{\ell-1}p^\ell}\right)^{e(F)+1},
\]
which is $O(1)$ by our assumption \eqref{dc3}.
Also, $\E(1_Z^{e(F)+1})=\E(1_Z)=\PR(Z)\le\eps$.
Hence,
\begin{equation}\label{eZ}
 \E \left(1_Z \prod_{u_iu_j\in E(F)} X_{ij}\right)
 =O(\eps^{1/(e(F)+1)}).
\end{equation}
In other words, the contribution to \eqref{seform} from semiembeddings
mapping some edge of $F$ into a pair $(u,v)\in S$ is negligible.
By definition of $S$, the contribution from all other semiembeddings is 
`what it should be', and it follows that
\[
 |s_p^+(F_\ell,G_n)-s(F_\ell,\ka)| \le O(\eps^{1/(e(F)+1)}) +O(\eps).
\]
Since $\eps>0$ was arbitrary, we thus have $s_p^+(F_\ell,G_n)\sim s(F_\ell,\ka)$.

In the end, of course, it is $s_p(F_\ell,G_n)$ that we wish to bound,
not $s_p^+(F_\ell,G_n)$. Since $s_p(F_\ell,G_n)\le s_p^+(F_\ell,G_n)$
it remains to show that most semiembeddings are in fact embeddings,
i.e., that the paths in $G_n$ making up a typical semiembedding
are internally vertex disjoint.
For paths corresponding to vertex disjoint edges of $F$, this is 
quite easy, using the fact that $s_p(T,G_n)$ is bounded for each
tree, which tells us that almost all pairs of paths of length
$\ell$ are vertex disjoint. For paths corresponding to edges of $F$
sharing a vertex, there is a similar argument. We shall not spell
these arguments out as there is a third case that cannot be handled
in this way, namely paths corresponding to duplicate edges in $F$.
We must allow these, since we include, for example, $C_{2\ell}$
in $F_\ell$. It is in handling these paths that our `crude' partition
$(Q_i)$ comes in.

Let us classify paths $w_0w_1,\ldots,w_\ell$ in $G_n$ into $N^{\ell+1}$
{\em types}, according to which part $Q_i$ each $w_i$ lies in.
We say that a pair $(u,v)$ of distinct vertices of $G_n$ is {\em good} if,
for all $N^{\ell-1}$ possible types of $u$--$v$ path,
the number of $u$--$v$ paths of this type is `close' to what it should
be, i.e., within $\eps |Q_1|^{\ell-1}p^{\ell} \sim \eps n^{\ell-1}p^\ell/N^{\ell-1}$
of what it should be.
As usual, `what it should be' means the expected number in $G_p(n,\ka)$,
which depends not only on which parts $P_i$ the vertices $u$ and $v$ lie in,
but also on the type of path being considered.
Let $S'$ be the set of pairs $(u,v)$, $u\ne v$, that are {\em bad}, i.e.,
not good.

Since $N$ is fixed before $\eta$ is chosen, it is not hard to see that the argument
giving \eqref{ss} (applied with $\eps/N^{\ell-1}$ in place of $\eps$)
also shows that $|S'|\le \eps n(n-1)$; we omit the details. In other words,
almost all pairs of vertices are joined by about the right number of paths
of any given type. As before, we break down the set of embeddings
of $F_\ell$ into $G_n$ according to which vertices $v_1,\ldots,v_k$
of $G_n$ the `branch vertices' $u_1,\ldots,u_k$ are mapped to.
Defining $Z'$ analogously to $Z$, but using $S'$ instead of $S$,
the argument giving  \eqref{eZ} shows that we may assume
that $(v_1,\ldots,v_k)\notin Z'$, i.e., that no pair $(v_i,v_j)$
is in $S'$. Counting embeddings with $v_1,\ldots,v_k$ fixed,
it remains to choose $e(F)$ paths joining the appropriate
pairs $v_i$, $v_j$. Let us choose these paths one by one.
Since the total number of paths
joining $v_i$ to $v_j$ is about what it should be, all we must
show is that few (say at most $\eps n^{\ell-1}p^{\ell}$)
paths from $v_i$ to $v_j$ meet one of our at most $e(F)-1$ earlier paths.
But this is now easy: we must avoid a set $X$ of at most $(e(F)-1)(\ell-1)=O(1)$
vertices, the internal vertices of the previously chosen paths.
In fact, we shall do much more, avoiding any part $Q_a$ that meets $X$!
This rules out at most $(\ell-1)|X|N^{\ell-2}$ of the $N^{\ell-1}$ types
of $v_i$--$v_j$ paths. Choosing $N$ large enough (larger
than $1/\eps$), this is only a fraction $O(\eps)$ of all possible
types. Since $(v_i,v_j)\notin S'$, we have almost the right number
of paths of each remaining type, and hence almost the right number
of paths in total.
This completes our outline proof of Theorem~\ref{dthm}.
\end{proof}

Of course, there is a variant of Theorem~\ref{dthm} which is to
Theorem~\ref{dthm} as Theorem~\ref{cthm2} is to Theorem~\ref{cthm};
we shall not state this separately.

\medskip
Let us close this section by giving one simple example of a setting
in which the conditions of Theorem~\ref{dthm} are satisfied.
Fix $\ell\ge 3$, and
suppose that our sequence $(G_n)$ has the following two properties.
Firstly, the maximal degree $\Delta(G_n)$ is not too large:
\begin{equation}\label{deg}
 \Delta(G_n)\le Mpn,
\end{equation}
for some constant $M$.
Secondly,
\begin{equation}\label{UT}
 p_{\ell-1}(u,v)\le M n^{\ell-2}p^{\ell-1}
\end{equation}
for all $u\ne v\in V(G_n)$.
Condition \eqref{deg} is called DEG in Chung and Graham~\cite{CG};
condition \eqref{UT} is related to their condition $U(\ell)$,
but, as noted in the paragraph containing \eqref{vv},
is much weaker. In particular, it is easy to
check that if $p=n^{-\alpha}$ with $0<\alpha<1$ constant,
and $\ka$ is any bounded kernel, then the random graphs
$G_p(n,\ka)$ satisfy \eqref{deg} and \eqref{UT} with probability 1,
as long as $\alpha<1-1/(\ell-\nobreak1)$.
If \eqref{deg} and \eqref{UT} hold then $p_\ell(v,w)\le M^2 n^{\ell-1}p^\ell$
for all $v$ and $w$, while $s_p(T,G_n)\le M^{e(T)}$
for any tree $T$, so the conditions of Theorem~\ref{dthm} are satisfied.
Similarly, $p_t(v,w)\le M^{t-\ell+2} n^{t-1}p^t$ holds
for all $t\ge \ell$, so the variant of Theorem~\ref{dthm} corresponding
to Theorem~\ref{cthm2} applies.

It follows that conditions \eqref{deg} and \eqref{UT}
provide an answer to Question~\ref{qq}.
Indeed, Theorem~\ref{dthm} tells us that, under these
conditions, if $\ka$ is a bounded
kernel, then $\dc(G_n,\ka)\to 0$ implies $s_p(F,G_n)\to s(F,\ka)$
for all $F\in F_\ell$; its variant gives us $s_p(F,G_n)\to s(F,\ka)$
for all $F\in F_{\ge\ell}$. By Theorem~\ref{kdet}, the counts $s(F,\ka)$,
$F\in F_{\ge\ell}$, do determine the kernel (up to equivalence),
so conditions \eqref{deg} and \eqref{UT} are `suitable'
in the sense of Question~\ref{qq}.
As noted after Question~\ref{qq}, this implies the following result.

\begin{theorem}\label{th_sum}
Fix $\ell\ge 3$, let $p=p(n)$ be any function, and let $(G_n)$ be a
sequence of graphs satisfying \eqref{deg}, \eqref{UT} and the
bounded density assumption~\ref{AC}.
Then, for any bounded kernel $\ka$, we
have $\dc(G_n,\ka)\to 0$ if and only if $\de(G_n,\ka)\to 0$, where
$\de$ is defined using $\A=\T\cup \F_{\ge\ell}$ for the set of
admissible graphs.\noproof
\end{theorem}
 
In this section we discussed how to extend the 
subgraph (count) metric to sparse graphs, noting that
there are various possibilities (depending on the choice
of the set $\A$ of admissible graphs), and conjectured
that one particular extension is equivalent to the cut
metric. In the next section we turn to a different
metric, that extends much more easily to sparse graphs.

\section{The partition metric}\label{sec_partition}

As noted in Section~\ref{dense}, for dense graphs there are many
natural metrics that turn out to be equivalent, in the sense 
of generating the same topology. So far
we have focussed on the cut and subgraph (or count) metrics;
we now turn to the {\em partition metric}, introduced
by Borgs, Chayes, Lov\'asz, S\'os and Vesztergombi~\cite{BCLSV:2}.
In the dense case, it turns out to be relatively easy to show
that the partition and cut metrics are equivalent; in this brief section
we show that, under mild assumptions, this equivalence holds
also in the sparse setting, as long as $np\to\infty$.

On the one hand, this result (Theorem~\ref{dpdc}, below) shows that for
graphs with $\omega(n)$ edges, no new questions arise
by considering the partition metric. On the other hand, it reinforces
the conclusion that the cut metric remains extremely natural
for sparse graphs, and gives a way of considering the cut metric
from a very different point of view. There is another, very important,
motivation for introducing partition metrics for sparse
graphs: when we come to extremely spare graphs, with $\Theta(n)$ edges,
the cut metric turns out to make very little sense, while
the partition metric (which is no longer equivalent) remains
natural. This is a major topic in its own right and
will be discussed in a companion paper~\cite{BRsparse2}.

\subsection{Partition matrices and the partition metric}
Turning to the formal definitions, as in the rest of the
paper,
let $p=p(n)$ be a normalizing function and $G_n$ a
graph with $n$ vertices. Let $k\ge 2$ be fixed.
For $n\ge k$ and $\Pi=(P_1,\ldots,P_k)$ a partition of $V(G_n)$
into $k$ non-empty parts, let
$M_\Pi(G_n)=(d_p(P_i,P_j))_{1\le i,j\le k}$ be the matrix encoding
the normalized densities of edges between the parts of $\Pi$ (see \eqref{dpdef}).
Since $M_\Pi(G_n)$ is symmetric, we may think of this
matrix as an element of $\RR^{k(k+1)/2}$.
Set
\[
 \M_k(G_n) = \{ M_\Pi(G_n) \} \subset \RR^{k(k+1)/2},
\]
where $\Pi$ runs over all balanced partitions of $V(G_n)$ into
$k$ parts, i.e., all partitions $(P_1,\ldots,P_k)$ with $|P_i|-|P_j|\le 1$.

As usual, we assume that $G_n$ has $O(pn^2)$ edges. For definiteness,
let us assume that $e(G_n)\le Cpn^2/2$. Since each part of a balanced
partition has size at least $n/(2k)$, the entries of
any $M_\Pi(G_n)\in \M_k(G_n)$ are bounded by $C_k=(2k)^2C$, say.
Thus, $\M_k(G_n)$ is a subset of the compact space $\Mk=[0,C_k]^{k(k+1)/2}$.

Let $\C_0(\Mk)$ denote the set of non-empty compact subsets of $\Mk$,
and let $\dH$ be the Hausdorff metric on $\C_0(\Mk)$,
defined with respect to the $\ell_\infty$ distance, say.
Thus
\[
 \dH(X,Y)=\inf \{ \eps>0: X^{(\eps)}\supset Y,\, Y^{(\eps)}\supset X\},
\]
where $X^{(\eps)}$ denotes the $\eps$-neighbourhood of $X$ in the $\ell_\infty$ metric.
Since $(\Mk,\ell_\infty)$ is compact, by standard results (see, for example, Dugundji~\cite[p.\ 253]{Dugundji}),
the space $(\C_0(\Mk),\dH)$ is compact.
To ensure that the metric we are about to define is a genuine metric,
it is convenient to work with $\C(\Mk)=\C_0(\Mk)\cup\{\emptyset\}$,
setting $\dH(\emptyset,X)=C_k$, say, for any $X\in \C(\Mk)$, so
the empty set is an isolated point in $(\C(\Mk),\dH)$.

Let $\C=\prod_{k\ge 2} \C(\Mk)$, and let 
$\M:\F\mapsto \C$ be the map defined by
\[
 \M(G_n) = (\M_k(G_n))_{k=2}^\infty
\]
for every graph $G_n$ on $n$ vertices, noting that $\M_k(G_n)$ is empty if $k > n$.
Then we may define the {\em partition metric} $\dP$ by
\[
 \dP(G,G') = d(\M(G),\M(G')),
\]
where $d$ is any metric
on $\C$ giving rise to the product topology.
Considering the partition of an $n$ vertex graph into $n$ parts
shows that $\dP$ is a metric on the set $\F$ of isomorphism classes of finite graphs.
Recalling that each space $(\C(\Mk),\dH)$ is compact,
the key property of the partition metric is that
$(G_n)$ is Cauchy with respect to $\dP$ if and only if
there are compact sets $Y_k\subset \Mk$ such that
$\dH(\M_k(G_n),Y_k)\to 0$ for each $k$.
In particular, convergence in $\dP$ is equivalent to convergence
of the set of partition matrices for each fixed $k$. Thus we may
always think of $k$ as fixed and $n$ as much larger than $k$.

In the dense case, a metric equivalent to $\dP$ has been introduced
independently by Borgs, Chayes, Lov\'asz, S\'os and Vesztergombi~\cite{BCLSV:2};
the only difference is that in~\cite{BCLSV:2}, all partitions into
$k$ parts are considered, rather than just balanced partitions.
Of course, one then needs to take care to ensure that the densities
between small parts are counted with an appropriate weight when computing
the distance between
density matrices $\M_k$. Whether one takes all partitions or just
balanced partitions is a matter of taste: it is very easy to see
that convergence in either of the resulting metrics implies convergence in the other.

We may extend the map $\M:\F \to \C$, and hence $\dP$, to bounded
kernels in a natural way: instead
of partitioning the vertex set into $k$ almost equal parts, we partition
$[0,1]$ into $k$ exactly equal parts, and consider the closure of the set of `density
matrices' that may be obtained from $\ka$ using such partitions; we omit the details.
Note that, as shown by Borgs, Chayes, Lov\'asz, S\'os and Vesztergombi~\cite[Example 4.4]{BCLSV:2},
the set of density matrices is not in general closed.

As for the cut metric, it is easy to check that it makes little difference
whether we define $\dP$ for graphs directly, or by going via kernels.
(The corresponding dense result appears in~\cite{BCLSV:2}: the sparse
case here is slightly more complicated due to the possibility of
`high-degree' vertices.)

\begin{lemma}\label{dPsame}
Let $p=p(n)$ satisfy $p\ge 1/n$, and let $(G_n)$ be a sequence
of graphs with $e(G_n)=O(pn^2)$ and $\Delta(G_n)=o(pn^2)$.
Then $\dP(G_n,\ka_{G_n})\to 0$ as $n\to\infty$.
\end{lemma}
\begin{proof}
By definition, we must show that
$\dH\bb{ \M_k(G_n), \M_k(\ka_{G_n}) }\to 0$ for each $k\ge 1$. Fix $k$.
Since $e(G_n)=O(pn^2)$, there is a constant $D$ such
that at most $n/(2k)$ vertices of $G_n$ have degree more than $Dpn$.
Let $L$ denote the set of `low-degree' vertices, with degree
at most $Dpn$, so $|L|\ge n-n/(2k)$.

We must show that for any density matrix in $\M_k(G_n)$
there is a nearby matrix in $\M_k(\ka_{G_n})$, and vice versa.
The forward implication is trivial: a balanced partition $\Pi$
of $V(G_n)$ corresponds to a partition of $[0,1]$
into sets whose sizes differ by $O(1/n)=o(1)$. Adjusting these
parts slightly, making changes only in subintervals
of $[0,1]$ corresponding to low degree vertices,
the entries of the corresponding density matrix change by $o(1)$.

For the reverse implication, let $\Pi$ be a partition of $[0,1]$
into $k$ parts $P_1,\ldots,P_k$, and let $M\in\M_k(\ka_{G_n})$
be the corresponding density matrix, with entries $m_{ij}$.
For $v\in V(G_n)=[n]$ and $1\le i\le k$, let
$p_{v,i}$ be the fraction of the subinterval of $[0,1]$
corresponding to the vertex $v$ that lies in $P_i$,
noting that $\sum_i p_{v,i}=1$ for each $v$,
and $\sum_v p_{v,i}=n/k$ for each $i$.
Form a random partition $\Pi'=(P_1',\ldots,P_k')$ as follows:
put each vertex $v$ into a random part $P_{i_v}'$ with
$\PR(i_v=i)=p_{v,i}$, with the choices independent for different
vertices $v$.

It is immediate that $\E(|P_i'|)=n/k$ and $\Var(|P_i'|)\le n/k$.
It follows that for some constant $C$ we have
\begin{equation}\label{even}
  \forall i: \bigl| |P_i'| -n/k \bigr|\le C\sqrt{n}
\end{equation}
with probability at least $0.99$.
Writing $v\sim w$ if $vw\in E(G_n)$,
for $1\le i,j\le k$ we have
\begin{multline*}
 \E(e(P_i',P_j')) = \sum_{(v,w)\,:\,v\sim w} \E(1_{i_v=i}1_{i_w=j})
 = \sum_{(v,w)\,:\,v\sim w} p_{v,i}p_{w,j} \\
 = n^2p \int_{P_i\times P_j} \ka_{G_n}(x,y) \dd x\dd y,
\end{multline*}
so the expectation of $e(P_i',P_j')/(n^2p)$ is exactly $m_{ij}/k^2$.
For edges $vw$, $v'w'$ of $G_n$, the random variables
$1_{i_v=i}1_{i_w=j}$ and $1_{i_{v'}=i}1_{i_{w'}=j}$ are
independent unless $vw$ and $v'w'$ share a vertex,
in which case their covariance is at most one.
It follows that $\Var(e(P_i',P_j'))$
is bounded by $2\hom(P_2,G_n)$; the factor $2$ arises
since we may put the common vertex of two incident
edges into $P_i$ or $P_j$.
But $\hom(P_2,G_n)\le 2e(G_n)\Delta(G_n)$, which is $o(n^4p^2)$
by assumption. Hence, for any $\eps$, the probability that we have
\begin{equation}\label{dclose}
 \left| \frac{e(P_i',P_j')}{n^2p} -\frac{m_{ij}}{k^2} \right| \le \eps
\end{equation}
for every $i$ and $j$ with $1\le i,j\le k$ is at least $0.99$, provided
$n$ is large enough.

From the comments above, if $n$ is large enough,
there is a partition $\Pi'$ for which both \eqref{even} and \eqref{dclose}
hold. Starting from such a partition and moving at most $O(\sqrt{n})=o(n)$
vertices of $L$ (the set of low-degree vertices)
between parts, we may find a balanced partition
with almost the same density matrix. In other words,
we may find an element of $\M_k(G_n)$ close to $M$, completing the proof.
\end{proof}

If $np\to \infty$, then the condition of Lemma~\ref{dPsame}
that $\Delta(G_n)=o(n^2p)$ holds trivially, since $\Delta(G_n)\le n=o(n^2p)$.
When $np$ is bounded, this condition is necessary.
Taking $G_n$ to be a star, for example, every partition of $V(G_n)$
has the property that there is one part meeting all edges. But the corresponding
kernel has partitions which are very far from having this property,
namely those
in which, roughly speaking, the central vertex of the star has been split between parts.

\subsection{The relationship between the cut and partition metrics}\label{ss_pdense}

We now turn to the main result
of this section, showing the equivalence of $\dc$ and $\dP$
under mild assumptions. The key idea of the proof is that one can identify
the density matrix corresponding to a weakly $(\eps,p)$-regular partition from
the set of density matrices.

\begin{theorem}\label{dpdc}
Let $np\to\infty$, and let $(G_n)$ be a sequence of graphs with $|G_n|=n$
satisfying the bounded density assumption~\ref{AC}.
Let $\ka$ be a bounded kernel. Then $\dP(G_n,\ka)\to 0$ if and only
if $\dc(G_n,\ka)\to 0$.
\end{theorem}
\begin{proof}
Suppose first that $\dc(G_n,\ka)\to 0$, i.e., that $\dc(\ka_{G_n},\ka)\to 0$.
If $\ka_1$ and $\ka_2$ are any kernels with $\dc(\ka_1,\ka_2)<d$,
and $M\in\M_k(\ka_1)$, then there is an $M'\in\M_k(\ka_2)$ whose
entries differ from those of $M$ by at most $k^2d$: one simply takes
the corresponding partition for $\ka_2$, after rearranging so that
$\cn{\ka_1-\ka_2}<d$. It follows that $\dH(\M_k(\ka_1),\M_k(\ka_2))\le k^2\dc(\ka_1,\ka_2)$.
Hence, $\dH(\M_k(\ka_{G_n}),\M_k(\ka))\to 0$. Using Lemma~\ref{dPsame}, it
follows that $\dP(G_n,\ka)\to 0$.

Now suppose that $\dP(G_n,\ka)\to 0$.
By the {\em index}
$\ind(M)$ of a density matrix $M=(m_{ij})\in \M_k(\ka')$
we mean simply $k^{-2}\sum m_{ij}^2$.
Let $f(k,\eps)\ge k$ be a function to be specified later.
A $k$-by-$k$ density matrix $M\in \M_k(\ka')$ is {\em locally $\eps$-optimal} for a
kernel $\ka'$
if
\[
 \sup_{\ell\le f(k,\eps)}\  \sup_{M'\in M_\ell(\ka')} \ind(M') \le \ind(M)+\eps,
\]
i.e., if $M$ has almost maximal index among density matrices with not too many parts;
the definition of local optimality for $M\in \M_k(G_n)$ is similar.

Fix $\eps>0$. Since $(G_n)$ has bounded density, whenever $n$ is large enough
as a function of $k$, any density matrix in $\M_k(G_n)$ has index at most
some constant $C$.
It follows that there is a $K=K(C,\eps)$ such that, for $n$ large enough,
every $G_n$ has some locally optimal density matrix $M_k(n)$ of size at most $K$.
(This statement is a key part of the proof of Szemer\'edi's Lemma.)

Since $\dP(G_n,\ka)\to 0$, if $n$ is large enough,
there is an $M_k'(n)\in \M_k(\ka)$ with all entries within $\eps/(10C)$ of those
of $M_k(n)$. It follows that $\ind(M_k'(n))\ge \ind(M_k(n))-\eps/2$.
Similarly, for $n$ large, every $M'\in \cup_{\ell\le f(k,\eps)} \M_\ell(\ka)$
has all entries within $\eps/(10C)$ of some $M\in \cup_{\ell\le f(k,\eps)} \M_\ell(G_n)$,
which implies
\[
 \ind(M')\le \ind(M)+\eps/2 \le \ind(M_k(n))+3\eps/2 \le \ind(M_k'(n)) +2\eps,
\]
using the assumption that $M_k(n)$ is locally $\eps$-optimal for $G_n$ for the second
inequality. Thus
$M_k'(n)$ is locally $2\eps$-optimal
for $\ka$.

Recall that a partition $\Pi$ of $[0,1]$ is {\em weakly $(\eps,p)$-regular}
with respect to a kernel $\ka'$ if the corresponding averaged
kernel $\ka'/\Pi$ satisfies $\cn{\ka'/\Pi-\ka'}\le\eps$. 
The proof of Lemma~\ref{lSzw} (a sparse form of the Frieze-Kannan form
of Szemer\'edi's Lemma)
shows that if $(G_n)$ has bounded density, then there is a function
$f(k,\eps)$ such that, if $n\ge n_0(k,\eps)$ and $M\in \M_k(G_n)$
is locally $\eps$-optimal, then the corresponding partition of
$\ka_{G_n}$ is weakly $(\eps,p)$-regular; the same applies to $\ka$.
It follows that for $n$ large, identifying each density
matrix with a corresponding kernel, we have
$\dc(\ka_{G_n},M_k(n))$, $\dc(M_k(n),M_k'(n))$ and $\dc(M_k'(n),\ka)$
all of order $O(\eps)$. Since $\eps$ was arbitrary, it follows
that $\dc(\ka_{G_n},\ka)\to 0$, as required.
\end{proof}

In the light of Corollary~\ref{cSz}, Theorem~\ref{dpdc} implies that
a sequence $(G_n)$ satisfying Assumption~\ref{AC}
is Cauchy with respect to $\dP$ if and only if it is Cauchy with
respect to $\dc$.

The bounded density assumption in Theorem~\ref{dpdc}, which is
trivially satisfied in the dense case $p=\Theta(1)$, is
necessary in general. This can be seen by considering, for example,
a graph $G_n$ made up of $n/m$ complete graphs of order $m$, with
$m\sim pn=o(n)$ chosen so that $G_n$ has $pn^2/2$ edges.
By compactness, any sequence with $e(G_n)=O(pn^2)$ has a
subsequence that is Cauchy with respect to $\dP$ (here, in fact, the
original sequence is Cauchy). However, it is easy to check that no
subsequence of $(G_n)$ is Cauchy with respect to $\dc$.

The proof of Theorem~\ref{dpdc} applies just as well to kernels
as to graphs (and one can in any case approximate kernels by
dense graphs), showing that $\dP(\ka_n,\ka)\to 0$ if and only
if $\dc(\ka_n,\ka)\to 0$. It follows that $\dP$ induces
a metric on $\K$, the set of kernels quotiented by equivalence,
and that $\dP$ and $\dc$ give rise to the same topology
on $\K$. This was proved by
Borgs, Chayes, Lov\'asz, S\'os and Vesztergombi~\cite{BCLSV:2}
in their study of the dense case,
as part of their Theorem 3.5.

\section{Discussion and closing remarks}\label{sec_end}

For dense graphs, with $\Theta(n^2)$ edges, the results of
Borgs, Chayes, Lov\'asz, S\'os and
Vesztergombi~\cite{BCLSV:1,BCLSV:2} show that one single metric, say
$\dc$, effectively captures several natural notions of local and
global similarity. Indeed, convergence in $\dc$ is equivalent to
convergence in the partition metric 
$\dP$ (a natural global notion) and to convergence in
$\de$, i.e., convergence of all small subgraph counts,
a natural local notion.
These results apply to all sequences $(G_n)$ of graphs, but
if $G_n$ has $o(n^2)$ edges then they become trivial:
any such sequence is Cauchy with respect to any of the
metrics, and indeed converges to the zero kernel.
To make interesting statements about sparse graphs
one should adapt the metrics so that, roughly speaking,
given an `edge density
function' $p=p(n)$ satisfying $p\to 0$,
one compares a graph $G_n$ with $p\binom{n}{2}$ edges
to the Erd\H os--R\'enyi random graph $G(n,p)$
and its inhomogeneous variants rather than to $K_n$.
Our main aim in this paper has been to introduce
such metrics, and to discuss the relationships between them.
In this final section we turn to a slightly different question,
that of the relationship between metrics and random graph models.

\subsection{Models and metrics}

In the dense case, there is a very natural correspondence
between limit points of sequences converging in $\dc$, and 
the inhomogeneous random graph model $G(n,\ka)$.
In general, given any metric, we can ask whether there
is a corresponding random graph model:
for each metric $d$ on some class of (sparse) graphs satisfying
certain restrictions, we can ask the following question.

\begin{question}\label{qmodel}
Given a metric $d$, can we find a `natural' family of random graph models with the following
two properties:
(i) for each model, the sequence of random graphs $(G_n)$ generated
by the model is
Cauchy with respect to $d$ with probability $1$, and
(ii)
for any sequence $(G_n)$ with $|G_n|=n$ that is Cauchy with respect
to $d$, there is a model from the family
such that, if we interleave $(G_n)$ with a sequence of random
graphs from the model, the resulting sequence is still Cauchy with
probability $1$.
\end{question}

In the above question, we are implicitly assuming a coupling
between the probability spaces on which the graphs $(G_n)$
are defined. There is of course no need to do so: we can
replace `Cauchy with probability $1$' with the less familiar
`Cauchy in probability', which is equivalent to convergence
in probability in the completion; see Kallenberg~\cite[Lemma 4.6]{Kallenberg}.

Although Question~\ref{qmodel} is rather vague, for $d=\dc$
the answer is `yes' in the dense case, since $(G_n)$ is Cauchy
if and only if $\dc(G_n,\ka)\to 0$
for some kernel $\ka$, while the dense inhomogeneous random graphs
$G(n,\ka)$ converge to $\ka$ in $\dc$ with probability $1$.
Thus our family consists of one model $G(n,\ka)$ for each 
kernel $\ka$ (to be precise, for each equivalence
class of kernels under the relation $\sim$ defined in Subsection~\ref{ss_equiv}).
 
In the sparse case we do not have an entirely satisfactory answer
for any of the metrics considered in this paper.
Assuming that $np\to\infty$, there is an almost completely
satisfactory answer for $\dc$: if we impose the bounded density
assumption~\ref{AC}, then Corollary~\ref{cSz}
and Lemma~\ref{lrgcut} show that the sparse inhomogeneous models $G_p(n,\ka)$
answer Question~\ref{qmodel}.
For $\de$, defined with respect to certain restricted sets of subgraphs,
the results in Section~\ref{sec_compar} (in particular, Theorem~\ref{th_sum})
show that once again $G_p(n,\ka)$ answers this question for suitably restricted
sequences.

The extremely sparse case, where $p=\Theta(1/n)$, turns
out to be even more complicated; we shall discuss
this in a forthcoming paper~\cite{BRsparse2}.

There is an even vaguer, but perhaps more important, `mirror image'
of Question~\ref{qmodel}. Suppose that we have a random
graph model, and we would like to test whether it is appropriate
for some network in the real world. Then we would like to have a suitable
metric to compare a `typical' graph from the model with the real-world
network. It is too much to hope that one metric will be appropriate
in all situations; in particular, taking the simple case
in which our model is $G(n,p)$ for
some $p=p(n)\to 0$, the unnormalized metrics
$\dc$, $\de$ or $\dP$, that are very suitable for dense graphs,
will declare any graph with $o(n^2)$ edges to be close to the model.

In general, a random graph model (or family of models)
may suggest an appropriate metric,
or at least properties such a metric should have. For example,
the inhomogeneous models $G_p(n,\ka)$ and the results here
suggest the sparse version of $\dc$. Suppose, however,
that we are trying to model a network with rather few
edges but high `clustering', i.e., many triangles and other small subgraphs.
One possible model is a denser version of the
sparse random graphs with clustering introduced by
Bollob\'as, Janson and Riordan~\cite{BJRclust}: given, for each
fixed graph $F$, a `kernel' $\ka_F:[0,1]^{|F|}\to[0,\infty)$
and a normalizing function $p_F(n)$,
we choose vertex types $x_1,\ldots,x_n$ independently and uniformly at random and then,
for each $F$, 
add each possible copy of $F$ with vertex set $v_1,\ldots,v_k$,
$1\le v_1<v_2<\cdots<v_k\le n$,
with probability $\ka_F(x_{v_1},\ldots,x_{v_k})p_F(n)$.

In this model, a huge family of normalizations
are possible: we can take each $p_F$ to be any function of $n$
bounded by $1$. Of course, certain restrictions will be necessary for the
model to make much sense; otherwise, for example, the copies of some $F_1$
added directly may be swamped by copies of $F_1$ arising as subgraphs
of some $F_2$, in which case there was no point adding any copies of $F_1$
directly. However, there is no doubt that many different normalizations will be
interesting: for example, for any $0<a\le 4/3$,
we can produce graphs with, say, $\Theta(n^{4/3})$
edges and $\Theta(n^a)$ triangles. Indeed,
to do so we need only two kernels, one for edges (which we may take
to generate a bipartite graph if needed), and one for triangles.

If, for some reason, we are considering graphs with, say, around
$n^{4/3}$ edges and $n^{6/5}$ triangles, which is many more triangles than
expected in $G(n,n^{-2/3})$, then the triangles are an important part of the structure,
so in comparing two such graphs we should certainly compare the number of triangles,
normalized by dividing by $n^{6/5}$. This suggests a family of metrics generalizing
$\de$.

For each $F\in \F$ let $N_F=N_F(n)$ be a normalizing function satisfying $0<N_F\le
\infty$. (We allow infinity to include the possibility of totally ignoring
copies of some $F$. In fact, $N_F=n^{|F|+1}$ will do just as well.)
Then we may define a subgraph metric associated to ${\bf N}=(N_F)_{F\in \F}$
by modifying the definition of $\de$ given in Section~\ref{sec_hsp},
using the normalized count $\emb(F,G)/N_F(|G|)$ in place of $s_p(F,G)$.
This metric will only make sense for suitably restricted families of graphs, but
for such families, it will make much better sense than $\de$.

\subsection{Closing Remarks}\label{ssec_final}

The main aim of this paper is to draw attention to the possibility
that there is a rich theory of sparse (quasi-)random graphs waiting to
be explored. The beginnings of such a theory can be found in the
papers of Bollob\'as, Janson and Riordan~\cite{BJR,BJRclust}
in the very sparse case, and of Borgs, Chayes, Lov\'asz, S\'os,
Szegedy and Vesztergombi~\cite{BCLSV:homcount, BCLSSV:stoc, LSz1} in
the dense case; it would be desirable to build a theory
encompassing these two extreme threads. As we have just shown, this
task is unlikely to be easy: there are numerous unexpected
difficulties and pitfalls, and much work has to be done even to
arrive at concrete problems whose solutions would represent genuine
progress in this endeavour. In this paper we have attempted to do
some of this groundwork, and have identified some intriguing problems.

Our main focus has been the introduction of normalized versions
of the metrics $\dc$, $\de$ and $\dP$, adapted to the study
of graphs with $\Theta(pn^2)$ edges, where $p=p(n)\to 0$.
We have shown in Section~\ref{sec_partition} that 
(under a mild assumption) $\dc$ and $\dP$ have the same Cauchy sequences,
and in Section~\ref{sec_Sz} that (again under a mild assumption) these metrics 
have the property that any sequence $(G_n)$ contains a subsequence
converging to a kernel.

Turning to $\de$, things become more difficult.
We have conjectured that if our $p$-normalized subgraph counts are
suitably bounded and $p=p(n)$ is not too small then an appropriate
Cauchy sequence does converge to a kernel (see Conjectures~\ref{q1a}
and \ref{q1}). Tantalizingly, we cannot even prove this convergence
in just about the simplest case, when we know that the limit {\em
has to be} a constant kernel (Conjecture~\ref{q2}).

Section~\ref{sec_compar} is devoted to the relationship
between $\dc$ and $\de$. A sound understanding of the
relationship between these two metrics, the cut and count metrics,
would bring us much closer to a proper theory of sparse
inhomogeneous quasi-random graphs. We have conjectured that under
some natural and not too restrictive conditions, these two metrics
are equivalent in the sense that if $(G_n)$ is a sequence of graphs
that are not too `lumpy' then $(G_n)$ converges to a kernel $\kappa$ in
the $p$-cut metric if and only if it converges to $\kappa$ in the
$p$-count metric (see Conjecture~\ref{cce}). As one of our main
results, we have proved that $p$-cut convergence does imply
$p$-count convergence for a restricted set
of subgraph counts, under a mild assumption on the distribution
of paths of certain lengths (see Theorems~\ref{cthm2} and~\ref{dthm}).

The case of graphs of bounded average degree turns out to be even
more difficult, and will be discussed in a companion paper~\cite{BRsparse2}.

\begin{ack}
The authors would like to thank an anonymous referee
for many detailed suggestions improving the presentation
of the paper.
\end{ack}

\end{document}

%% file: pathcount.pstex_t
\begin{picture}(0,0)%
\includegraphics{pathcount.pstex}%
\end{picture}%
\setlength{\unitlength}{1658sp}%
\begingroup\makeatletter\ifx\SetFigFont\undefined%
\gdef\SetFigFont#1#2#3#4#5{%
  \reset@font\fontsize{#1}{#2pt}%
  \fontfamily{#3}\fontseries{#4}\fontshape{#5}%
  \selectfont}%
\fi\endgroup%
\begin{picture}(11216,6626)(1381,-8474)
\put(1951,-5161){\makebox(0,0)[rb]{\smash{{\SetFigFont{10}{12.0}{\familydefault}{\mddefault}{\updefault}{$u$}%
}}}}
\put(8101,-2161){\makebox(0,0)[lb]{\smash{{\SetFigFont{10}{12.0}{\familydefault}{\mddefault}{\updefault}{$P''$}%
}}}}
\put(8851,-5011){\makebox(0,0)[lb]{\smash{{\SetFigFont{10}{12.0}{\familydefault}{\mddefault}{\updefault}{$A_3$}%
}}}}
\put(8851,-6211){\makebox(0,0)[lb]{\smash{{\SetFigFont{10}{12.0}{\familydefault}{\mddefault}{\updefault}{$A_2$}%
}}}}
\put(8551,-7261){\makebox(0,0)[lb]{\smash{{\SetFigFont{10}{12.0}{\familydefault}{\mddefault}{\updefault}{$A_1$}%
}}}}
\put(8251,-8011){\makebox(0,0)[lb]{\smash{{\SetFigFont{10}{12.0}{\familydefault}{\mddefault}{\updefault}{$A_0$}%
}}}}
\put(11551,-4111){\makebox(0,0)[lb]{\smash{{\SetFigFont{10}{12.0}{\familydefault}{\mddefault}{\updefault}{$U_u$}%
}}}}
\end{picture}%

%% file: sparse.bbl
\begin{thebibliography}{99}

\bibitem{Alon} N.~Alon,
 Explicit {R}amsey graphs and orthonormal labelings,
 {\em Elec. J. Combin.} {\bf 1} (1994), R12 (electronic), 8 pp.

\bibitem{BlakleyRoy} G.R.~Blakley and P.~Roy,
 A H\"older type inequality for symmetric matrices with nonnegative entries,
 {\em Proc. Amer. Math. Soc.} {\bf 16} (1965), 1244--1245.

\bibitem{BBthr} B.~Bollob\'as,
 Threshold functions for small subgraphs,
 {\em Math. Proc. Cam. Phil. Soc} {\bf 90} (1981), 197--206.

\bibitem{BB:MGT} B.~Bollob\'as, {\em Modern Graph Theory}, Graduate Texts
  in Mathematics, vol.~184, Springer, New York (1998), xiv + 394~pp.

\bibitem{BBLinAnal} B.~Bollob\'as,
 {\em Linear Analysis}, 2nd ed., Cambridge University Press (1999), xii +240 pp.

\bibitem{BBRG} B.~Bollob\'as,
 {\em Random Graphs},  2nd ed.,
 Cambridge University Press (2001), 
xviii + 498~pp.

\bibitem{BBCR} B.~Bollob\'as, C.~Borgs, J.T.~Chayes and O.~Riordan,
 Percolation on dense graph sequences,
 preprint (2007) (revised 2008).
 \arxiv{0701346}

\bibitem{BJR} B.~Bollob\'as, S.~Janson and O.~Riordan,
 The phase transition in inhomogeneous random graphs,
 {\em Random Structures and Algorithms}, {\bf 31} (2007), 3--122.

\bibitem{BJRclust} B.~Bollob\'as, S.~Janson and O.~Riordan,
 Sparse random graphs with clustering, preprint (2008).
\arxiv{0807:2040}

\bibitem{BJRcm} B.~Bollob\'as, S.~Janson and O.~Riordan,
 The cut metric, random graphs, and branching process,
 preprint (2009). \arxiv{0901.2091}

\bibitem{BRsparse2} B.~Bollob\'as and O.~Riordan,
 Sparse graphs: metrics and random models, preprint (2008).
\arxiv{0812.2656}

\bibitem{BCL:unique} C.~Borgs, J.T.~Chayes and L.~Lov\'asz,
 Moments of two-variable functions and the uniqueness of graph limits,
 preprint (2008). \arxiv{0803.1244}

\bibitem{BCLSV:homcount} C.~Borgs, J.T.~Chayes, L.~Lov\'asz, V.T.~S\'os and K.~Vesztergombi,
 Counting graph homomorphisms,
 in {\em {T}opics in {D}iscrete {M}athematics}
 (eds. M. Klazar, J. Kratochvil, M. Loebl, J. Matousek, R. Thomas, P. Valtr), Springer (2006), pp
 315--371.

\bibitem{BCLSSV:stoc} C.~Borgs, J.T.~Chayes, L.~Lov\'asz, V.T.~S\'os, B.~Szegedy and K.~Vesztergombi,
 Graph limits and parameter testing,
 {\em Proc. 38th ACM Symp. Theory of Computing} (2006), 261--270.

\bibitem{BCLSV:1} C.~Borgs, J.T.~Chayes, L.~Lov\'asz, V.T.~S\'os and K.~Vesztergombi,
 Convergent sequences of dense graphs {I}: {S}ubgraph frequencies, metric properties
 and testing, preprint (2006) (revised 2007). \arxiv{0702004}


\bibitem{BCLSV:2} C.~Borgs, J.T.~Chayes, L.~Lov\'asz, V.T.~S\'os and K.~Vesztergombi,
 Convergent sequences of dense graphs {II}: {M}ultiway cuts and statistical physics,
preprint (2007). \webcite{http://www.cs.elte.hu/~lovasz/ConvRight.pdf}.

\bibitem{CG} F.~Chung and R.~Graham,
 Sparse quasi-random graphs, 
 {\em Combinatorica} {\bf 22} (2002), 217--244.

\bibitem{CGW89} F.R.K.~Chung, R.L.~Graham and R.M.~Wilson,
 Quasi-random graphs, {\em Combinatorica} {\bf 9} (1989),  345--362.

\bibitem{SJ209}
P. Diaconis and S. Janson,
Graph limits and exchangeable random graphs,
\emph{Rendiconti di Matematica}
\textbf{28} (2008), 33--61.

\bibitem{Dugundji} J.~Dugundji,
 {\em Topology},
 Allyn and Bacon, Inc., Boston, Mass. (1966), xvi+447 pp.

\bibitem{ER_RG1} P.~Erd\H os and A.~R\'enyi,
 On random graphs. {I},
 {\em Publ. Math. Debrecen} {\bf 6} (1959), 290--297.

\bibitem{ERpolarity} P.~Erd\H os and A.~R\'enyi,
 On a problem in the theory of graphs, 
 {\em Magyar Tud. Akad. Mat. Kutat\'o Int. K\"ozl.} {\bf 7} (1962), 623--641.

\bibitem{FKquick} A.~Frieze and R.~Kannan,
 Quick approximation to matrices and applications,
 {\em Combinatorica} {\bf 19} (1999), 175--220. 

\bibitem{GSsurvey} S.~Gerke and A.~Steger,
 The sparse regularity lemma and its applications,
 in {\em  Surveys in combinatorics 2005},
 London Math. Soc. Lecture Notes {\bf 327}, Cambridge
 University Press (2005),  pp. 227--258.

\bibitem{Gilbert_Gnp} E.N.~Gilbert,
  Random graphs, {\em Ann. Math. Statist.} {\bf 30} (1959), 1141--1144.


\bibitem{J_papr} S.~Janson,
 Poisson approximation for large deviations,
 {\em Random Structures and Algorithms} {\bf 1} (1990), 221--230.

\bibitem{SJ210}
 S. Janson,
 Standard representation of multivariate functions on a general
 probability space,
 preprint (2008).\arxiv{0801.0196}

\bibitem{JOR} S.~Janson, K.~Oleszkiewicz and A.~Ruci{\'n}ski,
 Upper tails for subgraph counts in random graphs,
 {\em Israel J. Math.} {\bf 142} (2004), 61--92. 

\bibitem{Kallenberg} O. Kallenberg,
 \emph{Foundations of Modern  Probability},
 2nd ed., Springer, New York, 2002.

\bibitem{Kim} J.H.~Kim,
 The Ramsey number $R(3,t)$ has order of magnitude $t\sp 2/\log t$,
 {\em Random Structures Algorithms} {\bf 7} (1995), 173--207. 

\bibitem{K97} Y.~Kohayakawa,
 Szemer\'edi's regularity lemma for sparse graphs,
 in {\em Foundations of computational mathematics (Rio de Janeiro, 1997)},
 Springer (1997), pp. 216--230.

\bibitem{KR-2003} Y.~Kohayakawa and V.~R{\"o}dl,
 Szemer\'edi's regularity lemma and quasi-randomness, 
 in {\em Recent advances in algorithms and combinatorics},
 CMS Books Math. {\bf 11}, Springer (2003), pp. 289--351


\bibitem{LSgenQT}
 L.~Lov\'asz and V.T.~S\'os,
 Generalized quasirandom graphs,
 {\em J. Combin. Theory B} {\bf 98} (2008), 146-163.

\bibitem{LSz1}
 L.~Lov\'asz and B.~Szegedy,
 Limits of dense graph sequences,
 {\em J. Combin. Theory B} {\bf 96} (2006), 933--957.

\bibitem{Ruc} A.~Ruci{\'n}ski,
 When are small subgraphs of a random graph normally distributed?
 {\em Probab. Th. and Related Fields} {\bf 78} (1988), 1--10.

\bibitem{Szem} E.~Szemer\'edi,
 Regular partitions of graphs, in {\em Probl\`emes combinatoires et th\'eorie des graphes}
(Colloq. Internat. CNRS, Univ. Orsay, Orsay, 1976),
{\em Colloq. Internat. CNRS} {\bf 260}, CNRS, Paris (1978),  pp. 399--401.

\bibitem{Tho} A. Thomason, Pseudo-random graphs, in
{\em Proceedings of Random Graphs} (M. Karonski, ed.),  Pozna\'n,
1985, {\em Annals of Discrete Mathematics}, {\bf 33} (1987)
307--331.

\bibitem{Tho2} A.~Thomason, 
 Random graphs, strongly regular graphs and pseudorandom graphs,
 in {\em Surveys in Combinatorics 1987}, 
  London Math. Soc. Lecture Note Ser. {\bf 123},
 Cambridge Univ. Press, Cambridge  (1987), pp 173--195.

\end{thebibliography}
